\documentclass[a4paper,12pt,english]{smfart}

\def\Ho{\operatorname{Ho}\nolimits}
\def\opp{{\operatorname{opp}\nolimits}}
\def\mperf{\operatorname{\!-perf}\nolimits}
\def\mMod{\operatorname{\!-mod}\nolimits}
\def\mperm{\operatorname{\!-perm}\nolimits}
\def\mproj{\operatorname{\!-proj}\nolimits}

\usepackage{amssymb}
\usepackage{url}
\RequirePackage{mathrsfs}
\DeclareSymbolFont{rsfscript}{OMS}{rsfs}{m}{b}
\DeclareSymbolFontAlphabet{\mathrsfs}{rsfscript}
\usepackage[utopia,expert]{mathdesign}
\usepackage{lscape}
\usepackage{calligra}
\usepackage[T1]{fontenc}

\usepackage{frcursive}

\usepackage{palatino}
\usepackage{rotating}
\usepackage{graphicx}

\usepackage{enumerate}

\usepackage{color}
\definecolor{shadecolor}{gray}{0.95}

\def\bfit{\bfseries\itshape}

\def\longmapright#1{\hspace{0.3em}\smash{
     \mathop{\longrightarrow}\limits^{#1}}\hspace{0.3em}}
\def\longisom{{\longmapright{\sim}}}

\def\proofname{D\'emonstration}

\input xypic
\xyoption{all}
\xyoption{arc}

\newtheorem{theo}{Theorem}[section]
\def\thetheo{\thesection.\arabic{theo}}
\newtheorem{prop}[theo]{Proposition}

\newtheorem{lem}[theo]{Lemma}

\newtheorem{coro}[theo]{Corollary}

\renewcommand\theequation{\thetheo}
\def\equat{\refstepcounter{theo}\begin{equation}}
\def\endequat{\end{equation}}

\renewcommand\thetable{\thetheo}

\renewcommand\thesection{\arabic{section}}
\renewcommand\thesubsection{\thesection.\Alph{subsection}}

\def\contentsname{Table of Contents}
\def\listfigurename{Liste des figures}
\def\listtablename{Liste des tables}
\def\appendixname{Appendice}

\def\refname{R\'ef\'erences}

    \def\FM{{\mathbb{F}}}

    \def\LM{{\mathbb{L}}}

    \def\QM{{\mathbb{Q}}}

    \def\ZM{{\mathbb{Z}}}

  \def\ab{{\mathbf a}}  \def\AC{{\mathcal{A}}}
\def\Bb{{\mathbf B}}  \def\bb{{\mathbf b}}  \def\BC{{\mathcal{B}}}
    \def\CC{{\mathcal{C}}}
    
    \def\EC{{\mathcal{E}}}
\def\Fb{{\mathbf F}}    
\def\Gb{{\mathbf G}}

  \def\kb{{\mathbf k}}  
\def\Lb{{\mathbf L}}    \def\LC{{\mathcal{L}}}
    
\def\Nb{{\mathbf N}}  \def\nb{{\mathbf n}}  
\def\Ob{{\mathbf O}}    \def\OC{{\mathcal{O}}}
\def\Pb{{\mathbf P}}    \def\PC{{\mathcal{P}}}
    
    \def\RC{{\mathcal{R}}}
\def\Sb{{\mathbf S}}    
\def\Tb{{\mathbf T}}    \def\TC{{\mathcal{T}}}
\def\Ub{{\mathbf U}}    
\def\Vb{{\mathbf V}}    
  \def\wb{{\mathbf w}}  
\def\Xb{{\mathbf X}}    \def\XC{{\mathcal{X}}}
\def\Yb{{\mathbf Y}}    \def\YC{{\mathcal{Y}}}
\def\Zb{{\mathbf Z}}

\def\Drm{{\mathrm{D}}}

\def\Grm{{\mathrm{G}}}    
\def\Hrm{{\mathrm{H}}}

\def\Mrm{{\mathrm{M}}}    
\def\Nrm{{\mathrm{N}}}

\def\Rrm{{\mathrm{R}}}

    \def\YCB{{\boldsymbol{\mathcal{Y}}}}
\def\Zrm{{\mathrm{Z}}}

  \def\sti{{\tilde{s}}}  
  \def\tti{{\tilde{t}}}

  \def\shat{{\hat{s}}}

          \def\sdo{{\dot{s}}}

          \def\xdo{{\dot{x}}}

\def\Gbt{{\tilde{\Gb}}}

\def\Lbt{{\tilde{\Lb}}}          
          
\def\Nbt{{\tilde{\Nb}}}          
          
\def\Pbt{{\tilde{\Pb}}}

\def\Tbt{{\tilde{\Tb}}}

\def\Ybt{{\tilde{\Yb}}}

\def\a{\alpha}
\def\b{\beta}

\def\G{\Gamma}
\def\d{\delta}
\def\D{\Delta}

\def\l{\lambda}
\def\L{\Lambda}

\def\o{\omega}
\def\O{\Omega}

\def\s{\sigma}

\def\th{\theta}
\def\Th{\Theta}

\def\z{\zeta}

\DeclareMathOperator{\ad}{{\mathrm{ad}}}

\DeclareMathOperator{\End}{{\mathrm{End}}}

\DeclareMathOperator{\Hom}{{\mathrm{Hom}}}

\DeclareMathOperator{\Id}{{\mathrm{Id}}}

\DeclareMathOperator{\Ind}{{\mathrm{Ind}}}

\DeclareMathOperator{\Irr}{{\mathrm{Irr}}}
\DeclareMathOperator{\Ker}{{\mathrm{Ker}}}

\DeclareMathOperator{\Res}{{\mathrm{Res}}}

\DeclareMathOperator{\Tr}{{\mathrm{Tr}}}

\DeclareMathOperator{\Comp}{{\mathrm{Comp}}}

\DeclareMathOperator{\rhom}{{\mathrm{RHom}}}
\DeclareMathOperator{\rgamma}{{\mathrm{R}\Gamma}}
\DeclareMathOperator{\rgammac}{{\mathrm{R}\Gamma_{\!\! c}}}

\DeclareMathOperator{\rgammacdim}{{\mathrm{R}\Gamma_{\!\! c}^{\mathrm{dim}}}}

\def\to{\rightarrow}
\def\longto{\longrightarrow}
\def\injto{\hookrightarrow}

\def\longtrait#1{\hspace{0.3em}{~ \SS{#1} ~ \over ~}\hspace{0.3em}}

\def\fonction#1#2#3#4#5{\begin{array}{rccc}
{#1} : & {#2} & \longto & {#3} \\
& {#4} & \longmapsto & {#5} 
\end{array}}

\def\fonctio#1#2#3#4{\begin{array}{ccc}
{#1} & \longto & {#2} \\
{#3} & \longmapsto & {#4} 
\end{array}}

\def\vide{\varnothing}
\def\DS{\displaystyle}
\def\SS{\scriptstyle}
\def\SSS{\scriptscriptstyle}

\def\lexp#1#2{\kern\scriptspace\vphantom{#2}^{#1}\kern-\scriptspace#2}
\def\le{\hspace{0.1em}\mathop{\leqslant}\nolimits\hspace{0.1em}}
\def\ge{\hspace{0.1em}\mathop{\geqslant}\nolimits\hspace{0.1em}}

\mathchardef\inferieur="321E
\mathchardef\superieur="321F

\def\eqna{\begin{eqnarray*}}
\def\endeqna{\end{eqnarray*}}

\def\itemth#1{\item[${\mathrm{(#1)}}$]}

\def\qlb{{\overline{\QM}_{\! \ell}}}

\catcode`\@=11
\long\def\@car#1#2\@nil{#1}
\long\def\@first#1#2{#1}
\long\def\@second#1#2{#2}
\long\def\ifempty#1{\expandafter\ifx\@car#1@\@nil @\@empty
  \expandafter\@first\else\expandafter\@second\fi}
\catcode`\@=12

\renewcommand{\tocsection}[3]{%
  {\bf \indentlabel{\ifempty{#2}{}{\ignorespaces#1 #2.\quad}}#3}}

\def\SL{{\Sb\Lb}}
\def\SU{{\Sb\Ub}}

\def\ve{{\SSS{\vee}}}

\def\surto{\twoheadrightarrow}

\theoremstyle{remark}
\newtheorem{rema}[theo]{Remark}
\newtheorem{exemple}[theo]{Example}
\newtheorem{exemples}[theo]{Examples}
\newtheorem{contre}[theo]{Counter-example}
\theoremstyle{plain}

\def\xyinj{\ar@{^{(}->}}
\def\xysur{\ar@{->>}}

\bigskip

\makeatletter
\def\hlinewd#1{%
\noalign{\ifnum0=`}\fi\hrule \@height #1 %
\futurelet\reserved@a\@xhline}
\makeatother

\newlength\epaisLigne

\usepackage{multirow}

\newcommand{\longsurto}{\relbar\joinrel\twoheadrightarrow}
\newcommand{\longinjto}{\lhook\joinrel\longrightarrow}

\def\perm{\operatorname{\!-perm}\nolimits}
\def\red{{\mathrm{red}}}
\def\modules{\operatorname{\!-mod}\nolimits}

\def\brauer{\operatorname{Br}}
\def\shift{{\mathrm{sh}}}

\begin{document}

\baselineskip=16pt

\title{Derived categories and Deligne-Lusztig varieties~II}

\author{{\sc C\'edric Bonnaf\'e}}
\address{
Institut de Math\'ematiques et de Mod\'elisation de Montpellier (CNRS: UMR 5149), 
Universit\'e Montpellier 2,
Case Courrier 051,
Place Eug\`ene Bataillon,
34095 MONTPELLIER Cedex,
FRANCE} 

\makeatletter
\email{cedric.bonnafe@univ-montp2.fr}
\makeatother

\author{{\sc Jean-Fran\c{c}ois Dat}}

\address{Institut de Math\'ematiques de Jussieu, 
4 place Jussieu,
75252 Paris cedex 05,
FRANCE}
\email{dat@math.jussieu.fr}

\makeatother

\author{{\sc Rapha\"el Rouquier}}

\address{UCLA Mathematics Department
Los Angeles, CA 90095-1555, 
USA}
\email{rouquier@math.ucla.edu}

\makeatother


\date{\today}

\thanks{The first author is partly supported by the ANR (Project No ANR-12-JS01-0003-01 ACORT). 
The second author thanks the Institut Universitaire de France
and the project ANR-14-CE25-0002-01  PerCoLaTor for their support.
The third author is partly supported by the NSF (grant DMS-1161999) and by a
grant from the Simons Foundation (\#376202, Rapha\"el Rouquier)}

\begin{abstract} 
This paper is a continuation and a completion of~\cite{BR}. We extend the Jordan decomposition
of blocks: we show that blocks of finite groups
of Lie type in non-describing characteristic are Morita equivalent to blocks of subgroups
associated to isolated elements of the dual group --- this is the modular version of a
fundamental result of Lusztig, and the best approximation of the character-theoretic Jordan
decomposition that can be obtained via Deligne-Lusztig varieties. The key new result is the
invariance of the part of the cohomology in a given modular series
of Deligne-Lusztig varieties associated to a given Levi subgroup,
under certain variations of parabolic subgroups.

We also bring in local block theory methods: we show that the equivalence arises from a 
splendid Rickard equivalence. Even in the setting of \cite{BR}, the finer homotopy equivalence
was unknown. As a consequence, the equivalences
preserve defect groups and categories of subpairs. We finally determine when
Deligne-Lusztig induced representations of tori generate the derived category of representations.
An additional new feature is an extension of the results to disconnected reductive
algebraic groups, which is required to handle local subgroups.
\end{abstract}

\maketitle

\pagestyle{myheadings}

\markboth{\sc C. Bonnaf\'e, J.-F. Dat \& R. Rouquier}{\sc Derived categories and Deligne-Lusztig varieties}

\setcounter{tocdepth}{3}
\tableofcontents

%

\section{Introduction}

Let $\Gb$ be a connected reductive algebraic group over an algebraic closure of a finite field,
endowed with an endomorphism $F$, a power
of which is a Frobenius endomorphism. Let $\ell$ be a prime number distinct from the defining
characteristic of $\Gb$ and $K$ a finite extension of
$\QM_\ell$, large enough for the finite groups considered. Let
$\OC$ be the ring of integers of $K$ over $\ZM_\ell$ and $k$ the residue field.
We will denote by $\Lambda$ a ring that is either $K$, $\OC$ or $k$.

\smallskip
The main tool for the study of representations of $\Gb^F$ over $\L$
is the Deligne-Lusztig induction.
Let $\Lb$ be an $F$-stable Levi subgroup of $\Gb$ contained in a parabolic subgroup $\Pb$
with unipotent radical $\Vb$ so that $\Pb=\Vb\rtimes\Lb$. Consider the Deligne-Lusztig
variety
$$\Yb_\Pb=\{g\Vb \in \Gb/\Vb~|~g^{-1}F(g) \in \Vb\cdot F(\Vb)\}.$$
It has a left action of $\Gb^F$ and a right action of $\Lb^F$ by multiplication.
The corresponding complex of $\ell$-adic cohomology induces a triangulated functor
$$\RC_{\Lb\subset \Pb}^{\Gb}:D^b(\L\Lb^F)\to D^b(\L\Gb^F),\
M\mapsto R\G_c(\Yb_\Pb,\L) \otimes_{\L\Lb^F}^\LM M$$
and a morphism
$$R_{\Lb\subset\Pb}^{\Gb}=[\RC_{\Lb\subset \Pb}^{\Gb}]:
G_0(\L\Lb^F)\to G_0(\L\Gb^F).$$
This is the usual Harish-Chandra construction when $\Pb$ is $F$-stable.

\subsection{Jordan decomposition}
Let $\Gb^*$ be a group Langlands dual to $\Gb$, with Frobenius $F^*$.
Consider the set $\Irr(\Gb^F)$ of characters of irreducible representations of $\Gb^F$ over $K$.
Deligne and Lusztig gave a decomposition of $\Irr(\Gb^F)$ into rational series
$$\mathrm{Irr}(\Gb^F)=\coprod_{(s)}\mathrm{Irr}(\Gb^F,(s))$$
where $(s)$ runs over the set of $\Gb^{*F^*}$-conjugacy classes of semi-simple elements of $\Gb^{*F^*}$.
The {\em unipotent characters} of $\Gb^F$ are those in $\mathrm{Irr}(\Gb^F,1)$.

Let $\Lb$ be an $F$-stable Levi subgroup of $\Gb$ with dual $\Lb^*\subset \Gb^*$ containing
$C_{\Gb^*}(s)$.
Lusztig constructed a bijection
$$\mathrm{Irr}(\Lb^F,(s))\xrightarrow{\sim}\mathrm{Irr}(\Gb^F,(s)),\ \psi\mapsto\pm
R_{\Lb}^{\Gb}(\psi).$$

If $s\in Z(\Lb^*)$, then there is a bijection
$$\mathrm{Irr}(\Lb^F,(1))\xrightarrow{\sim}\mathrm{Irr}(\Lb^F,(s)),\ \psi\mapsto
\eta\psi$$
where $\eta$ is the one-dimensional character of $\Lb^F$ corresponding to $s$, and we obtain
a bijection
$$\mathrm{Irr}(\Lb^F,(1))\xrightarrow{\sim}\mathrm{Irr}(\Gb^F,(s)).$$
This provides a description of irreducible characters of $\Gb^F$ in the rational series $(s)$
in terms of unipotent characters of an other group, when $C_{\Gb^*}(s)$ is a Levi subgroup
of $\Gb^*$.

\bigskip
Let us now consider the modular version of the theory described above.
Let $s$ be a semi-simple element of $\Gb^{*F^*}$ of order prime to $\ell$.
Consider $\coprod_{t}\Irr(\Gb^F,(t))$, where
$(t)$ runs over conjugacy classes of semi-simple elements of $\Gb^{*F^*}$ whose $\ell'$-part
is $(s)$. Brou\'e and Michel \cite{BrMi} have shown this is a union of blocks of $\OC\Gb^F$.
The sum of the corresponding block idempotents is
an idempotent $e_s^{\Gb^F}\in Z(\OC\Gb^F)$, and we obtain a decomposition
$$\OC\Gb^F\mMod=\bigoplus_{(s)} \OC \Gb^Fe_s^{\Gb^F}\mMod$$
where $(s)$ runs over $\Gb^{*F^*}$-conjugacy classes of semi-simple $\ell'$-elements of
$\Gb^{*F^*}$.

\smallskip
Let $\Lb$ be an $F$-stable Levi subgroup of $\Gb$ with dual $\Lb^*$ containing
$C_{\Gb^*}(s)$. Let
$\Pb$ be a parabolic subgroup of $\Gb$ with unipotent radical $\Vb$ and Levi complement
$\Lb$.
Brou\'e \cite{broue} conjectured that the $(\OC\Gb^F,\OC\Lb^F)$-bimodule 
$\Hrm^{\dim\Yb_\Pb}(\Yb_{\Pb},\OC)e_s^{\Lb^F}$ induces
a Morita equivalence between $\OC\Gb^Fe_s^{\Gb^F}$ and $\OC\Lb^Fe_s^{\Lb^F}$. This was proven
by Brou\'e \cite{broue} when $\Lb$ is a torus and in \cite{BR} in general. 

Brou\'e also
conjectured that the truncated complex of cohomology $\Grm\Gamma_c(\Yb_\Vb,\OC)e_s^{\Lb^F}$
(Rickard's refinement of $\Rrm\Gamma_c(\Yb_\Vb,\OC)e_s^{\Lb^F}$,
well defined in the homotopy category \cite{Ri}) induces a splendid Rickard equivalence
between $\OC\Gb^Fe_s^{\Gb^F}$ and $\OC\Lb^Fe_s^{\Lb^F}$: it induces not only an equivalence
of derived categories, but even an equivalence of homotopy categories, and it induces a similar
equivalence for centralizers of $\ell$-subgroups. One of our main results here is a proof of
that conjecture.
In order to show that there is a homotopy equivalence, for connected groups, we show that the
global functor induces local derived equivalences for centralizers of $\ell$-subgroups.
Since such centralizers need not be connected, we need to extend
the results of \cite{BR} to disconnected groups. So, part of
this work involves working with disconnected groups.

We also extend the ``Jordan decomposition equivalences'' (Morita and splendid Rickard) to the
``quasi-isolated case'': assume now only $C_{\Gb^*}^\circ(s)\subset\Lb^*$, and that 
$\Lb^*$ is minimal with respect to this property. We show that
the right action of $\Lb^F$ on $\Hrm^{\dim\Yb_\Pb}(\Yb_{\Pb},\OC)e_s^{\Lb^F}$ extends to
an action of $N=N_{\Gb^F}(\Lb,e_s^{\Lb^F})$ commuting with the action of $\Gb^F$, and the
resulting bimodule induces a Morita equivalence between 
$\OC\Gb^Fe_s^{\Gb^F}$ and $\OC Ne_s^{\Lb^F}$. Similarly, the complex $\Grm\Gamma_c(\Yb_\Vb,\OC)e_s^{\Gb^F}$
induces a splendid Rickard equivalence
 between $\OC\Gb^Fe_s^{\Gb^F}$ and $\OC Ne_s^{\Lb^F}$.

As a consequence, we deduce that the bijection between blocks of
$\OC\Gb^Fe_s^{\Gb^F}$ and $\OC Ne_s^{\Lb^F}$ preserves the local structure,
and in particular, preserves defect groups. Cabanes and Enguehard
have proven this under some assumptions on $\ell$ 
\cite[Proposition 5.1]{CaEn1}, and Kessar and Malle in the
setting of \cite{BR}, when
one of the blocks under consideration has abelian defect groups (modulo a central $\ell$-subgroup)
\cite[Theorem 1.3]{KeMa}, an important step in their proof
of half of Brauer's height zero conjecture for all finite groups~\cite[Theorem~1.1]{KeMa} and 
the second half for quasi-simple groups~\cite[Main~Theorem]{KeMa2}.

\medskip
Let us summarize this.

\begin{theo}
\label{th:introequiv}
Assume $C_{\Gb^*}^\circ(s)\subset\Lb^*$ and that 
$\Lb^*$ is minimal with respect to this property.

The right action of $\Lb^F$ on $\Grm\Gamma_c(\Yb_\Pb,\OC)e_s^{\Lb^F}$ extends
to an action of $N$ and the resulting complex $C$ induces a splendid
Rickard equivalence between $\OC\Gb^Fe_s^{\Gb^F}$ and $\OC Ne_s^{\Lb^F}$.
The bimodule $\Hrm^{\dim\Yb_\Pb}(C)$ induces a Morita equivalence
between $\OC\Gb^Fe_s^{\Gb^F}$ and $\OC Ne_s^{\Lb^F}$. 

The bijections between blocks of
$\OC\Gb^Fe_s^{\Gb^F}$ and $\OC Ne_s^{\Lb^F}$ induced by those equivalences
preserve the local structure.
\end{theo}

Significant progress has been made recently
on counting conjectures for finite groups, using the classification of
finite simple groups, and \cite{BR} has proved very useful. We hope
this theorem will lead to simplifications and new results.

The character-theoretic consequence of this theorem is
that, for groups with disconnected center, the Jordan decomposition shares many of the
properties of that for the connected case (commutation with Deligne-Lusztig induction for example).
In type $A$, the Jordan decomposition of characters links all series to unipotent series of
smaller groups: even in that case, the good behaviour of those correspondences was known only
when $q$ is large (Bonnaf\'e \cite{asterisque} for $\SL_n$ and Cabanes \cite{Ca} for
$\SU_n$).

\subsection{Generation of the derived category}

One of the two key steps in \cite{BR} was the proof that the category of
perfect complexes for $\OC\Gb^F$ is generated by the complexes
$\Rrm\Gamma_c(\Yb_\Bb)$,
where $\Bb$ runs over Borel subgroups of $\Gb$ with
an $F$-stable maximal torus. We show here a more precise
result of generation of the derived category of $\OC\Gb^F$.
Let $\EC$ be the set $\{\Rrm\Gamma_c(\Yb_\Bb)\otimes_{\OC\Tb^F}^\Lb M\}$, where
$\Tb$
runs over $F$-stable maximal tori of $\Gb$, $\Bb$ over
Borel subgroups of $\Gb$ containing $\Tb$, and $M$ over isomorphism
classes of $\OC\Tb^F$-modules.

\begin{theo}
The set $\EC$ generates $D^b(\OC\Gb^F)$ (as a thick subcategory) if and only
if all elementary abelian $\ell$-subgroups of $\Gb^F$ are contained in
tori.
\end{theo}

\smallskip
 This, in turn, requires an extension
of the results of Brou\'e-Michel \cite{BrMi}
on the compatibility between Deligne-Lusztig series of
characters and the Brauer morphism, to disconnected groups. We are able to achieve this by
refining our result on the generation of the category of perfect complexes to a
generation of the category of $\ell$-permutation modules whose vertices are contained
in tori (the crucial case is that of connected groups).
Such a result allows us to obtain a generating result for the full
derived category, under the assumption that all elementary abelian $\ell$-subgroups
are contained in tori.

Note that the condition
is automatically satisfied for $\Gb\Lb_n$ and $\Ub_n$ (see Examples \ref{ex:elementaryabelian}) and
when $\ell$ is {\it very good} for $\Gb$.

\subsection{Independence of the Deligne-Lusztig induction of the parabolic in a given series}

It is known in most cases, and conjectured in general, that the map
$R_{\Lb\subset\Pb}^\Gb$ on Grothendieck groups is actually independent of $\Pb$ 
(\cite{DL,L} when $\Lb$ is a torus and \cite{bm} when $q>2$ and $F$ is
a Frobenius endomorphism over $\Fb_q$).
On the other hand,
the functor $\RC_{\Lb\subset \Pb}^{\Gb}$ does depend on $\Pb$. Our main
new geometrical result proves the independence after truncating by a suitable
series.

\smallskip
Let $\Pb_1$ and $\Pb_2$ be two parabolic subgroups admitting a common
Levi complement $\Lb$. Denote by $\Vb_i^*$ the unipotent radical of the
parabolic subgroup of $\Gb^*$ corresponding to $\Pb_i$.

\begin{theo}
\label{thD}
Let $s$ be a semi-simple element of $\Lb^{*F^*}$ of order prime to $\ell$.
If
$$C_{\Vb_1^* \cap \lexp{F^*}{\Vb_1^*}}(s) \subset C_{\Vb_2^*}(s)
\quad\text{\it and}\quad
C_{\Vb_2^* \cap \lexp{F^*}{\Vb_2^*}}(s) \subset C_{\lexp{F^*}{\Vb_1^*}}(s)$$
then there is an isomorphism 
of functors between 
$$\RC_{\Lb \subset \Pb_1}^\Gb : \Drm^b(\L\Lb^F e_s^{\Lb^F}) \longto \Drm^b(\L\Gb^Fe_s^{\Gb^F})$$ 
$$\RC_{\Lb \subset \Pb_2}^\Gb[m] : \Drm^b(\L\Lb^F e_s^{\Lb^F}) \longto \Drm^b(\L\Gb^Fe_s^{\Gb^F}),
\leqno{\text{\it and}}$$
where $m=\dim(\Yb_{\Pb_2}^\Gb)-\dim(\Yb_{\Pb_1}^\Gb)$.
\end{theo}

\bigskip

For instance, if $C_{\Vb_1^*}(s)=C_{\Vb_2^*}(s)$, then 
the assumption of Theorem~\ref{thD} is satisfied.

\bigskip

This is the key result to prove Theorem \ref{th:introequiv}.
This result shows that when $C_{\Gb^*}^\circ(s)\subset\Lb^*$, the
$(\OC\Gb^F,\OC\Lb^F)$-bimodule $\Hrm^{\dim\Yb_\Pb}(\Yb_{\Pb},\OC)e_s^{\Lb^F}$ 
is independent of $\Pb$, a question left open in \cite{BR}.
We deduce that the bimodule is stable under the action
of $N=N_{\Gb^F}(\Lb^F,e_s^{\Lb^F})$. Using an embedding in a group with
connected center, we show that the obstruction for extending the action
of $\Lb^F$ to $N$ does vanish.

\bigskip

\noindent{\bf Remark.--- } Theorem 1.3 is used in \cite{Dat} to construct 
equivalences of categories between tamely ramified blocks of $p$-adic general linear groups.
Roughly speaking,  the main idea of \cite{Dat} is to ``glue'' the bimodules giving the Morita
equivalences of \cite{BR} along a suitable building. The gluing
process crucially uses the independence of the bimodules on the choice of
parabolic subgroups.

\subsection{Structure of the article}

We begin in \S\ref{se:generation} with the study of generation of
the category of perfect complexes, then we
move to complexes of $\ell$-permutation modules and finally we derive
our result on the derived category. A key tool, due to Rickard, is that
the Brauer functor applied to the complex of cohomology of a variety is
the complex of cohomology of the fixed point variety.

Section  \S\ref{se:rational} is devoted to the study of
rational series and their compatibility with local block theory.
Brou\'e and Michel proved a commutation formula between generalized
decomposition maps and Deligne-Lusztig induction. We need to extend
the compatibility between Brauer and Deligne-Lusztig theory to disconnected
groups, and check that the local blocks obtained from a series
satisfying $C^\circ_{\Gb^*}(s)\subset\Lb^*$ also satisfy a similar assumption
$C^\circ_{(C^\circ_\Gb(Q))^*}(s)\subset(\Lb\cap C^\circ_\Gb(Q))^*$.

From \S\ref{se:comparing} onwards, the group $\Gb$ is assumed to be connected.
Sections \S\ref{se:comparing} and \S\ref{sec:proof} are devoted to
the study of the dependence of the Deligne-Lusztig induction with respect to the
parabolic subgroup. The first section is devoted to the particular case
of varieties associated with Borel subgroups (and generalizations involving
sequences of elements). It is convenient there to work with a reference maximal torus.
This is the crucial case, from which the general
one is deduced in the latter section, where we go back to Levi subgroups that do not necessarily
contain that fixed maximal torus.

The final section \S\ref{se:jordan} is devoted to the Jordan decomposition.
We start by providing an extension of the action of $N$ on the cohomology
bimodule by proving that the cocycle obstruction would survive in a similar
setting for a group with connected center, where the action does exist.
The Rickard equivalence is obtained inductively, and that induction requires
working with disconnected groups.

In an appendix, we provide some results on the homotopy category of complexes
of $\ell$-permutation modules for a general finite group.

\medskip
We would like to thank the Referee for an extraordinarily thorough
list of suggestions, which greatly improved our paper.

\bigskip

\bigskip


\section{Notations}

\subsection{Modules}
Let $\ell$ be a prime number, $K$ a finite extension of
$\QM_\ell$ large enough for the finite groups considered,
 $\OC$ its ring of integers over $\ZM_\ell$ and $k$ its residue field.
We will denote by $\Lambda$ a ring that is either $K$, $\OC$ or $k$.

Given $\CC$ an additive category, we denote by
$\Comp^b(\CC)$ the category of bounded complexes of objects of $\CC$
and by  $\mathrm{Ho}^b(\CC)$ its homotopy category.

\smallskip
Let $A$ be a $\Lambda$-algebra, finitely generated
and projective as a $\Lambda$-module. We denote by $A^\opp$ the algebra opposite to $A$.
We denote by $A\mMod$ the category of finitely generated $A$-modules and by $A\mproj$ its
full subcategory of projective modules. We denote by $G_0(A)$ the
Grothendieck group of $A\mMod$.

 We put $\Comp^b(A)=\Comp^b(A\mMod)$,
$D^b(A)=D^b(A\mMod)$ and $\Ho^b(A)=\Ho^b(A\mMod)$. We denote by $A\mperf\subset D^b(A)$ the thick
full subcategory of perfect complexes (complexes quasi-isomorphic to objects of
$\Comp^b(A\mproj)$).

\smallskip
Let $C \in \Comp^b(A)$. There 
is a unique (up to a non-unique isomorphism) complex $C^\red$ 
which is isomorphic to $C$ in the homotopy category $\Ho^b(A)$ and which 
has no non-zero direct summand that is homotopy equivalent to $0$. 
Note that $C \simeq C^\red \oplus C'$ for some $C'$ homotopy equivalent to zero.

We denote by $\End^\bullet_A(C)$ the total $\Hom$-complex, with
degree $n$ term $\bigoplus_{j-i=n}\Hom_A(C^i,C^j)$.

\smallskip
Let $B$ be $\Lambda$-algebra, finitely generated
and projective as a $\Lambda$-module.
Let $C$ be a bounded complex of $(A\otimes_\Lambda B^\opp)$-modules,
finitely generated and projective as left $A$-modules and as right $B$-modules.
We say that $C$ induces a {\em Rickard equivalence} between $A$ and $B$ if
the canonical map $B\to \End_A^\bullet(C)$ is an isomorphism 
in
$\Ho(B\otimes_\Lambda B^\opp)$ and the canonical map 
$A\to \End_{B^\opp}^\bullet(C)^\opp$ is an isomorphism in
$\Ho(A\otimes_\L A^\opp)$.

\subsection{Finite groups}
\label{se:finitegroups}
Let $G$ be a finite group. We denote by $G^\opp$ the opposite group to $G$.
We put $\Delta G=\{(g,g^{-1}) | g\in G\}\subset G\times G^\opp$.
Given $g\in G$, we denote by $|g|$ the order of $g$.

Let $H$ be a subgroup of $G$ and $x\in G$.
We denote by $x_*$ the equivalence of categories 
$$x_* : \L (x^{-1}Hx)\modules \longisom \L H\modules$$
where $x_*(M)=M$ as a $\Lambda$-module and the action of $h\in H$ on $x_*(M)$ is given by
the action of $x^{-1}hx$ on $M$. We also denote by $x_*$ the corresponding isomorphism
of Grothendieck groups
$$x_* : G_0(\L (x^{-1}Hx)) \longisom G_0(\L H).$$

\medskip
 We assume $\Lambda=\OC$ or 
$\Lambda=k$ in the remainder of \S\ref{se:finitegroups}.

An {\it $\ell$-permutation $\Lambda G$-module} is defined
to be a direct summand of a finitely generated permutation module. We denote by
$\Lambda G\mperm$ the full subcategory of $\Lambda G\mMod$ with objects the
$\ell$-permutation $\Lambda G$-modules.

Let $Q$ be an $\ell$-subgroup $Q$ of $G$.
We consider the {\it Brauer functor}
$\brauer_Q : \L G\perm \to k[N_G(Q)/Q]\perm$. Given $M\in \L G\perm$,
we define $\brauer_Q(M)$ as the image of 
$M^Q$ in $(kM)_Q$, where $(kM)_Q$ is the largest quotient of $kM = k \otimes_\L M$ on which $Q$ 
acts trivially. 

We denote by $\mathrm{br}_Q:(\Lambda G)^Q\to kC_G(Q)$ the algebra morphism given by
$\mathrm{br}_Q(\sum_{g\in G}\lambda_g g)=
\sum_{g\in C_G(Q)}\lambda_g g$ where $\lambda_g\in\Lambda$ for $g\in G$.
Given $M\in\Lambda G\mperm$ and $e\in Z(\Lambda G)$ an idempotent, we have
$\brauer_Q(Me)=\brauer_Q(M)\mathrm{br}_Q(e)$.

\smallskip
Let $H$ be a subgroup of $G$, let $b$ be an idempotent of $Z(\Lambda G)$ and
$c$ an idempotent of $Z(\Lambda H)$. Let $C\in\Comp^b(\Lambda Gb\otimes (\Lambda Hc)^\opp)$.
We say that $C$ is {\em splendid} if the $(C^\red)^i$'s are $\ell$-permutation modules whose
indecomposable direct summands have a vertex contained in $\Delta H$.


\subsection{Varieties}
Let $p$ be a prime number different from $\ell$ and 
$\FM$ an algebraic closure of $\FM_p$. By variety, we mean a quasi-projective
algebraic variety over $\FM$.

Let $\Xb$ be a variety acted
on by a finite group $G$.
There is an object $\Grm\Gamma_c(\Xb,\Lambda)$ of $\Ho^b(\Lambda G\mperm)$, well defined up to a unique isomorphism.
It is a representative in the homotopy category of $\Lambda G$-modules
of the isomorphism class of the complex of \'etale $\Lambda$-cohomology with
compact support of $\Xb$ constructed as $\tau_{\le 2\dim \Xb}$ of the Godement resolution
(cf \cite[\S 2]{Rou1}, \cite[\S 1.2]{DuRou}, and \cite{Ri}).
We denote by $\Rrm\Gamma_c(\Xb,\Lambda)$ the image of $\Grm\Gamma_c(\Xb,\Lambda)$ in $D^b(\Lambda G)$.

\smallskip
Assume $\Lambda=\OC$ or $k$ and let $Q$ be an $\ell$-subgroup of $G$. The inclusion
$\Xb^Q\hookrightarrow \Xb$ induces an isomorphism \cite[Theorem 4.2]{Ri}
$$\Grm\Gamma_c(\Xb^Q,k) \xrightarrow{\sim} \brauer_Q(\Grm\Gamma_c(\Xb,\Lambda))
\text{ in }\Ho^b(kN_G(Q)\mperm).$$

\subsection{Reductive groups}

Let $\Gb$ be a (possibly disconnected) reductive algebraic group endowed
with an endomorphism $F$, a power $F^\delta$ of which is a Frobenius endomorphism defining a
rational structure over a finite field $\FM_q$ of characteristic $p$.
We refer to \cite{dm-nonc,dm-nonc-2} for basic results on disconnected groups.

\smallskip
Recall that a torus of $\Gb$ is torus of $\Gb^\circ$. Following the classical terminology
(cf for example \cite[\S 6.2]{Springer}), we define a
{\em Borel subgroup} of $\Gb$ to be a maximal connected solvable subgroup of $\Gb$.
We define a {\em parabolic subgroup} of $\Gb$ to be a
subgroup $\Pb$ of $\Gb$ such that $\Gb/\Pb$ is complete. We define the {\em unipotent radical} $\Vb$ of
a parabolic subgroup $\Pb$ to be its unique maximal connected unipotent normal subgroup. A {\em
Levi complement} to $\Vb$ in $\Pb$ is a subgroup $\Lb$ of $\Pb$ such that $\Pb=\Vb\rtimes\Lb$.

Note that a closed subgroup $\Pb$ of $\Gb$ is a parabolic subgroup of $\Gb$ if and only if
$\Pb^\circ$ is a parabolic subgroup of $\Gb^\circ$.
Let $\Pb$ be a parabolic subgroup of $\Gb$.
We have $\Pb^\circ=\Pb\cap\Gb^\circ$.
The unipotent radical $\Vb$ of $\Pb$ coincides with that of $\Pb^\circ$.
A Levi complement to $\Vb$ in $\Pb$ is a subgroup of the form
$N_{\Pb}(\Lb_\circ)$, where $\Lb_\circ$ is a Levi
complement of $\Vb$ in $\Pb_\circ$ (then $\Lb^\circ=\Lb_\circ$ and
$\Pb=\Vb\rtimes\Lb$). Note that our definition of parabolic subgroup is more general than that
of ``parabolic'' subgroup of \cite{dm-nonc}, which requires $\Pb=N_\Gb(\Pb^\circ)$.

\smallskip
We denote by $\nabla(\Gb,F)$ the set of pairs $(\Tb,\theta)$ where $\Tb$ is an $F$-stable maximal
torus of $\Gb$ and $\theta$ is an irreducible character of $\Tb^F$. Note that here $\Tb$ is a torus of $\Gb^\circ$.

Given an integer $d$, we denote by $\nabla_{d'}(\Gb,F)$ the set of pairs
$(\Tb,\theta)\in\nabla(\Gb,F)$ such that the order of $\theta$ is prime to $d$.
We put $\nabla_\Lambda(\Gb,F)=\nabla(\Gb,F)$ if $\Lambda=K$ and
$\nabla_\Lambda(\Gb,F)=\nabla_{\ell'}(\Gb,F)$ if $\Lambda=\OC$ or $k$ (recall that $k$ is a field
of characteristic $\ell$).

\subsection{Deligne-Lusztig varieties}

Given $\Pb$ a parabolic subgroup of $\Gb$ with unipotent radical $\Vb$ and
$F$-stable Levi complement $\Lb$, we define the Deligne-Lusztig variety
$$\Yb_\Vb=
\Yb_\Vb^\Gb=\Yb_\Pb=\Yb_\Pb^\Gb=\{g\Vb \in \Gb/\Vb~|~g^{-1}F(g) \in \Vb\cdot F(\Vb)\}.$$
This is a smooth variety, as in the case of connected reductive groups. It has
a left action by multiplication of $\Gb^F$ and a right action by
multiplication of $\Lb^F$ (note that the left and right actions of $Z(\Gb)^F$ coincide).
This provides a triangulated functor

\equat\label{eq:rlg-def}
\fonction{\RC_{\Lb\subset \Pb}^{\Gb}}{D^b(\L\Lb^F)}{D^b(\L\Gb^F)}{
M}{R\G_c(\Yb_\Vb,\L) \otimes_{\L\Lb^F}^\LM M}
\endequat
and a morphism
$$R_{\Lb\subset \Pb}^{\Gb}=[\RC_{\Lb\subset \Pb}^{\Gb}]:
G_0(\L\Lb^F)\to G_0(\L\Gb^F).$$

We put $\Xb_\Pb^\Gb=\{g\Pb \in \Gb/\Pb~|~g^{-1}F(g) \in \Pb\cdot F(\Pb)\}=\Yb_\Pb^\Gb/\Lb^F$.

\begin{rema}
\label{re:actionN}
Since $\Yb_\Pb$ depends only on $\Vb$, it is endowed with an action of $N_{\Gb^F}(\Pb^\circ,\Lb^\circ)$, which
is the group of rational points of the maximal Levi subgroup with connected component $\Lb^\circ$.
\end{rema}

\medskip

\bigskip

\def\brauer{{\mathrm{Br}}}

\section{Generation}
\label{se:generation}

\medskip

The aim of this section is to
extend ~\cite[Theorem~A]{BR} to the case of disconnected 
groups, and to deduce a generation theorem for the derived category.

\smallskip
In this section \S\ref{se:generation}, $\Gb$ is a (possibly disconnected) reductive
algebraic group.

\bigskip

\subsection{Centralizers of $\ell$-subgroups}

Let $\Pb$ be a parabolic subgroup of $\Gb$ admitting 
an $F$-stable Levi complement $\Lb$, and let $\Vb$ denote 
the unipotent radical of $\Pb$.
It is easily checked~\cite[Proof~of~Proposition~2.3]{dm-nonc} that 
\equat\label{eq:gy}
\Yb_\Vb^\Gb = \coprod_{g \in \Gb^F/\Gb^{\circ F}} g \Yb_\Vb^{\Gb^\circ}=\Gb^F\times_{\Gb^{\circ F}}\Yb_\Vb^{\Gb^\circ}.
\endequat
It follows immediately from~(\ref{eq:gy}) that 
\equat\label{eq:ind-rlg}
\RC_{\Lb\subset \Pb}^{\Gb} \circ \Ind_{\Lb^{\circ F}}^{\Lb^F} 
\simeq \RC_{\Lb^\circ \subset \Pb^\circ}^{\Gb} \simeq \Ind_{\Gb^{\circ F}}^{\Gb^F} \circ 
\RC_{\Lb^\circ \subset \Pb^\circ}^{\Gb^\circ}.
\endequat
If $\Gb=\Pb\cdot\Gb^\circ$, then the isomorphism $\Gb^\circ/\Vb\times_{\Lb^\circ}\Lb\simeq \Gb/\Vb$ induces an
isomorphism
\equat\label{eq:res-rlg}
\RC_{\Lb^\circ \subset \Pb^\circ}^{\Gb^\circ} \circ 
\Res_{\Lb^{\circ F}}^{\Lb^F} 
\simeq \Res_{\Gb^{\circ F}}^{\Gb^F} \circ \RC_{\Lb \subset \Pb}^{\Gb}.
\endequat

\bigskip

\begin{prop}\label{prop:centralisateur}
Let $Q$ be a finite solvable $p'$-group of automorphisms of $\Gb$ that commute with $F$ and normalize $(\Pb,\Lb)$.
\begin{itemize}
\itemth{a} The group $\Gb^Q$ is reductive.

\itemth{b} $\Pb^Q$ is a parabolic subgroup of $\Gb^Q$ 
whose unipotent radical is $\Vb^Q$ and 
admitting $\Lb^Q$ as an $F$-stable Levi complement. 
In particular, $\Vb^Q$ is connected.

\itemth{c} The natural map $\Gb^Q/\Vb^Q \to (\Gb/\Vb)^Q$ 
is an isomorphism of $(\Gb^Q,N_\Gb(\Pb^\circ,\Lb^\circ)^Q)$-varieties. 

\itemth{d} $(\Vb \cdot \lexp{F}{\Vb})^Q=\Vb^Q\cdot \lexp{F}{(\Vb^Q)}$.

\itemth{e} The natural map $\Yb_{\Vb^Q}^{\Gb^Q} \to (\Yb_\Vb^{\Gb})^Q$ 
is an isomorphism of $((\Gb^Q)^F,(N_{\Gb}(\Pb^\circ,\Lb^\circ)^Q)^F)$-varieties. If $Q$ is an $\ell$-group, it gives rise to an isomorphism
$\brauer_Q(\Grm\Gamma_c(\Yb_\Vb^\Gb,k))\xrightarrow{\sim}\Grm\Gamma_c(\Yb_{\Vb^Q}^{\Gb^Q},k)$ in
$\Ho^b(k((\Gb^Q)^F\times (N_{\Gb}(\Pb^\circ,\Lb^\circ)^Q)^{F\opp})$.
\end{itemize}
\end{prop}

\bigskip

\begin{proof}
Assume first $Q$ is cyclic, generated by an element $l$.

\medskip

(a) and (b) follow from~\cite[Proposition~1.3, Theorem~1.8, Proposition~1.11]{dm-nonc}.

\medskip

(c) Note that that both varieties are smooth (for $(\Gb/\Vb)^Q$, this follows from the fact
that $Q$ is a $p'$-group and $\Gb/\Vb$ is smooth).
The injectivity of the map is clear.

Let us prove the surjectivity. Let $g\Vb \in (\Gb/\Vb)^Q$. Then, 
$g^{-1}l(g)\in \Vb$. Denote by $\ad(g)$ the automorphism $x\mapsto gxg^{-1}$ of $\Gb$. Since
$\ad(g)^{-1}l\ad(g)$ is semisimple, it stabilizes a maximal torus 
of $\Pb^\circ$ (see~\cite[Theorem~7.5]{steinberg}), hence it stabilizes the
unique Levi complement $\Lb^\prime$ of $\Pb^\circ$ containing this maximal torus.
Since all Levi complements are conjugate under the action of $\Vb$, there exists 
$v \in \Vb$ such that $v^{-1}\Lb^\prime v=\Lb^\circ$. It follows that $(gv)^{-1}l(gv)\in\Vb$ and
$(gv)^{-1}l(gv)$ normalizes $\Lb^\circ$, hence $(gv)^{-1}l(gv)=1$, so $gv\in \Gb^Q$,
as desired.

The tangent space at $\Vb$ of $(\Gb/\Vb)^Q$ is the $Q$-invariant part of
the tangent space of $\Gb/\Vb$ at $\Vb$. That last tangent space is a quotient of
the tangent space of $\Gb$ at the origin. It follows that the canonical map
$\Gb^Q\to (\Gb/\Vb)^Q$ induces a surjective map between tangent spaces
at the origin. Consequently, the canonical map $\Gb^Q/\Vb^Q \to (\Gb/\Vb)^Q$
induces a surjective map between tangent spaces at the origin. We deduce that the map is
an isomorphism.

\medskip

(d) The number of $F$-stable maximal tori of $\Lb$ is a power of $p$ 
(see~\cite[Corollary~14.16]{steinberg}). Since $Q$ is a $p'$-group,
it normalizes 
some $F$-stable maximal torus. Using now the root system with respect to 
this maximal torus, we deduce that there exists a $Q$-stable subgroup $\Vb'$ of $\Vb$ such that 
$\Vb=\Vb' \cdot (\Vb \cap F(\Vb))$ and $\Vb' \cap F(\Vb)=1$. 
Therefore, $\Vb \cdot F(\Vb)=\Vb' \cdot F(\Vb)$ and the result 
follows.

\medskip

(e) follows immediately from (c) and (d).

\medskip
We prove now the proposition by induction on $|Q|$.
Let $Q_1$ be a normal subgroup of $Q$ of index a prime number
and let $l\in Q$, $l{\not\in}Q_1$. Let $Q_2$ be the subgroup of $Q$ generated by $l$.
By induction, the proposition holds for $Q$ replaced by $Q_1$: we have
a reductive group $\Gb_1=\Gb^{Q_1}$ and a parabolic subgroup $\Pb_1=\Pb^{Q_1}$ with unipotent
radical $\Vb_1=\Vb^{Q_1}$ and an $F$-stable Levi complement $\Lb_1=\Lb^{Q_1}$. These are all stable under $Q_2$.
The cyclic case of the proposition applied to the action of $Q_2$ on $(\Gb_1,\Pb_1,\Vb_1,\Lb_1)$ establishes the
proposition for the action of $Q$ on $(\Gb,\Pb,\Vb,\Lb)$.
\end{proof}

\bigskip

\begin{rema}
If $Q$ is a finite solvable $p'$-subgroup of $\Gb^F$, then $N_{\Gb}(Q)$ is reductive. If in addition $Q$ 
normalizes $(\Pb,\Lb)$, then $N_{\Pb}(Q)$ is a parabolic subgroup of $N_{\Gb}(Q)$ with unipotent radical
$\Vb^Q$ and Levi complement $N_{\Lb}(Q)$. The maps defined in (e) of Proposition \ref{prop:centralisateur}
are equivariant for the diagonal action of $N_{\Gb}(\Pb^\circ,\Lb^\circ,Q)^F$.
\end{rema}

We will need a converse to Proposition \ref{prop:centralisateur}, in the case of tori.

\begin{lem}\label{lem:bijtori}
Let $Q$ be a finite solvable $p'$-group of automorphisms of $\Gb$ that commute with $F$. We assume $Q$
stabilizes a maximal torus of $\Gb$ and a Borel subgroup containing that maximal torus.

Let $\Tb_Q$ be an $F$-stable maximal torus of $\Gb^Q$ contained in a Borel subgroup $\Bb_Q$ of
$\Gb^Q$. Then, $C_{\Gb^\circ}(\Tb_Q)$ is an $F$-stable maximal torus of $\Gb$ that is contained in
a $Q$-stable Borel subgroup $\Bb$ of $\Gb$ such that $(\Bb^Q)^\circ=\Bb_Q$.
\end{lem}

\begin{proof}
Note that the lemma holds for $Q$ cyclic by \cite[Theorem 1.8]{dm-nonc}.
We proceed by induction on $|Q|$ as in the proof of Proposition \ref{prop:centralisateur}, and we keep the notations
of that proof.
We know (Lemma for $Q_2$) that $\Tb_{Q_1}=C_{\Gb_1^\circ}(\Tb_Q)$ is an $F$-stable maximal torus of $\Gb_1^\circ$.
By induction,
$C_{\Gb^\circ}(\Tb_Q)=N_{\Gb^\circ}(\Tb_{Q_1})^\circ=C_{\Gb^\circ}(\Tb_{Q_1})$ is an $F$-stable maximal
torus of $\Gb$.

The existence of the Borel subgroup can be obtained as in \cite[p.350]{dm-nonc}.
Let $\Tb'$ be a maximal torus of $\Gb$ stable under $Q$ and $\Bb'$ be a Borel subgroup of $\Gb$ containing $\Tb'$ and
stable under $Q$. By Proposition \ref{prop:centralisateur}, $(\Tb^{\prime Q})^\circ$ is a maximal torus of $\Gb^Q$ and 
$(\Bb^{\prime Q})^\circ$ is a Borel subgroup containing it. So, there is $x\in (\Gb^Q)^\circ$ such that
$\Tb_Q={^x(\Tb^{\prime Q})^\circ}$ and $\Bb_Q={^x(\Bb^{\prime Q})^\circ}$. Let $\Bb={^x\Bb'}$. This is
a $Q$-stable Borel subgroup of $\Gb$ containing $C_{\Gb^\circ}(\Tb_Q)$. By Proposition \ref{prop:centralisateur},
$(\Bb^Q)^\circ$ is a Borel subgroup of $\Gb^Q$, hence $(\Bb^Q)^\circ=\Bb_Q$.
\end{proof}

\bigskip

To complete Proposition \ref{prop:centralisateur}, note the following result.

\bigskip

\begin{lem}\label{lem:delta}
Let $P$ be an $\ell$-subgroup of $\Gb^F \times N_{\Gb^F}(\Pb,\Lb)^\opp$ such that 
$(\Yb_\Vb^\Gb)^P \neq \vide$. Then $P$ is $(\Gb^F\times 1)$-conjugate
to a subgroup of $\Delta N_{\Gb^F}(\Pb,\Lb)$.
\end{lem}

\bigskip

\begin{proof}
Replacing $\Lb$ by $N_\Gb(\Pb,\Lb)$, we can assume that 
$N_\Gb(\Pb,\Lb)=\Lb$.

Let $Q \subset \Lb^F$ (respectively $R \subset \Gb^F$) denote the 
image of $P$ through the second (respectively first) projection and let $y\Vb \in (\Yb_\Vb^\Gb)^P$. 

If $g \in R$, then there exists $l \in Q$ such that $(g,l) \in P$. 
Therefore, $gyl\Vb=y\Vb$, hence $y^{-1}gy \Vb=l^{-1}\Vb$. This implies 
that $y^{-1}Ry \subset Q \Vb$. We denote by $\eta : R \to Q$ the composition 
$R \xrightarrow{\sim} y^{-1}Ry \injto Q\Vb \surto Q$. Since $R$ (respectively $Q$) acts 
freely on $\Gb/\Vb$ as they are $\ell$-groups, the previous computation shows that $\eta$ is an isomorphism, 
and that 
$$P=\{(g,\eta(g))~|~g \in R\}.$$

Now, there exists a positive integer $m$ such that $F^m(\Pb)=\Pb$ 
and $y^{-1}Ry \subset \Pb^{F^m}$. 
So $y^{-1}Ry$ acts by left translation on $\Pb^{F^m}/\Lb^{F^m}$. 
Since $y^{-1}Ry$ is a finite $\ell$-group and $|\Pb^{F^m}/\Lb^{F^m}|=|\Vb^{F^m}|$ 
is a power of $p$, it follows that $y^{-1}Ry$ has a fixed point in 
$\Pb^{F^m}/\Lb^{F^m}$. 
Consequently, there exists $v \in \Vb$ such that $y^{-1}lyv\Lb=v\Lb$ 
for all $l \in R$. In other words, $(yv)^{-1} R (yv) \subset \Lb$. This means that, 
by replacing $y$ by $yv$ if necessary, we may assume that $y^{-1}Ry \subset \Lb$. 
Therefore, $y^{-1}Ry = Q$ and $P=\{(yly^{-1},l)~|~l \in Q\}$. 

Now, $y^{-1}F(y)\in \Vb \cdot F(\Vb)$ but, since $F(yly^{-1})=yly^{-1}$ for all 
$l \in Q$, we deduce that $y^{-1}F(y) \in C_{\Gb}(Q)$. So 
$$y^{-1}F(y) \in (\Vb \cdot F(\Vb)) \cap C_{\Gb}(Q)=C_\Vb(Q)\cdot F(C_\Vb(Q)) 
\subset C_{\Gb}^\circ(Q)$$
(see Proposition~\ref{prop:centralisateur}(b) and~(d)). So, by Lang's Theorem, 
there exists $x \in C_{\Gb}^\circ(Q)$ such that $y^{-1}F(y)=x^{-1}F(x)$. 
This implies that $h=yx^{-1} \in \Gb^F$, and 
$$P=\{(hlh^{-1},l)~|~l \in Q\},$$
as expected.
\end{proof}

\begin{coro}
\label{cor:deltavertex}
The indecomposable summands of the $\OC(\Gb^F\times N_{\Gb^F}(\Pb,\Lb)^\opp)$-module
$\Grm\Gamma_c(\Yb_\Vb^\Gb,\OC)^{\red}$
have a vertex contained in
$\Delta N_{\Gb^F}(\Pb,\Lb)$.
\end{coro}

\bigskip

\begin{proof}
Let $Q$ be an $\ell$-subgroup of $\Gb^F\times N_{\Gb^F}(\Pb,\Lb)^\opp$ that
is not $(\Gb^F\times 1)$-conjugate to a subgroup of 
$\Delta N_{\Gb^F}(\Pb,\Lb)$.
We have $\brauer_Q(\Grm\Gamma_c(\Yb_\Vb^\Gb,\OC))\simeq
\Grm\Gamma_c((\Yb_\Vb^\Gb)^Q,k)\simeq 0$ in 
$\Ho^b(kN_{\Gb^F\times (\Lb^F)^\opp}(Q))$ by Lemma \ref{lem:delta}.
The result follows now from Lemma \ref{lem:bounded}.
\end{proof}

\bigskip

\subsection{Perfect complexes and disconnected groups} 

\smallskip
Given $M$ a simple $\L\Gb^F$-module, we denote by $\YC(M)$ the set of 
pairs $(\Tb,\Bb)$ such that $\Tb$ is an $F$-stable maximal torus of 
$\Gb$ and $\Bb$ is a (connected) Borel subgroup of $\Gb$ containing $\Tb$
such that 
$\rhom_{\L\Gb^F}^\bullet(R\G_c(\Yb_{\Bb},\Lambda),M) \neq 0$. 
We then set $d(M)=\min_{(\Tb,\Bb) \in \YC(M)} \dim(\Yb_{\Bb})$. 
The following two theorems are proved in~\cite[Theorem~A]{BR} whenever 
$\Gb$ is connected.

\bigskip

\begin{theo}\label{theo:A}
Let $M$ be a simple $\L\Gb^F$-module. Then $\YC(M)\neq \vide$. Moreover, 
given $(\Tb,\Bb) \in \YC(M)$ such that $d(M)=\dim(\Yb_{\Bb})$, we have 
$$\Hom_{\Drm^b(\L\Gb^F)}(R\G_c(\Yb_{\Bb},\Lambda),M[-i]) = 0$$
for all $i \neq d(M)$.
\end{theo}

\bigskip

\begin{proof}
By~(\ref{eq:ind-rlg}), we have 
$$\Hom_{\Drm^b(\L\Gb^F)}(R\G_c(\Yb_{\Bb}^{\Gb},\Lambda),M[-i]) = 
\Hom_{\Drm^b(\L\Gb^{\circ F})}(R\G_c(\Yb_{\Bb}^{\Gb^\circ},\Lambda),
\Res_{\Gb^{\circ F}}^{\Gb^F} M[-i]).
$$
Since $M$ is simple and $\Gb^{\circ F}\lhd\Gb^F$, it follows that
$\Res_{\Gb^{\circ F}}^{\Gb^F} M$ is semisimple. 
Since the theorem holds 
in $\Gb^{\circ F}$ (see~\cite[Proof of Theorem~A]{BR}), we know that $\YC(M)$ is not empty. 
The second statement follows from the fact that, if two simple 
$\L\Gb^{\circ F}$-modules $M_1$ and $M_2$ occur in the semisimple module
$\Res_{\Gb^{\circ F}}^{\Gb^F} M$, then they 
are conjugate under $\Gb^F$, and so $d(M_1)=d(M_2)=d(M)$.
\end{proof}

\bigskip

\begin{theo}\label{theo:A-engendrement}
The triangulated category $\L\Gb^F\mperf$ 
is generated by the complexes $R\G_c(\Yb_{\Bb},\L)$, where 
$\Tb$ runs over the set of $F$-stable maximal tori of $\Gb$ and $\Bb$ runs over the 
set of Borel subgroups of $\Gb$ containing $\Tb$.
\end{theo}

\bigskip

\subsection{Generation of the derived category}
\label{se:genderived}

In this subsection~\S\ref{se:genderived}, we assume $\Lambda=\OC$ or $k$. 
We refer to Appendix~\ref{app:l-perm} for the needed facts about $\ell$-permutation 
modules.

\smallskip
Let $Q$ be an $\ell$-subgroup of $\Gb^F$
and let $M$ be an indecomposable $\ell$-permutation $\L[\Gb^F \times Q^\opp]$-module with vertex
$\D Q$. 
We denote by $\YC[M]$ the set of pairs $(\Tb,\Bb)$ satisfying the following conditions:
\begin{itemize}
\item $\Tb$ is an $F$-stable maximal 
torus of $\Gb$ contained in a Borel subgroup $\Bb$ of $\Gb$ such that
$Q$ normalizes $(\Tb,\Bb)$
\item $M$ is a direct summand of a term of the complex 
$\bigl(\Res_{\Gb^F \times Q^\opp}^{\Gb^F \times N_{\Gb^F}(\Bb,\Tb)^\opp}
 \Grm\G_c(\Yb_{\Bb},\L)\bigr)^\red$. 
\end{itemize}

We set $d[M]=\min_{(\Tb,\Bb) \in \YC[M]} \dim(\Yb_{C^\circ_{\Bb}(Q)}^{C^\circ_{\Gb}(Q)})$.

\bigskip

\begin{lem}\label{lem:l-perm}
If $Q$ normalizes a pair $(\Tb\subset\Bb)$ where $\Tb$ is an $F$-stable
maximal torus and $\Bb$ a Borel subgroup of $\Gb$, then
$\YC[M] \neq \vide$. Moreover, given $(\Tb,\Bb) \in \YC[M]$ 
such that $d[M]=\dim(\Yb_{C^\circ_{\Bb}(Q)})$, the degree $i$ term of the complex 
$\bigl(\Res_{\Gb^F \times Q^\opp}^{\Gb^F \times N_{\Gb^F}(\Bb,\Tb)^\opp} \Grm\G_c(\Yb_{\Bb},\L)\bigr)^\red$ 
has no direct summand isomorphic to $M$ if $i \neq d[M]$.
\end{lem}

\bigskip

\begin{proof}
%
Note that $N_{\Gb^F \times Q^\opp}(\D Q)=(C_\Gb(Q)^F \times 1)\D Q$,
and we identify $C_\Gb(Q)^F$ with $N_{\Gb^F \times Q^\opp}(\D Q)/\D Q$ via the first projection.
Let $V=\brauer_{\D Q}(M)$, an indecomposable projective 
$kC_\Gb(Q)^F$-module. Let $L$ be the simple quotient of $V$.


Now, let $\Bb_Q$ be a Borel subgroup of $C_\Gb(Q)$ admitting an $F$-stable maximal 
torus $\Tb_Q$. By Lemma \ref{lem:bijtori}, $C_\Gb(\Tb_Q)^\circ$ is an $F$-stable maximal torus of $\Gb$ and it is contained
in a Borel subgroup $\Bb$ of $\Gb$ such that $\Bb_Q=C_{\Bb}^\circ(Q)$.

We set 
$D=\bigl(\Res_{C_\Gb(Q)^F \times 1}^{C_\Gb(Q)^F \times \Tb_Q^{F\opp}} 
\Grm\G_c(\Yb_{\Bb_Q}^{C_\Gb(Q)},k)\bigr)^\red$.
By Proposition~\ref{prop:centralisateur}(e), we have
$$\brauer_{\D Q}(\Grm\G_c(\Yb_{\Bb}^\Gb,\L))
\simeq \Grm\G_c\left((\Yb_{\Bb}^\Gb)^{\D Q},k\right)
\simeq \Grm\G_c(\Yb_{\Bb_Q}^{C_\Gb(Q)},k) \simeq D$$
in $\Ho^b(kC_\Gb(Q)^F)$. It follows from Lemma~\ref{lem:bounded} that $M$ is a direct summand 
of the $i$-th term of 
$\bigl(\Res_{\Gb^F \times Q^\opp}^{\Gb^F\times\Tb_Q^{F\opp}} \Grm\G_c(\Yb_{\Bb}^\Gb,\L)\bigr)^\red$ 
if and only if $V$ is a direct summand of $D^i$. So the result follows from 
Theorem~\ref{theo:A}. Note that $d[M]=d[V]=d(L)=\dim \Yb_{\Bb_Q}^{C_\Gb(Q)}$.
\end{proof}

\bigskip

Recall that given a Borel subgroup $\Bb$ of $\Gb$ with an $F$-stable maximal torus $\Tb$, the variety
$\Yb_\Bb$ has a right action of $N_{\Gb^F}(\Tb,\Bb)$ (cf Remark \ref{re:actionN}).

\smallskip
Let $\AC$ be the thick subcategory of $\Ho^b(\L \Gb^F)$ generated by the complexes 
of the form 
$$\Grm\G_c(\Yb_{\Bb},\L)\otimes_{\L Q} L,$$
where 
\begin{itemize}
\item $\Tb$ runs over $F$-stable maximal tori of $\Gb$
\item $\Bb$ runs over Borel subgroups of $\Gb$ containing $\Tb$
\item $Q$ is an $\ell$-subgroup of $N_{\Gb^F}(\Tb,\Bb)$
\item and $L$ is a $\L Q$-module, free of rank $1$ over $\L$. 
\end{itemize}

Let $\BC$ be the full subcategory of $\L\Gb^F\modules$ consisting of modules whose indecomposable direct 
summands have a one-dimensional source and a vertex $Q$ which normalizes a pair $(\Tb\subset\Bb)$
where $\Tb$ is an $F$-stable maximal torus and $\Bb$ a Borel subgroup of $\Gb$.

\begin{theo}\label{theo:engendrement}
We have $\AC=\Ho^b(\BC)$.
\end{theo}

\bigskip

\begin{proof}
Given $N$ an indecomposable $\L\Gb^F$-module with a one-dimensional source 
$L$ and a vertex $Q$ which normalizes a pair $(\Tb\subset\Bb)$, where
$\Tb$ is an $F$-stable maximal torus and $\Bb$ a Borel subgroup, we set  
$d[N]$ to be the minimum of the numbers $d[M]$, where $M$ runs over the set of 
indecomposable $\ell$-permutation $\L(\Gb^F \times Q^\opp)$-modules with vertex $\D Q$ and 
such that $N$ is a direct summand of $M \otimes_{\L Q} L$. 

Note that if $M$ is an indecomposable $\ell$-permutation $\L(\Gb^F \times Q^\opp)$-module with vertex 
properly contained in $\D Q$, then the indecomposable direct summands of $M \otimes_{\L Q} L$ 
have vertices of size $< |Q|$ and a one-dimensional source.
Since the $\L(\Gb^F \times Q^\opp)$-module $\L G$ is a direct sum 
of indecomposable modules with vertices contained in $\D Q$, we deduce that there is an 
indecomposable $\ell$-permutation $\L(\Gb^F \times Q^\opp)$-module $M$ with vertex $\D Q$ 
and such that $N$ is a direct summand of $M \otimes_{\L Q} L$. 

\smallskip
We now proceed by induction on the pair $(|Q|,d[N])$ (ordered lexicographically) 
to show that $N \in \AC$. Fix $M$ an indecomposable $\ell$-permutation $\L(\Gb^F \times Q^\opp)$-module 
$M$ with vertex $\D Q$ and such that $N$ is a direct summand of $M \otimes_{\L Q} L$, 
with $d[N]=d[M]$. Let $(\Tb,\Bb) \in \YC[M]$ be such that 
$\dim(\Yb_{\Bb})=d[M]$ and let 
$D=\bigl(\Res_{\Gb^F \times Q^\opp}^{\Gb^F \times N_{\Gb^F}(\Tb,\Bb)^{\opp}}
\Grm\G_c(\Yb_{\Bb}^\Gb,\L)\bigr)^\red$. 

If $i \neq d[M]$, then Lemma \ref{lem:l-perm} and Corollary
\ref{cor:deltavertex} show that the indecomposable 
direct summands $M'$ of $D^i$ have vertices of size $< |Q|$, or have vertex $\D Q$ and satisfy $d[M'] < d[M]$. 
Therefore, the indecomposable direct summands $N'$ of $D^i \otimes_{\L Q} L$ 
have vertices of size $< |Q|$ or have vertex $Q$ and satisfy $d[N'] < d[N]$. We 
deduce from the induction hypothesis that $D^i \otimes_{\L Q} L \in \AC$ for $i \neq d[N]$. 
Since $N$ is a direct summand of $D^{d[N]} \otimes_{\L Q} L$ and $D \otimes_{\L Q} L \in \AC$ 
by construction, we deduce that $N \in \AC$.
\end{proof}

\bigskip

\begin{coro}\label{coro:engendrement}
Assume that every elementary abelian $\ell$-subgroup of $\Gb^F$ normalizes a pair
$(\Tb\subset\Bb)$ where $\Tb$ is an $F$-stable maximal torus and $\Bb$ a Borel subgroup
of $\Gb$. Then $D^b(\L \Gb^F)$ is generated, as a triangulated category closed under
direct summands, by the complexes
$\Rrm\Gamma_c(\Yb_\Bb,\L)\otimes_{\L Q} L$, where $\Tb$ runs over the set of
$F$-stable maximal 
tori of $\Gb$, $\Bb$ runs over the set of Borel subgroups of $\Gb$ containing $\Tb$, 
$Q$ runs over the set of $\ell$-subgroups of $N_{\Gb^F}(\Tb,\Bb)$ 
and $L$ runs over the set of (isomorphism classes) of $\L Q$-modules 
which are free of rank $1$ over $\L$.
\end{coro}

\begin{proof}
Since the category $\Drm^b(\L\Gb^F)$ is generated, as a triangulated category closed under
taking direct summands, by indecomposable modules with 
elementary abelian vertices and one-dimensional source~\cite[Corollary~2.3]{Rou3}, 
the statement follows from Theorem \ref{theo:engendrement}.
\end{proof}

\bigskip

\begin{rema}
It is easy to show conversely that if $D^b(\L \Gb^F)$ is generated by the
 complexes
$\Rrm\Gamma_c(\Yb_\Bb,\L)\otimes_{\L Q} L$
as in Corollary 
\ref{coro:engendrement}, then $\Drm^b(\L\Gb^F)$ is
generated by indecomposable modules with a one-dimensional source and
an elementary abelian vertex that normalizes a pair
$(\Tb\subset\Bb)$ where $\Tb$ is an $F$-stable maximal torus and $\Bb$ a Borel subgroup.

In particular, the generation assumption
for $\Lambda=k$ implies that all elementary abelian $\ell$-subgroups of $\Gb^F$
are contained in maximal tori.

The particular case
$\Gb^F=\Gb\Lb_n(\FM_q)$ (for arbitrary $n$) is enough to ensure that $\Drm^b(H)$ is
generated by indecomposable modules with 
elementary abelian vertices and one-dimensional source, for any finite group $H$ --- this fact is
a straightforward consequence of Serre's product of Bockstein's Theorem, but we know of no
other proof. It would be
interesting to find a direct proof of that result for $\Gb\Lb_n(\FM_q)$.
\end{rema}

\medskip
Recall that an element of $G_0(\L\Gb^F)$ is {\em uniform} if it is in the image of
$\sum_{\Tb}R_\Tb^\Gb(G_0(\L\Tb^F))$, where $\Tb$ runs over the set of
$F$-stable maximal tori of $\Gb$.

One can actually describe exactly which complexes are ``uniform''.

\begin{coro}
\label{cor:uniform}
Let $\TC$ be the full triangulated subcategory of $D^b(\L \Gb^F)$ generated by
the complexes
$\Rrm\Gamma_c(\Yb_\Bb,\L)\otimes_{\L N_{\Gb^F}(\Tb,\Bb)} M$
where $\Tb$ runs over the set of $F$-stable maximal 
tori of $\Gb$, $\Bb$ runs over the set of Borel subgroups of $\Gb$ containing $\Tb$
and $M$ runs over the set of (isomorphism classes) of finitely generated
$\L N_{\Gb^F}(\Tb,\Bb)$-modules.
Assume that every elementary abelian $\ell$-subgroup of $\Gb^F$ normalizes a pair
$(\Tb\subset\Bb)$ where $\Tb$ is an $F$-stable maximal torus and $\Bb$ a Borel subgroup
of $\Gb$.

An object $C$ of $\Drm^b(\L \Gb^F)$ is in $\TC$ if and only if 
$[C]\in G_0(\L\Gb^F)$ is uniform.
\end{coro}

\begin{proof}
The statement follows from Corollary \ref{coro:engendrement} and from Thomason's classification of
full triangulated dense subcategories \cite[Theorem 2.1]{Tho}.
\end{proof}

\medskip

\begin{rema}
Note that Corollary \ref{cor:uniform} holds also for $\Lambda=K$: in the proof, Corollary
\ref{coro:engendrement} is replaced by Theorem \ref{theo:A-engendrement}.
\end{rema}

\medskip

\begin{exemples}
\label{ex:elementaryabelian}
(1) If $\Gb=\Gb\Lb_n(\FM)$ or $\Sb\Lb_n(\FM)$, then all abelian subgroups 
consisting of semisimple elements are contained in maximal tori. 
This just amounts to the classical result in linear algebra which says 
that a family of commuting semisimple elements always admits a basis 
of common eigenvectors.

\medskip

(2) Assume $\Gb$ is connected. Let $\pi(\Gb)$
denote the set of prime numbers which are bad for $\Gb$ or divide 
$|(\Zrm(\Gb^*)/\Zrm(\Gb^*)^\circ)^{F^*}|$.
If $\ell \not\in \pi$ and 
if $t$ is an $\ell$-element of $\Gb^F$, then $C_\Gb(t)$ is a
Levi subgroup of $\Gb$ and $\pi(C_\Gb(t))\subset\pi(\Gb)$ 
\cite[Proposition 13.12(iii)]{CaEnbook}.
An induction argument shows the following fact.
\equat\label{eq:bons-bons}
\text{\it If $\ell \not\in \pi$, then all abelian $\ell$-subgroups 
of $\Gb^F$ are contained in maximal tori.}
\endequat
So Corollary~\ref{coro:engendrement} can be applied if $\ell \not\in \pi$. This generalizes 
(1).
\end{exemples}

\bigskip

\begin{contre}
Assume here, and only here, that $\ell = 2$ (so that $p \neq 2$) and that $\Gb=\Pb\Gb\Lb_2(\FM)$. 
Let $t$ (respectively $t'$) denote the class of the matrix $\begin{pmatrix} 1 & 0 \\ 0 & -1 \end{pmatrix}$
(respectively $\begin{pmatrix} 0 & 1 \\ 1 & 0 \end{pmatrix}$) in $\Gb$. 
Then $\langle t,t'\rangle$ is an elementary abelian $2$-subgroup of $\Gb$ 
which is not contained in any maximal torus of $\Gb$ (indeed, since $\Gb$ has rank $1$, 
all finite subgroups of maximal tori of $\Gb$ are cyclic).
\end{contre}

\section{Rational series}
\label{se:rational}

\medskip

\subsection{Rational series in connected groups}
We assume in this subsection~\S\ref{se:rationalseries} that $\Gb$ is connected.

\label{se:rationalseries}
\medskip

\medskip

Let $d$ be a positive integer divisible by $\delta$ and such that 
$(wF)^d(t)=t^{q^{d/\delta}}$ for all $t\in\Tb$ and $w\in N_{\Gb}(\Tb)$. Let
$\zeta$ be a generator of $\FM_{\!q^{d/\d}}^\times$.
Recall \cite[Proposition 13.7]{dmbook} that the map
$$\fonction{\Nrm}{Y(\Tb)}{\Tb^F}{\l}{N_{F^d/F}(\l(\z))=
\l(\z)\,\lexp{F}{(\l(\z))}\cdots\lexp{F^{d-1}}{(\l(\z))}}$$
is surjective and it induces an isomorphism $Y(\Tb)/(F-1)(Y(\Tb))\xrightarrow{\sim}\Tb^F$.
The morphism 
$$Y(\Yb)\times X(\Tb)\to K^\times,\ (\lambda,\mu)\mapsto \zeta^{\langle\mu,\lambda+F(\lambda)+\cdots+
F^{d-1}(\lambda)\rangle}$$
factors through $N\times 1$ and induces a morphism $\Tb^F\times X(\Tb)\to K^\times$. The corresponding morphism
$X(\Tb)\to\Hom(\Tb^F,K^\times)=\Irr(\Tb^F)$ is surjective and induces an isomorphism
$X(\Tb)/(F-1)(X(\Tb))\xrightarrow{\sim}\Irr(\Tb^F)$ \cite[Proposition 13.7]{dmbook}.


Let $(\Gb^*,\Tb^*,F^*)$ be a triple dual to $(\Gb,\Tb,F)$ \cite[Definition 5.21]{DL}.
The isomorphisms $X(\Tb)/(F-1)(X(\Tb))\xrightarrow{\sim}\Irr(\Tb^F)$ and $X(\Tb)/(F-1)(X(\Tb))=
Y(\Tb^*)/(F^*-1)(Y(\Tb^*))\xrightarrow{\sim}\Tb^{*F}$ induce an isomorphism
$\Irr(\Tb^F)\xrightarrow{\sim}\Tb^{*F^*}$.

\medskip
Let $(\Tb,\th) \in \nabla(\Gb,F)$ and let $\Phi$ (respectively $\Phi^\vee$) 
denote the root (respectively coroot) system of $\Gb$ relative to $\Tb$. 

We set $\th^Y=\th \circ \Nrm : Y(\Tb) \to K^\times$ and
$$\Phi^\vee(\th)=\Phi^\vee \cap \Ker(\th^Y).$$

Note that $\Phi^\vee(\th)$ is closed and symmetric, hence it defines a root system. We
denote by $W_\Gb^\circ(\Tb,\th)$ its Weyl group.
It is a subgroup of the Weyl group $N_\Gb(\Tb)/\Tb$ and it is
contained in the stabilizer $W_\Gb(\Tb,\th)$ of $\th^Y$.

\medskip
This can be translated as follows in the dual group \cite[Proposition 2.3]{dmbook}.
Let $s \in \Tb^{*F^*}$ be the element corresponding to $\th$. 
Identifying the coroot system $\Phi^\vee$ with the root system of
$\Gb^*$, we obtain that
$$\Phi^\vee(\th)=\{\a^\vee \in \Phi^\vee~|~\a^\vee(s)=1\}$$
is the root system of $C_{\Gb^*}(s)$. If $\Vb^*$ is a unipotent subgroup of $\Gb^*$ normalized by
$\Tb^*$, then $C_{\Vb^*}(s)$ is generated by the one-parameter subgroups of $\Gb^*$ normalized by $\Tb^*$,
contained in $\Vb^*$, and corresponding to elements of $\Phi^\vee(\th)$.

The group $W^\circ_\Gb(\Tb,\th)$ is identified with the Weyl group $W^\circ(\Tb^*,s)$ 
of $C_{\Gb^*}^\circ(s)$ relative to $\Tb^*$ while $W_\Gb(\Tb,\th)$ is identified 
with the Weyl group $W(\Tb^*,s)$ of $C_{\Gb^*}(s)$. 

\medskip
Recall that $(\Tb_1,\th_1)$ and $(\Tb_2,\th_2)$ are in the same {\em geometric series} if
there exists $x \in \Gb$ such that 
$(\Tb_2,\th_2^Y) = \lexp{x}{(\Tb_1,\th_1^Y)}$ and $x^{-1}F(x)\Tb_1 \in W_\Gb(\Tb_1,\th_1)$.
The pairs are in the same {\em rational series} if in addition the element
$s_2\in \Tb_1^{*F^*}$ corresponding to $\lexp{x^{-1}}{\th_2}$ is
$\Gb^{*F^*}$-conjugate to $s_1$.
We have now a direct description of rational series.

\bigskip


\begin{prop}\label{prop:rational-series}
The pairs $(\Tb_1,\th_1)$ and $(\Tb_2,\th_2)$ are in the same rational series 
if and only if there exists $x \in \Gb$ such that 
$(\Tb_2,\th_2^Y) = \lexp{x}{(\Tb_1,\th_1^Y)}$ and $x^{-1}F(x)\Tb_1 \in W^\circ_\Gb(\Tb_1,\th_1)$. 
\end{prop}

%
%
%

\begin{proof}
Note that given $x\in\Gb$ such that $\lexp{x}{\Tb_1}$ is $F$-stable, then
$x^{-1}F(x)\in N_{\Gb_1}(\Tb_1)$.

Let $\Tb_i^*$ be an $F^*$-stable maximal torus of $\Gb^*$ and let $s_i \in \Tb_i^{*F^*}$ 
be such that the $\Gb^{*F^*}$-orbit of $(\Tb_i^*,s_i)$ corresponds to the $\Gb^F$-orbit 
of $(\Tb_i,\th_i)$. Then the statement of the proposition is equivalent to the following:
\begin{itemize}
\item[$(*)$] {\it $s_1$ and $s_2$ are $\Gb^{*F^*}$-conjugate if and only if there exists 
$x \in \Gb^*$ such that $(\Tb_2^*,s_2)=\lexp{x}{(\Tb_1^*,s_1)}$ and $x^{-1}F^*(x)\Tb_1^*
\in W^\circ(\Tb_1^*,s_1)$.}
\end{itemize}
So let us prove $(*)$. 

\medskip

First, if $s_1$ and $s_2$ are $\Gb^{*F^*}$-conjugate, then there exists $x \in \Gb^{*F^*}$ such that 
$s_2=xs_1x^{-1}$. Then $\Tb_1^*$ and $x^{-1}\Tb_2^*x$ are two maximal tori of $C_{\Gb^*}^\circ(s_1)$, so 
there exists $y \in C_{\Gb^*}^\circ(s_1)$ such that $y\Tb_1^* y^{-1} = x^{-1} \Tb_2^* x$. Then
$(\Tb_2^*,s_2)=\lexp{xy}{(\Tb_1^*,s_1)}$ and
$$(xy)^{-1} F^*(xy)=y^{-1}F^*(y) \in C_{\Gb^*}^\circ(s_1),$$
as desired.

\medskip

Conversely, assume that there exists $x \in \Gb^*$ such that $(\Tb_2^*,s_2)=\lexp{x}{(\Tb_1^*,s_1)}$ 
and $x^{-1}F^*(x)\Tb_1^* \in W^\circ(\Tb_1^*,s_1)$. By Lang's Theorem applied to the connected
group $C_{\Gb^*}^\circ(s_1)$, 
there exists $y \in C_{\Gb^*}^\circ(s_1)$ such that $x^{-1}F^*(x)=y^{-1}F^*(y)$. Then 
$xy^{-1} \in \Gb^{*F^*}$ and $s_2=xy^{-1} s_1 y x^{-1}$. The proof of $(*)$ is complete.
\end{proof}

\bigskip

We can now translate the properties of regularity and super-regularity defined in~\cite[\S{11.4}]{BR}. 
Let $\Pb$ be a parabolic subgroup of $\Gb$ and let $\Lb$ be a Levi subgroup of $\Pb$. We assume that 
$\Lb$ is $F$-stable. Let $\XC \subset \nabla(\Lb,F)$ be a rational series.

\bigskip

\begin{prop}\label{prop:regularity}
The rational series $\XC$ is $(\Gb,\Lb)$-regular (respectively $(\Gb,\Lb)$-super-regular) if 
and only if $W^\circ_\Gb(\Tb,\th) \subset \Lb$ (respectively $W_\Gb(\Tb,\th) \subset \Lb$) 
for some (or any) pair $(\Tb,\th) \in \XC$. 
\end{prop}

\bigskip

\begin{proof}
This follows immediately from~\cite[Lemma~11.6]{BR}.
\end{proof}

\bigskip

\subsection{Coroots of fixed points subgroups}

\medskip
We consider now again a non-necessarily connected reductive group $\Gb$.

We fix an element $g \in \Gb$ which stabilizes a pair $(\Tb,\Bb)$ 
where $\Bb$ is a Borel subgroup of $\Gb$ and $\Tb$ is a maximal torus of $\Bb$. 
Such an element is called {\it quasi-semisimple} in~\cite{dm-nonc} and~\cite{dm-nonc-2}. 
For instance, any semisimple element of $\Gb$ is quasi-semisimple. 
Recall from~\cite[Theorem~1.8]{dm-nonc} that $C_\Gb(g)^\circ$ is a reductive group, 
that $C_\Bb(g)^\circ=\Bb \cap C_\Gb(g)^\circ$ is a Borel subgroup of $C_\Gb(g)$ and that 
$C_\Tb(g)^\circ=\Tb \cap C_\Gb(g)^\circ$ is a maximal 
torus of $C_\Bb(g)$. We shall be interested in determining the coroot system 
of the fixed points subgroup $C_\Gb(g)^\circ$.

Let $\Phi$ (respectively $\Phi^\vee$) be the root (respectively coroot) system of $\Gb^\circ$ 
relative to $\Tb$. Let $\Phi(g)$ (respectively $\Phi^\vee(g)$) denote the 
root (respectively coroot) system of $C_\Gb(g)^\circ$ relative to $C_\Tb(g)^\circ$. 
If $\O$ is a $g$-orbit in $\Phi$, we denote by $c_\O \in \FM^\times$ 
the scalar by which $g^{|\O|}$ acts on the one-parameter unipotent subgroup associated 
with $\a$ (through any identification of this one-parameter subgroup with the additive group 
$\FM$). We denote by $(\Phi/g)^a$ the set of $g$-orbits $\O$ in 
$\Phi$ such that there exist $\a$, $\b \in \O$ such that $\a + \b \in \Phi$. 
We denote by $(\Phi/g)^b$ the set of other orbits. We set
$$\Phi[g]=\{\O \in (\Phi/g)^a~|~c_\O= 1 \text{ and }p{\not=}2\} \cup \{\O \in (\Phi/g)^b~|~c_\O = 1\}.$$
Finally, if $\O \in (\Phi/g)^a$ (respectively $\O \in (\Phi/g)^b$), let 
$\overline{\O}^\vee=2\sum_{\a \in \O} \a^\vee$ 
(respectively $\overline{\O}^\vee=\sum_{\a \in \O} \a^\vee$).
Note that $\overline{\O}^\vee$ 
is $g$-invariant, so it belongs to $Y(\Tb)^g=Y(C_\Tb(g)^\circ)$.

\bigskip

\begin{prop}\label{prop:cooot}
$\Phi^\vee(g)=\{\overline{\O}^\vee~|~\O \in \Phi[g]\}$. 
\end{prop}

\bigskip

\begin{proof}
The statement depends only on the automorphism induced by $g$ on $\Gb^\circ$ and can be proved 
with assuming that $\Gb^\circ$ is semisimple. Since this automorphism 
can then be lifted uniquely to the simply-connected covering of $\Gb^\circ$ 
(see~\cite[9.16]{steinberg}), we may also assume that $\Gb^\circ$ 
is simply-connected. Therefore, $g$ permutes the irreducible components of $\Gb^\circ$ 
and an easy 
reduction argument shows that we may assume that $\Gb^\circ$ is quasi-simple. 
Let $\Ub$ denote the unipotent radical of $\Bb$, $\Ub^-$ the unipotent radical 
of the opposite Borel subgroup and, if $\a \in \Phi$, let $\Ub_{\a}$ denote 
the corresponding one-parameter unipotent subgroup. We also denote by $\Gb_{\a}$ 
the subgroup generated by $\Ub_{\a}$ and $\Ub_{-\a}$: it is isomorphic to 
$\Sb\Lb_2(\FM)$ because $\Gb^\circ$ is simply-connected. 

If $\O \in \Phi/g$, we denote by $\Ub_\O$ the unipotent subgroup generated by $(\Ub_\a)_{\a \in \O}$. 
We follow the proof of~\cite[Theorem~8.2]{steinberg}. According to this proof, 
any one-parameter unipotent subgroup $\Vb$ normalized by $C_\Tb(g)^\circ$ is contained 
in some of these $\Ub_\O$'s, and one of the following holds:
\begin{itemize}
 \itemth{1} $\O \in (\Phi/g)^b$ and $c_\O=1$.
 
 \itemth{2} $\O \in (\Phi/g)^a$ and $c_\O=-1$.
 
 \itemth{3} $\O \in (\Phi/g)^a$, $c_\O=1$ and $p \neq 2$.
\end{itemize}
In all cases, $\Vb=C_{\Ub_\O}(g)$. Let $\Vb^-=C_{\Ub_\O}(g)$. 

In case (1), as $[\Ub_\a,\Ub_\b]=1$ if $\a$, $\b \in \O$, the group $\langle \Ub_\O,\Ub_{-\O} \rangle$
is a direct product of groups isomorphic to $\Sb\Lb_2(\FM)$ which are permuted by $g$. It then follows that 
the coroot corresponding to the one-parameter subgroup $\Vb=C_{\Ub_{\O}}(g)$ 
(since $c_\O=1$) is equal to $\o^\vee=\overline{\O}^\vee$. 

In cases (2) or (3), it follows from the classification that $|\O|=2$ (this case only 
occurs in type $A_{2n}$). Let $\a \in \O$. Then $\Ub_\O=\Ub_\a\Ub_{g(\a)}\Ub_{\a+g(\a)}$. 
In case (2),  the computations done in~\cite[Proof~of~Theorem~8.2,~item~($2''''$)]{steinberg} show that 
$\Vb=\Ub_{\a+g(\a)}$. Therefore $\Vb \subset \Ub_{\O'}$, where $\O'$ is the $g$-orbit (of cardinality $1$) 
of $\a + g(\a)$ and $\O' \in (\Phi/g)^b$, and we are back to case (1). 

In case~(3), the computations done in~\cite[Proof~of~Theorem~8.2,~item~($2''''$)]{steinberg} show that 
$\langle \Vb,\Vb^-\rangle \simeq \Sb\Ob_3(\FM) \simeq \Pb\Gb\Lb_2(\FM)$ and that 
the associated coroot is $2(\a^\vee+g(\a^\vee))=\overline{\O}^\vee$. 
\end{proof}

\bigskip
\begin{rema}\label{rem:ordre pair}
If $\O \in (\Phi/g)^a$, then it follows from the classification 
that $|\O|$ is even, and so the order of $g$ is even.
\end{rema}

\bigskip

\subsection{Centralizers and rational series}\label{sec:fixed points}

\medskip

Let $g \in \Gb^F$ be a quasi-semisimple element of $\Gb$.
Let $(\Sb,\th) \in \nabla(C_\Gb^\circ(g),F)$. We then set $\Sb^+=C_{\Gb^\circ}(\Sb)$. 
It follows from~\cite[Theorem~1.8(iv)]{dm-nonc} that $\Sb^+$ is a maximal torus of $\Gb^\circ$ 
(containing $\Sb$).
 It is stable under the action of $g$, so we have a 
map $\LC_g : \Sb^+ \to \Sb^+$, $t \mapsto t^{-1} gtg^{-1}=[g,t]$ (which is a morphism of groups because 
$\Sb^+$ is abelian). If $t=\LC_g(s)$, then $t\lexp{g}{t}\lexp{g^2}{t}\cdots\lexp{g^{m-1}}{t}=\LC_{g^{m}}(s)$. 
In particular, if $t \in C_{\Sb^+}(g)=\Ker \LC_g$, then 
$t^{m}=\LC_{g^m}(s)$. This shows that any element of $C_{\Sb^+}(g) \cap \LC_g(\Sb^+)$ has order 
dividing the order of $g$. Note further that
$C_{\Sb^+}(g)^\circ = \Sb$  (see~\cite[Theorem~1.8(iii)]{dm-nonc}). We have
\begin{align*}
\dim(\Sb \cdot \LC_g(\Sb^+))&=\dim(\Sb)+\dim(\LC_g(\Sb^+))=
\dim(\Sb)+\dim(\Sb^+)-\dim(\ker\LC_g)^\circ\\
&=\dim(\Sb)+\dim(\Sb^+)-\dim(\Sb)=\dim(\Sb^+).
\end{align*}
We deduce that
\equat\label{eq:fini-ell}
\text{$\Sb^+=\Sb \cdot \LC_g(\Sb^+)$ and $\Sb \cap \LC_g(\Sb^+)$ is finite of exponent dividing the order of $g$.}
\endequat
Now, if $H$ is a $g$-stable finite subgroup of $\Sb^+$ of order prime to the order of $g$ then 
$C_H(g) \subset \Sb$ (because the order of $C_{\Sb^+}(g)/\Sb$ divides the order of $g$ 
by~\cite[Proposition~1.28]{dm-nonc}) and 
\equat\label{eq:produit-h}
H = C_H(g) \times \LC_g(H).
\endequat
So, if the linear character $\th$ of $\Sb^F$ has order prime to the order of $g$, then it can be extended 
canonically to a linear character $\th^+$ of $\Sb^{+F}$ as follows (we use the discussion above with $H$ the subgroup
of $\Sb^{+F}$ of elements with order dividing a power of the order of $\th$): $\th^+$ is trivial on $\LC_g(\Sb^{+F})$, 
is trivial on elements of $\Sb^{+F}$ of order prime to the order of $\th$ and coincides with $\th$ 
on $\Sb^F$. The fact that $\th^+$ is trivial on $\LC_g(\Sb^{+F})$ is equivalent to
\equat\label{eq:theta-plus-trivial}
\text{$\th^+$ is $g$-stable.}
\endequat
Note that, since $\Sb^+ \cap C_\Gb^\circ(g)=\Sb$ by~\cite[Theorem~1.8]{dm-nonc}, we may identify
the Weyl group of $C_\Gb^\circ(g)$ relative to $\Sb$ to a subgroup of the Weyl group 
of $\Gb^\circ$ relative to $\Sb^+$. Through this identification, we get:

\bigskip

\begin{lem}\label{lem:w0}
If the order of $\th$ is prime to the order of $g$, then $W_{C_\Gb^\circ(g)}(\Sb,\th) \subset
W_{\Gb^\circ}(\Sb^+,\th^+)$ 
and $W^\circ_{C_\Gb^\circ(g)}(\Sb,\th) \subset W^\circ_{\Gb^\circ}(\Sb^+,\th^+)$.
\end{lem}

\bigskip

\begin{proof}
Let $w \in W_{C_\Gb^\circ(g)}(\Sb,\th)$. Then $w$ stabilizes $\Sb^+=C_{\Gb^\circ}(\Sb)$ and its action 
on $\Sb$ commutes with the action of $g$. So it follows from the construction of $\th^+$ 
that $w$ stabilizes $\th^+$.

Let us now prove the second statement. Let $\a^\vee$ be a coroot of $C_\Gb^\circ(g)$ 
relative to $\Sb$ such that $\th^Y(\a^\vee)=1$. Let $s_{g,\a}$ denote 
the corresponding reflection in $W^\circ_{C_\Gb^\circ(g)}(\Sb,\th)$. It is sufficient to prove 
that $s_{g,\a} \in W^\circ_{\Gb^\circ}(\Sb^+,\th^+)$. 
Then it follows from Proposition~\ref{prop:cooot}
 that there exists a coroot $\b^\vee$ of $\Gb^\circ$ 
relative to $\Sb^+$ and $m \in \{1,2\}$ such that 
$$\a^\vee= m \sum_{i=0}^{r-1} g^i(\b^\vee),$$
where $r \ge 1$ is minimal such that $g^r(\b^\vee)=\b^\vee$. It follows 
from Remark~\ref{rem:ordre pair} that, if $m=2$, then $g$ has even order. Now,
$$1=\th^{+Y}(\a^\vee)=\prod_{i=1}^{r-1} \th^{+Y}(g^i(\b^\vee))^m=\th^{+Y}(\b^\vee)^{mr},$$
because $\th^+$ is $g$-stable. Since $m$ and $r$ divide the order of $g$, $mr$ is prime 
to the order of $\th^+$, so this implies that $\th^{+Y}(\b^\vee)=1$.
 In particular, 
$$s_\b,s_{g(\b)},\dots,s_{g^{r-1}(\b)} \in W^\circ_{\Gb^\circ}(\Sb^+,\th^+).$$
It follows from~\cite[Proof~of~Theorem~8.2,~Statement~($2'''$)]{steinberg} 
that then $s_{g,\a}$ belongs to the subgroup generated by $s_\b$, $s_{g(\b)}$,\dots, $s_{g^{r-1}(\b)}$. 
\end{proof}

\bigskip
Let $(\Tb_1,\theta_1),(\Tb_2,\theta_2)\in\nabla(\Gb,F)$.
We say that
$(\Tb_1,\theta_1)$ and $(\Tb_2,\theta_2)$ are {\em geometrically conjugate} (resp. in the same
{\em rational series}) if there is $t\in N_{\Gb^F}(\Tb_1)$ such that 
$(\Tb_1,\lexp{t}{\theta_1})$ and $(\Tb_2,\theta_2)$ are 
geometrically conjugate (resp. in the same rational series)
for $\nabla(\Gb^\circ,F)$. We denote by $\nabla(\Gb,F)/\equiv$ the set of rational series.

\bigskip
Let $Q$ be the subgroup of $\Gb$ generated by $g$ and let $\Nb$ be a subgroup of $N_\Gb(Q)$ containing $C_\Gb(Q)$.

\begin{coro}\label{coro:brauer-rational}
Let $(\Sb_1,\th_1),(\Sb_2,\th_2)\in\nabla_{|g|'}(\Nb,F)$ .
\begin{itemize}
\itemth{a} If $(\Sb_1,\th_1)$ and $(\Sb_2,\th_2)$ are geometrically conjugate in $\Nb$, 
then $(\Sb_1^+,\th_1^+)$ and $(\Sb_2^+,\th_2^+)$ are geometrically conjugate in $\Gb$.

\itemth{b} If $(\Sb_1,\th_1)$ and $(\Sb_2,\th_2)$ are in the same rational series of $\Nb$, 
then $(\Sb_1^+,\th_1^+)$ and $(\Sb_2^+,\th_2^+)$ are in the same rational series of $\Gb$.
\end{itemize}
So, the injective map $\nabla_{|g|'}(\Nb,F)\to \nabla_{|g|'}(\Gb,F),\
(\Sb,\th)\mapsto (\Sb^+,\th^+)$ induces a map 
$$i_Q^{\Gb}:\nabla_{|g|'}(\Nb,F)/\equiv\ \ \longto\ \ \nabla_{|g|'}(\Gb,F)/\equiv.$$
\end{coro}

\bigskip

\begin{proof}
(a) If $(\Sb_1,\th_1)$ and $(\Sb_2,\th_2)$ are geometrically conjugate in $\Nb^\circ=C_\Gb^\circ(g)$ then, 
by definition, there exists $x \in C_\Gb^\circ(g)$ such that $\Sb_2=\lexp{x}{\Sb_1}$ and 
$\th_2^Y = \lexp{x}{\th_1^Y}=\lexp{F(x)}{\th_1^Y}$ (as linear characters of $Y(\Sb_2)$). 
Since $x$ commutes with $g$, it sends $\LC_g(\Sb_1^+)$ to $\LC_g(\Sb_2^+)$, so it is immediately 
checked that $\th_2^{+Y}= \lexp{x}{\th_1^{+Y}}=\lexp{F(x)}{\th_1^{+Y}}$. The case
of geometric conjugacy in $\Nb$ and $\Gb$ follows immediately.

\medskip

(b) If $(\Sb_1,\th_1)$ and $(\Sb_2,\th_2)$ are in the same rational series of $C_\Gb^\circ(g)$, 
then, by Proposition~\ref{prop:rational-series}, 
there exists $x \in C_\Gb^\circ(g)$ such that $\Tb_2=\lexp{x}{\Tb_1}$, 
$\th_2^Y = \lexp{x}{\th_1^Y}$ (as linear characters of $Y(\Sb_2)$) 
and $x^{-1}F(x) \in W^\circ_{C_\Gb^\circ(g)}(\Sb_1,\th_1)$. So the result follows from (a) 
and from Propositions~\ref{prop:rational-series} and~\ref{prop:cooot}.
The case of rational series in $\Nb$ and $\Gb$ follows immediately.
\end{proof}

\bigskip

Let $\Lb$ be an $F$-stable Levi complement of a parabolic subgroup $\Pb$ of $\Gb$ containing $g$. 
Then $C_\Lb(g)$ is an $F$-stable Levi complement of $C_\Pb(g)$ \cite[Proposition~1.11]{dm-nonc}.

\bigskip

\begin{coro}\label{coro:super-regular}
Let $\XC\in\nabla_{|g|'}(C_\Lb^\circ(g),F)/\equiv$ be a rational series.
If $i_Q^\Lb(\XC)$ is $(\Gb^\circ,\Lb^\circ)$-regular (respectively 
$(\Gb^\circ,\Lb^\circ)$-super regular), then $\XC$ is $(C_\Gb^\circ(g),C_\Lb^\circ(g))$-regular 
(respectively $(C_\Gb^\circ(g),C_\Lb^\circ(g))$-super regular).
\end{coro}

\bigskip

\begin{proof}
This follows from Proposition~\ref{prop:regularity} and Lemma \ref{lem:w0}.
\end{proof}

\bigskip

The results above extend by induction to general nilpotent $p'$-subgroups.
Let $Q$ be a nilpotent subgroup of $\Gb^F$ of order prime to $p$. 
Fix a sequence $1=Q_0\subset Q_1\subset\cdots\subset Q_r=Q$ of
normal subgroups of $Q$ such that $Q_i/Q_{i-1}$ is cyclic for $1\le i\le r$.
Let $\Gb_i=N_\Gb(Q_1\subset\cdots\subset Q_i)$.

The construction above defines a map
\begin{equation}
\label{eq:mapquotient}
\nabla_{|Q|'}(\Gb_{i+1}/Q_i,F)=\nabla_{|Q|'}(N_{\Gb_i/Q_i}(Q_{i+1}/Q_i),F)\to
\nabla_{|Q|'}(\Gb_i/Q_i,F)
\end{equation}
that preserves rational and geometric series.

Fix $0\le j\le i\le r$.
Let $(\Tb,\theta)\in\nabla_{|Q|'}(\Gb_i,F)$. 
The kernel of the canonical map $\Tb^F\to (\Tb/(\Tb\cap Q_j))^F$ has order a divisor of a power
of the order of $Q_j$, and so does its cokernel, since it is isomorphic to 
$H^1(F,\Tb\cap Q_j)$.
Since $\theta$ is trivial on
$\Tb^F\cap Q_j$, it factors through a character of
$\Tb^F/(\Tb^F\cap Q_j)$ that comes by restriction from a (unique) character
 $\theta'$ of $(\Tb/(\Tb\cap Q_j))^F$.
We obtain a pair $(\Tb/(\Tb\cap Q_j),\theta')\in
\nabla_{|Q|'}(\Gb_i/Q_j,F)$. This correspondence defines a bijection
$\nabla_{|Q|'}(\Gb_i,F)\xrightarrow{\sim}\nabla_{|Q|'}(\Gb_i/Q_j,F)$ that preserves
rational and geometric series.

Composing those bijections with the map in (\ref{eq:mapquotient}), we obtain a map
$$\nabla_{|Q|'}(\Gb_{i+1},F)\to \nabla_{|Q|'}(\Gb_i,F)$$
and composing all those maps, we obtain a map
$$\nabla_{|Q|'}(N_\Gb(Q_1\subset\cdots\subset Q_r),F)\to \nabla_{|Q|'}(\Gb,F).$$
Finally, composing with the canonical map $\nabla_{|Q|'}(C_\Gb(Q),F)\to
\nabla_{|Q|'}(N_\Gb(Q_1\subset\cdots\subset Q_r),F)$, we obtain a map
$$\nabla_{|Q|'}(C_\Gb(Q),F)\to \nabla_{|Q|'}(\Gb,F)$$
that preserves rational and geometric series. Note that this map depends not
only on $Q$, but on the filtration $Q_1\subset\cdots\subset Q_r$.
Summarizing, we have the following proposition.

\begin{prop}
\label{pr:seriesQ}
Let $Q$ be a nilpotent subgroup of $\Gb^F$ of order prime to $p$. 
Fix a sequence $1=Q_0\subset Q_1\subset\cdots\subset Q_r=Q$ of
normal subgroups of $Q$ such that $Q_i/Q_{i-1}$ is cyclic for $1\le i\le r$.

The constructions above define a map
$$i_{Q_\bullet}^\Gb:\nabla_{|Q|'}(C_\Gb(Q),F)/\equiv\ \to\ \nabla_{|Q|'}(\Gb,F)
/\equiv$$

Let $\Lb$ be an $F$-stable Levi complement of a parabolic subgroup $\Pb$ of $\Gb$ containing $Q$. 
Let $\XC \in \nabla_{|Q|'}(C_\Lb(Q),F)/\equiv$ be a rational series.
Then
\begin{itemize}
\item $C_\Lb(Q)$ is an $F$-stable Levi complement of $C_\Pb(Q)$
\item if $i_{Q_\bullet}^\Lb(\XC)$ is $(\Gb^\circ,\Lb^\circ)$-regular
then $\XC$ is $(C_\Gb^\circ(Q),C_\Lb^\circ(Q))$-regular 
\item if $i_{Q_\bullet}^\Lb(\XC)$ is $(\Gb^\circ,\Lb^\circ)$-super regular
then $\XC$ is $(C_\Gb^\circ(Q),C_\Lb^\circ(Q))$-super regular.
\end{itemize}
\end{prop}

The map $i_{Q_\bullet}^\Gb$ is actually independent of the choice
of the filtration of $Q$, cf Remark \ref{re:indepfiltration} below.
\bigskip

\subsection{Generation and series}

\smallskip

Given $(\Tb,\theta)\in\nabla_\L(\Gb,F)$, we denote by
$e^\circ_\theta$ the block idempotent of $\Lambda\Tb^F$ not vanishing on $\theta$.

\medskip
We have now a generalization of \cite[Th\'eor\`eme A]{BR}.

Given $\XC\in\nabla_\Lambda(\Gb,F)/\equiv$,
let $\CC_\XC$ be the thick subcategory of $(\Lambda\Gb^F)\mperf$ generated by the complexes
$\Rrm\Gamma_c(\Yb_\Bb)e^\circ_\theta$ where $(\Tb,\theta)$ runs over $\XC$ and $\Bb$ runs over
Borel subgroups of $\Gb^\circ$ containing $\Tb$.

  Note that, by definition of rational series for non-connected groups, we obtain the same thick subcategory by taking
instead the complexes $\Rrm\Gamma_c(\Yb_\Bb)e_\theta$ where 
$e_\theta=\sum_{t\in N_{\Gb^F}(\Tb,\Bb)/C_{N_{\Gb^F}(\Tb,\Bb)}(\theta)}e^\circ_{\lexp{t}{\theta}}$.

\begin{theo}
\label{th:blocksperf}
Let $\XC\in\nabla_\Lambda(\Gb,F)/\equiv$.
There is a (unique) central idempotent $e_\XC$ of $\Lambda\Gb^F$ such that
$\CC_\XC=(\Lambda\Gb^F e_\XC)\mperf$.

We have a decomposition in central orthogonal idempotents of $\Lambda\Gb^F$
$$1=\sum_{\XC\in\nabla_\L(\Gb,F)/\equiv} e_\XC.$$
\end{theo}

\begin{proof}
Note first that the theorem holds for $\Gb^\circ$ by \cite[Th\'eor\`eme A]{BR}.
Let $(\Tb_i,\theta_i)\in\nabla_\L(\Gb,F)$ and let $\Bb_i$ be a Borel subgroup of $\Gb^\circ$ containing $\Tb_i$
for $i\in\{1,2\}$.
By (\ref{eq:ind-rlg}) and (\ref{eq:res-rlg}), we have
$$\Hom^\bullet_{\Lambda\Gb^F}(\Rrm\Gamma_c(\Yb_{\Bb_1}^\Gb)e^\circ_{\theta_1},
\Rrm\Gamma_c(\Yb_{\Bb_2}^\Gb)e^\circ_{\theta_2})\simeq
\Hom^\bullet_{\Lambda\Gb^{\circ F}}(\Rrm\Gamma_c(\Yb_{\Bb_1}^{\Gb^\circ})e^\circ_{\theta_1},
\bigoplus_{t\in N_{\Gb^F}(\Tb_2,\Bb_2)/\Tb_2^F} \Rrm\Gamma_c(\Yb_{\Bb_2}^{\Gb^\circ})
e^\circ_{\lexp{t}{\theta_2}}).$$
The connected case of the theorem shows this is $0$ unless $(\Tb_1,\theta_1)$
and $(\Tb_2,\lexp{t}{\theta_2})$ are in the same rational series of $(\Gb^\circ,F)$ for
some $t$.

We have shown that the categories $\CC_{\XC_1}$ and $\CC_{\XC_2}$ are othogonal
for $\XC_1{\not=}\XC_2$. The theorem follows now from \cite[Proposition 9.2]{BR} and
Theorem \ref{theo:A-engendrement}.
\end{proof}

\bigskip

Let $\XC\in\nabla_\L(\Gb,F)/\equiv$.
Let $\AC_\XC$ be the thick subcategory of $\Ho^b(\L \Gb^F)$ generated by the complexes 
of the form 
$$\Grm\G_c(\Yb_{\Bb},\L) e_\theta \otimes_{\L Q} L,$$
where 
\begin{itemize}
\item $(\Tb,\theta)$ runs over $\XC$
\item $\Bb$ runs over Borel subgroups of $\Gb^\circ$ containing $\Tb$
\item $Q$ is an $\ell$-subgroup of $N_{\Gb^F}(\Tb,\Bb)$
\item and $L$ is a $\L Q$-module, free of rank $1$ over $\L$. 
\end{itemize}

Let $\BC_\XC$ be the full subcategory of $\L\Gb^Fe_\XC\modules$ consisting of modules whose indecomposable direct 
summands have a one-dimensional source and a vertex $Q$ which normalizes a pair $(\Tb\subset\Bb)$
where $\Tb$ is an $F$-stable maximal torus and $\Bb$ a Borel subgroup of $\Gb$.

\begin{theo}\label{theo:engendrementseries}
Let $\XC\in\nabla_\L(\Gb,F)/\equiv$. We have $\AC_\XC=\Ho^b(\BC_\XC)$.
\end{theo}

\bigskip

\begin{proof}
By Theorem \ref{th:blocksperf}, we have $\Grm\G_c(\Yb_{\Vb},\L) e_\theta\otimes_{\L Q} L
\in\Ho^b(\BC_\XC)$ if $(\Tb,\theta)\in\XC$. It follows that
$\AC_\XC\subset\Ho^b(\BC_\XC)$. Since $\AC=\bigoplus_{\XC\in\nabla_\L(\Gb,F)/\equiv}
\AC_{\XC}$, the theorem follows from Theorem \ref{theo:engendrement}.
\end{proof}

\bigskip

\subsection{Decomposition map and Deligne-Lusztig induction} 

The following result generalizes \cite[Th\'eor\`eme 3.2]{BrMi} to 
non-cyclic $\ell$-subgroups and to disconnected groups (needed to handle
the non-cyclic case by induction).

\begin{theo}\label{theo:brauer}
Let $Q$ be an $\ell$-subgroup of $\Gb^F$. The map
$i_{Q_\bullet}^\Gb$ (cf Proposition \ref{pr:seriesQ})
is independent of the filtration of $Q$ and we
denote it by $i_Q=i_Q^\Gb$.

Let $\XC\in\nabla_{\ell'}(\Gb,F)/\equiv$.
We have
$$\DS{\mathrm{br}_Q(e_\XC) = \sum_{\YC\in i_Q^{-1}(\XC)} e_\YC}.$$
\end{theo}

\bigskip

\begin{proof}
Assume first that $Q$ is cyclic.

Let $\YC\in i_Q^{-1}(\XC)$ and let
$(\Sb,\theta)\in\YC$.
Let $\Bb_Q$ be a Borel subgroup of $C_\Gb(Q)$ containing
$\Sb$. Note that $\Grm\G_c(\Yb_{\Bb_Q},k)e_\theta$ is not acyclic, because its class in $G_0(kC_{\Gb}(Q)^F)$ is
non-zero. We have
$\Grm\G_c(\Yb_{\Bb_Q},k)e_\theta\simeq e_\YC \Grm\G_c(\Yb_{\Bb_Q},k)e_\theta$.
Let $(\Sb^+,\theta^+)=i_Q(\Tb,\theta)\in\XC$ and let $\Bb$ be a $Q$-stable Borel subgroup of $\Gb$ containing $\Sb^+$
(cf Lemma \ref{lem:bijtori}).
We have 
\begin{align*}
\brauer_{\D Q}(\Grm\G_c(\Yb_{\Bb},k)e_{\theta^+})&\simeq
\brauer_{\D Q}(e_{\XC}\Grm\G_c(\Yb_{\Bb},k)e_{\theta^+})\simeq
\mathrm{br}_{\Delta Q}(e_{\XC}\otimes 1)\Grm\G_c(\Yb_{\Bb_Q},k)\mathrm{br}_{\Delta Q}(1\otimes e_{\theta^+})\\
&\simeq\mathrm{br}_Q(e_{\XC})\Grm\G_c(\Yb_{\Bb_Q},k)e_{\theta}\simeq
\mathrm{br}_Q(e_{\XC}) e_\YC\Grm\G_c(\Yb_{\Bb_Q},k)e_{\theta}.
\end{align*}
Similarly,
$$\brauer_{\D Q}(\Grm\G_c(\Yb_{\Bb},k)e_{\theta^+})\simeq
\Grm\G_c(\Yb_{\Bb_Q},k)e_{\theta}{\not=}0.$$
It follows that $\mathrm{br}_Q(e_{\XC}) e_\YC\not=0$.
Since $\sum_{\XC'\in\nabla_{\ell'}(\Gb,F)/\equiv}\mathrm{br}_Q(e_{\XC'})=
1=\sum_{\YC'\in\nabla_{\ell'}(C_{\Gb}(Q),F)/\equiv}e_{\YC'}$, we deduce that
$\mathrm{br}_Q(e_\XC) = \sum_{\YC\in i_Q^{-1}(\XC)} e_\YC$.

By transitivity of $\mathrm{br}_Q$, we obtain the formula for $\mathrm{br}_Q$ for a
general $Q$ by induction on $|Q|$, with $i_Q$ replaced by $i_{Q_\bullet}$. This shows
that actually $i_{Q_\bullet}$ is independent of the chosen filtration of $Q$.
\end{proof}

\begin{rema}
\label{re:indepfiltration}
Let
$Q=Q'\times Q''$ be a product of two cyclic groups of coprime orders. Fix a filtration
$Q_1=Q'$ and $Q_2=Q$. We have $i_{Q_\bullet}=i_Q$. It is easy to deduce now from
Theorem \ref{theo:brauer} that $i_{Q_\bullet}$ is independent of $Q$
for any nilpotent $p'$-group $Q$.
\end{rema}

\bigskip
Brou\'e-Michel's proof of Theorem \ref{theo:brauer} for $\Gb$ connected and $Q$ cyclic relies
on the compatibility of Deligne-Lusztig induction with generalized decomposition maps. This
does generalize to disconnected groups, as we explain below. A direct
approach along the lines of Brou\'e-Michel is possible, based on
the results of \cite{dm-nonc-2}.
While we will not use the results in the remaining part of this section,
they might be useful for character theoretic questions. 

\smallskip
Let $\pi$ be a set of prime numbers not containing $p$. An element of finite order of $\Gb$ is 
a $\pi$-element (resp. a $\pi'$-element) if its order is a product of primes in $\pi$
(resp. not in $\pi$).

Let $g$ be an automorphism of finite order of an algebrac variety $\Xb$. Write 
$g=lx=xl$ where $l$ is a $\pi$-element and $x$ a $\pi'$-element.
The following result is an immediate consequence of~\cite[Theorem~3.2]{DL}:
\equat\label{eq:dl}
\sum_{i \ge 0} (-1)^i \Tr(g, \Hrm_c^i(\Xb,\qlb)) = \sum_{i \ge 0} (-1)^i \Tr(x,\Hrm_c^i(\Xb^l,\qlb)).
\endequat
\begin{proof}
Write $x=su=us$, where $s$ has order prime to $p$ and $u$ has order a power of $p$. Then $l$, $s$ and $u$ 
commute and have coprime orders. By~\cite[Theorem~3.2]{DL}, we have
$$\sum_{i \ge 0} (-1)^i \Tr(g, \Hrm_c^i(\Xb,\qlb)) = \sum_{i \ge 0} (-1)^i \Tr(u,\Hrm_c^i(\Xb^{ls},\qlb))$$
$$\sum_{i \ge 0} (-1)^i \Tr(x, \Hrm_c^i(\Xb^l,\qlb)) = \sum_{i \ge 0} (-1)^i \Tr(u,\Hrm_c^i((\Xb^l)^s,\qlb)).
\leqno{\text{and}}$$
So the result follows from the fact that $\Xb^{ls}=(\Xb^l)^s$ because $\langle ls \rangle = \langle l, s\rangle$.
\end{proof}

\bigskip

Given $H$ a finite group and $h\in H$ a $\pi$-element, we have a generalized decomposition
map from the vector space of class functions $H\to K$ to the vector space of
class functions on $\pi'$-elements of $C_H(h)$ given by
$d_h^H(f)(u)=f(hu)$ for $u$ a $\pi'$-element of $C_H(h)$.

\smallskip
The following result generalizes the character formula for 
$\Rrm_{\Lb \subset \Pb}^\Gb$~\cite[Proposition 2.6]{dm-nonc}, which corresponds
to the case where $\pi$ is the set of all primes distinct from $p$, 

\bigskip

\begin{prop}\label{prop:dec-rlg}
Let $\Pb$ be a parabolic subgroup of $\Gb$, let $\Vb$ be its unipotent radical, let 
$\Lb$ be a Levi complement of $\Pb$ and assume that $\Lb$ is $F$-stable. Let $g \in \Gb^F$ 
be a $\pi$-element. We have
$$d_{g}^{\Gb^F} \circ \Rrm_{\Lb \subset \Pb}^\Gb = 
\sum_{\substack{x \in C_\Gb(g)^F\backslash \Gb^F/\Lb^F \\ g \in \lexp{x}{\Lb}}} 
\Rrm_{C_{\lexp{x}{\Lb}}(g) \subset C_{\lexp{x}{\Pb}}(g)}^{C_\Gb(g)} \circ d_g^{\lexp{x}{\Lb^F}} \circ x_*.
$$
\end{prop}

\bigskip

\begin{proof}
Given $H$ a finite group, we denote by $H_\pi$ (respectively $H_{\pi'}$) the set of $\pi$-elements 
(resp. $\pi'$-elements) of $H$. The proof follows essentially 
the same argument as the proof of the character formula~(see for instance~\cite[Proposition~12.2]{dmbook}).
Let $\l$ be a class function on $\Lb^F$ and let 
$u \in C_\Gb(g)_{\pi'}^F$ be a $\pi'$-element. 
By definition of the Deligne-Lusztig induction and by using (\ref{eq:dl}), we get
\eqna 
\Rrm_{\Lb \subset \Pb}^\Gb(\l)(gu) &=& 
\DS{\frac{1}{|\Lb^F|} \sum_{l \in \Lb_\pi^F} \sum_{v \in C_\Lb(l)_{\pi'}^F} \l(lv) 
\sum_{i \ge 0} (-1)^i \Tr((gu,lv),\Hrm_c^i(\Yb_\Vb,\qlb)).}\\
&=&\DS{\frac{1}{|\Lb^F|} \sum_{l \in \Lb_\pi^F} \sum_{v \in C_\Lb(l)_{\pi'}^F} \l(lv) 
\sum_{i \ge 0} (-1)^i \Tr((u,v),\Hrm_c^i(\Yb_\Vb^{(g,l)},\qlb)).}\\
\endeqna
But it follows from Lemma~\ref{lem:delta} that $\Yb_\Vb^{(g,l)} \neq \vide$ if and only if 
there exists $x \in \Gb^F$ such that $x^{-1}gx=l$. Moreover, in this case, 
then $\Yb_\Vb^{(g,l)} \simeq \Yb_{C_{\lexp{x}{\Vb}}(g)}^{C_\Gb(g)}$ by Proposition~\ref{prop:centralisateur}. 
Therefore, 
\eqna
\Rrm_{\Lb \subset \Pb}^\Gb(\l)(gu) &=&
\DS{\frac{1}{|\Lb^F|\cdot|C_\Gb(g)^F|} \sum_{\substack{x \in \Gb^F \\ g \in \lexp{x}{\Lb}}}  
\sum_{v \in C_\Lb(l)_{\pi'}^F} \l(x^{-1}gxv) 
\sum_{i \ge 0} (-1)^i \Tr((u,v),\Hrm_c^i(\Yb_\Vb^{(g,x^{-1}gx)},\qlb)).}\\
&=&
\DS{\frac{1}{|\Lb^F|\cdot|C_\Gb(g)^F|} \sum_{\substack{x \in \Gb^F \\ g \in \lexp{x}{\Lb}}}  
\sum_{v \in C_{\lexp{x}{\Lb}}(g)_{\pi'}^F} d_g^{\lexp{x}{\Lb}}(x_*(\l))(v) 
\sum_{i \ge 0} (-1)^i \Tr((u,v),\Hrm_c^i(\Yb_{C_{\lexp{x}{\Vb}}(g)}^{C_\Gb(g)},\qlb)).}
\endeqna
Now, if $x \in \Gb^F$ is such that $g \in \lexp{x}{\Lb}$, then 
$$|C_\Gb(g)^F x \Lb^F| = \frac{|C_\Gb(g)^F|\cdot |\Lb^F|}{|C_{\lexp{x}{\Lb}}(g)^F|}.$$
So the result follows.
\end{proof}

\def\closed{{\mathrm{cl}}}
\def\open{{\mathrm{op}}}

\section{Comparing $\Yb$-varieties}
\label{se:comparing}

\medskip

From now on, and until the end of this article, we assume $\Gb$ is connected. 

\smallskip
Deligne-Lusztig varieties can be associated to sequences of elements of $W$, and there is
a canonical isomorphism $\Xb(v,w)\xrightarrow{\sim}\Xb(vw)$ when $l(vw)=l(v)+l(w)$. We will show 
in this section that while such an isomorphism fails when $l(vw){\not=}l(v)+l(w)$, its consequence
on cohomology remains true for local systems associated to characters of tori satisfying
certain regularity conditions with respect to $(v,w)$.

\smallskip
In this section we will prove the preliminary statements necessary for our 
proof of Theorem \ref{thD}. Roughly speaking, the main 
result of this section (Theorem~\ref{theo:theo-d-borel}) is 
almost equivalent to Theorem \ref{thD} whenever $\Lb$ is a maximal torus. 
As Theorem \ref{thD} will be proved by reduction to this case, 
Theorem~\ref{theo:theo-d-borel} may be seen as the crucial step. 

In this section \S\ref{se:comparing},
we fix an $F$-stable maximal 
torus $\Tb$ contained in an $F$-stable Borel subgroup $\Bb$ and we denote by $\Ub$ its
unipotent radical. We put $W=N_\Gb(\Tb)/\Tb$. We denote by $\Phi$ the associated root system,
by $\Phi^+$ the set of positive roots and by $\Delta$ the basis of $\Phi$.
Let $\alpha\in\Phi$, we denote by $s_\alpha\in W$ the corresponding
reflection and by $\alpha^\vee\in\Phi^\vee$ the corresponding coroot.
We put $\Tb_{\alpha^\vee}=\mathrm{Im}(\alpha^\vee)\subset\Tb$ and we denote
by $\Ub_\alpha$ the one-parameter subgroup of $\Gb$
normalized by $\Tb$ and associated with $\alpha$. We define $\Gb_\alpha$
as the subgroup of $\Gb$ generated by $\Ub_\alpha$ and $\Ub_{-\alpha}$.

\bigskip

\subsection{Dimension estimates and further}\label{appendice:dimension}

\medskip

We fix in this section four parabolic subgroups 
$\Pb_1$, $\Pb_2$, $\Pb_3$ and $\Pb_4$ admitting a common Levi complement $\Lb$. We denote by 
$\Vb_1$, $\Vb_2$, $\Vb_3$ and $\Vb_4$ the unipotent radicals of $\Pb_1$, $\Pb_2$, $\Pb_3$ and
$\Pb_4$ respectively.

\bigskip

We define the varieties
$$\YCB_{\! 1,2,3}=\{(g_1\Vb_1,g_2\Vb_2,g_3\Vb_3) \in \Gb/\Vb_1 \times \Gb/\Vb_2 \times \Gb/\Vb_3~|~
g_1^{-1}g_2 \in \Vb_1 \cdot \Vb_2\text{ and }g_2^{-1}g_3 \in \Vb_2 \cdot \Vb_3\},$$
$$\YCB_{\! 1,2,3}^\closed=\{(g_1\Vb_1,g_2\Vb_2,g_3\Vb_3) \in \YCB_{\! 1,2,3}~|~
g_1^{-1}g_3 \in \Vb_1 \cdot \Vb_3\}$$
$$\YCB_{\! 1,3}=\{(g_1\Vb_1,g_3\Vb_3) \in \Gb/\Vb_1 \times \Gb/\Vb_3~|~
g_1^{-1}g_3 \in \Vb_1 \cdot \Vb_3\}.\leqno{\text{and}}$$
We denote by $i_{1,3} : \YCB_{\! 1,2,3}^\closed \injto \YCB_{\! 1,2,3}$ the
closed immersion and we define
$$\fonction{\pi_{1,3}}{\YCB_{\! 1,2,3}^\closed}{\YCB_{\! 1,3}}{(g_1\Vb_1,g_2\Vb_2,g_3\Vb_3)}{
(g_1\Vb_1,g_3\Vb_3).}$$
All these varieties are endowed with a diagonal action of $\Gb$, and the morphisms $i_{1,3}$ and 
$\pi_{1,3}$ are $\Gb$-equivariant.

\begin{prop}\label{prop:dim}
We have:
\begin{itemize}
\itemth{a} $\dim(\Vb_1)=\dim(\Vb_2)=\dim(\Vb_3)$.

\itemth{b} $\dim(\YCB_{\! 1,2,3}) - \dim(\YCB_{\! 1,3}) = \dim(\Vb_1) + \dim(\Vb_1 \cap \Vb_3) 
- \dim(\Vb_1 \cap \Vb_2) - \dim(\Vb_2 \cap \Vb_3)$.

\itemth{c} $\dim(\YCB_{\! 1,2,3}) - \dim(\YCB_{\! 1,3}) = 
2\bigl(\dim(\Vb_1 \cap \Vb_3) - \dim(\Vb_1 \cap \Vb_2 \cap \Vb_3)\bigr)$.
\end{itemize}
\end{prop}

\bigskip

\begin{proof}
(a) is well-known. Also, 
\eqna
\dim(\YCB_{\! 1,2,3})&=&\dim(\Gb/\Vb_1) + \dim(\Vb_1 \cdot \Vb_2/\Vb_2) + \dim(\Vb_2 \cdot \Vb_3 / \Vb_3) \\
&=& \dim(\Gb/\Vb_1) + \dim(\Vb_1)-\dim(\Vb_1 \cap \Vb_2) + \dim(\Vb_2)-\dim(\Vb_2 \cap \Vb_3) 
\endeqna
while
\eqna
\dim(\YCB_{\! 1,3})&=&\dim(\Gb/\Vb_1) + \dim(\Vb_1 \cdot \Vb_3/\Vb_3) \\
&=& \dim(\Gb/\Vb_1) + \dim(\Vb_1)-\dim(\Vb_1 \cap \Vb_3).
\endeqna
So (b) follows from the two equalities (and from (a)).

\medskip

Let us now prove (c). For this, we may assume that $\Tb \subset \Lb$. Let $\Phi_i$ denote the 
set of roots $\a \in \Phi$ such that $\Ub_\a \subset \Vb_i$. Then 
$\Phi_1 \cup -\Phi_1 = \Phi_2 \cup -\Phi_2 = \Phi_3 \cup - \Phi_3=\Phi\setminus\Phi_\Lb$. 
In particular, 
$$\Phi_1 \cup -\Phi_1 = (\Phi_1 \cup \Phi_2 \cup \Phi_3) \cup -(\Phi_1 \cap \Phi_2 \cap \Phi_3).$$
Therefore
$$2|\Phi_1|=|\Phi_1 \cup \Phi_2 \cup \Phi_3| + |\Phi_1 \cap \Phi_2 \cap \Phi_3|.$$
On the other hand, by general facts about the cardinality of a union of finite sets,
$$|\Phi_1 \cup \Phi_2 \cup \Phi_3|=|\Phi_1|+|\Phi_2|+|\Phi_3|
-|\Phi_1 \cap \Phi_2|-|\Phi_1 \cap \Phi_3|-|\Phi_2 \cap \Phi_3|+|\Phi_1 \cap \Phi_2 \cap \Phi_3|.$$
Hence (c) follows from (a), (b) and from these last two equalities.
\end{proof}

\bigskip

Let $d_{1,3}=\dim(\Vb_1 \cap \Vb_3) - \dim(\Vb_1 \cap \Vb_2 \cap \Vb_3)$. By Proposition~\ref{prop:dim}, 
we have
$$d_{1,3}=\frac{1}{2}\bigl(\dim(\YCB_{\! 1,2,3}) - \dim(\YCB_{\! 1,3})\bigr).$$
Let 
$$\fonction{\kappa_{1,3}}{\Gb/(\Vb_1 \cap \Vb_3)}{\YCB_{\! 1,3}}{g(\Vb_1\cap\Vb_3)}{(g\Vb_1,g\Vb_3)}$$
$$\fonction{\kappa_{1,2,3}^\closed}{\Gb/(\Vb_1 \cap \Vb_2 \cap \Vb_3)}{
\YCB_{\! 1,2,3}^\closed}{g(\Vb_1\cap\Vb_2\cap\Vb_3)}{(g\Vb_1,g\Vb_2,g\Vb_3).}\leqno{\text{and}}$$
Both maps are $\Gb$-equivariant morphisms of varieties.

\bigskip

\begin{prop}\label{prop:iso}
The maps $\kappa_{1,3}$ and $\kappa_{1,2,3}^\closed$ are isomorphisms of varieties.
\end{prop}

\bigskip

\begin{proof}
The fact that $\kappa_{1,3}$ is an isomorphism is clear. It is also clear that 
$\kappa_{1,2,3}^\closed$ is a closed immersion. It is so sufficient to prove 
that $\kappa_{1,2,3}^\closed$ is surjective. 

So, let $(g_1\Vb_1,g_2\Vb_2,g_3\Vb_3) \in \YCB_{\! 1,2,3}^\closed$. Using the 
$\Gb$-action and the fact that $\kappa_{1,3}$ is an isomorphism, 
we may assume that $g_1=g_3=1$. Therefore, 
$$g_2 \in (\Vb_1 \cdot \Vb_2) \cap (\Vb_3 \cdot \Vb_2).$$
Given $i\in\{1,3\}$, the multiplication map
$(\Vb_i\cap\Vb_2^-)\times (\Vb_i\cap\Vb_2)\to \Vb_i$
is an isomorphism of varieties, since $\Vb_i$ and $\Vb_2$ have a common Levi complement. Here,
$\Vb_2^-$ denotes the unipotent radical of the parabolic subgroup opposite to $\Pb_2$.
It follows that
$$(\Vb_1 \cdot \Vb_2) \cap (\Vb_3 \cdot \Vb_2)=(\Vb_1\cap\Vb_3\cap\Vb_2^-)\cdot\Vb_2=(\Vb_1 \cap \Vb_3) \cdot \Vb_2.$$
So there exists $h \in \Vb_1 \cap \Vb_3$ such that $h\Vb_2=g_2\Vb_2$. It is then 
clear that $(g_1\Vb_1,g_2\Vb_2,g_3\Vb_3)=\kappa_{1,2,3}^\closed(h)$, as desired.
\end{proof}

\bigskip

\begin{coro}\label{coro:dim-closed}
The map $\pi_{1,3}$ is a smooth morphism with fibers isomorphic to the affine
space of dimension $d_{1,3}$. Moreover, 
$$\dim(\YCB_{\! 1,2,3}) - \dim(\YCB_{\! 1,2,3}^\closed)
=\dim(\YCB_{\! 1,2,3}^\closed) - \dim(\YCB_{\! 1,3})=d_{1,3}.$$
\end{coro}

\bigskip

\begin{proof}
Using the isomorphisms $\kappa_{1,3}$ and $\kappa_{1,2,3}^\closed$ of Proposition~\ref{prop:iso}, 
the map $\pi_{1,3}$ may be identified with the canonical projection 
$\Gb/(\Vb_1 \cap \Vb_2 \cap \Vb_3)\longsurto \Gb/(\Vb_1 \cap \Vb_3)$.
The corollary follows.
\end{proof}

\bigskip

Let us now define
\eqna
\YCB_{\! 1,2,3,4}^\closed&=&\{(g_1\Vb_1,g_2\Vb_2,g_3\Vb_3,g_4\Vb_4) \in \Gb/\Vb_1 \times \Gb/\Vb_2 \times \Gb/\Vb_3 
\times \Gb/\Vb_4~|~\\
&&\hskip2cm 
g_1^{-1}g_2 \in \Vb_1 \cdot \Vb_2,~g_2^{-1}g_3 \in \Vb_2 \cdot \Vb_3,~g_3^{-1}g_4 \in \Vb_3 \cdot \Vb_4 \\
&&\hskip2cm\text{and } g_1^{-1}g_4 \in \Vb_1 \cdot \Vb_4\},
\endeqna
$$\YCB_{\! 1,2,3,4}^{\closed,2}=\{(g_1\Vb_1,g_2\Vb_2,g_3\Vb_3,g_4\Vb_4) \in \YCB_{\! 1,2,3,4}^\closed~|~
g_1^{-1}g_3 \in \Vb_1 \cdot \Vb_3\},$$
$$\YCB_{\! 1,2,3,4}^{\closed,3}=\{(g_1\Vb_1,g_2\Vb_2,g_3\Vb_3,g_4\Vb_4) \in \YCB_{\! 1,2,3,4}^\closed~|~
g_2^{-1}g_4 \in \Vb_2 \cdot \Vb_4\}.\leqno{\text{and}}$$
Then:

\bigskip

\begin{coro}\label{coro:egalite-y}
Assume that at least one of the following holds:
\begin{itemize}
\itemth{1} $\Vb_1 \subset \Vb_4 \cdot \Vb_2$.

\itemth{2} $\Vb_2 \subset \Vb_1 \cdot \Vb_3$.

\itemth{3} $\Vb_3 \subset \Vb_2 \cdot \Vb_4$.

\itemth{4} $\Vb_4 \subset \Vb_3 \cdot \Vb_1$.
\end{itemize}
Then $\YCB_{\! 1,2,3,4}^{\closed,2}=\YCB_{\! 1,2,3,4}^{\closed,3}=\YCB_{\! 1,2,3,4}^{\closed}$. 
\end{coro}

\bigskip

\begin{proof}
Using the fact that the map $(g_1\Vb_1,g_2\Vb_2,g_3\Vb_3,g_4\Vb_4) \mapsto (g_4\Vb_4,g_1\Vb_1,g_2\Vb_2,g_3\Vb_3)$ 
induces an isomorphism $\YCB_{\! 1,2,3,4}^\closed \xrightarrow{\sim} \YCB_{\! 4,1,2,3}^\closed$ (with obvious 
notation), we see that it is sufficient to prove only one of the statements. 

So let us assume that $\Vb_2 \subset \Vb_1 \cdot \Vb_3$. Let  
$(g_1\Vb_1,g_2\Vb_2,g_3\Vb_3,g_4\Vb_4) \in \YCB_{\! 1,2,3,4}^\closed$. Then 
$g_1^{-1}g_3=(g_1^{-1}g_2)(g_2^{-1}g_3) \in \Vb_1 \cdot \Vb_2 \cdot \Vb_3=\Vb_1 \cdot \Vb_3$ 
and so $\YCB_{\! 1,2,3,4}^{\closed,2}=\YCB_{\! 1,2,3,4}^\closed$. So it remains to 
prove that $(g_1\Vb_1,g_2\Vb_2,g_3\Vb_3,g_4\Vb_4) \in \YCB_{\! 1,2,3,4}^{\closed,3}$. 
Using the action of $\Gb$ and the isomorphism $\kappa_{1,3,4}^\closed$ 
of Proposition~\ref{prop:iso}, we may assume that $g_1=g_3=g_4=1$. 
Since $\Vb_2\subset \Vb_1\cdot \Vb_3$, we have $\Vb_1\cap\Vb_3\subset\Vb_2$, hence
$g_2^{-1}g_4=g_2^{-1} \in (\Vb_2\cdot\Vb_3)\cap (\Vb_2\cdot\Vb_1)\subset
\Vb_2 \cdot (\Vb_1\cap\Vb_3)\subset \Vb_2$, as desired.
\end{proof}

\bigskip

\begin{rema}\label{rem:borel-inclusion}
Let $w_1$, $w_2$ and $w_3$ be three elements of $W$ and let us assume here 
that $\Vb_1=\Ub$, $\Vb_2=\lexp{w_1}{\Vb_1}$, $\Vb_3=\lexp{w_1w_2}{\Vb_1}$ and $\Vb_4=\lexp{w_1w_2w_3}{\Vb_1}$. 
Then the conditions (1), (2), (3) and (4) of Corollary~\ref{coro:egalite-y} are respectively equivalent 
to:
\begin{itemize}
\itemth{1} $l(w_2w_3)=l(w_1w_2w_3)+l(w_1)$.

\itemth{2} $l(w_1w_2)=l(w_1)+l(w_2)$.

\itemth{3} $l(w_2w_3)=l(w_2)+l(w_3)$.

\itemth{4} $l(w_1w_2)=l(w_1w_2w_3)+l(w_3)$.
\end{itemize}
\end{rema}

\bigskip

\subsection{Setting}\label{sub:notation}

\medskip
We fix a positive integer $r$. Given a family of objects $m_1,\ldots,m_r$ belonging to
a structure acted on by $F$, we put $m_{j+er}=F^e(m_j)$ for $1\le j\le r$ and $e\ge 0$.

\smallskip
Let $\nb=(n_1,\dots,n_r)$ be a sequence of elements of $N_\Gb(\Tb)$.
We denote by $w_i$ the image of $n_i$ in $W$ and we put $w=w_1\cdots w_r$.

We define
$$\Yb(\nb)=\{(g_1\Ub,g_2\Ub,\dots,g_r\Ub)\in (\Gb/\Ub)^r\ |\
g_j\longtrait{n_j}g_{j+1}\ \ \forall 1\le j\le r\}$$
where $g_j\longtrait{n_j}g_{j+1}$ means $g_j^{-1}g_{j+1}\in\Ub n_j\Ub$.
This variety
has a left action by multiplication of $\Gb^F$ and a right action of
$\Tb^{wF}$ where $t\in\Tb^{wF}$ acts by right multiplication by $(t,t^{n_1},\ldots,t^{n_1\cdots n_{r-1}})$.

We define the functor 
$\RC_\nb=\rgammac(\Yb(\nb),\L)\otimes^{\Lb}_{\Lambda\Tb^{wF}}-: \Drm^b(\Lambda\Tb^{wF})\to \Drm^b(\Lambda\Gb^F)$
and we put $\Rrm_\nb=[\RC_\nb]:G_0(\Lambda\Tb^{wF})\to G_0(\Lambda\Gb^F)$ as in \cite[\S 5.2]{BR}.

\smallskip
We fix a positive integer $m$ such that $F^m(n_i)=n_i$ for all $i$. 
The action of $F^m$ on $(\Gb/\Ub)^r$ restricts to an action on $\Yb(\nb)$.

Given $\Zb$ a variety of pure dimension $n$, we put $\rgammacdim(\Zb,\L)=
\rgammac(\Zb,\L)[n](n/2)$, where $(n/2)$ denotes a Tate twist.

\bigskip


Given $2 \le j \le r$, we denote by $\Yb_j^\closed(\nb)$ the $F^m$-stable closed subvariety 
of $\Yb(\nb)$ defined by
$$\Yb_j^\closed(\nb)=\{(g_1\Ub,g_2\Ub,\dots,g_r\Ub) \in \Yb(\nb)~|~g_{j-1} \longtrait{n_{j-1}n_j} 
g_{j+1}\}$$
and we denote by $\Yb_j^\open(\nb)$ its open complement. They are both stable under the action 
of $\Gb^F \times \Tb^{wF}$. We denote by 
$\pi_j : (\Gb/\Ub)^r \to (\Gb/\Ub)^{r-1}$ the morphism of varieties which forgets the 
$j$-th component and we set 
$$c_j(\nb)=(n_1,n_2,\dots,n_{j-2},n_{j-1}n_{j},n_{j+1},\dots,n_r)$$
$$d_j(\nb)=\DS{\frac{l(w_{j-1})+l(w_{j})-l(w_{j-1} w_{j})}{2}}.\leqno{\text{and}}$$
Let  
$i_{\nb,j} : \Yb_j^\closed(\nb) \injto \Yb(\nb)$ denote the
closed immersion and 
$$\pi_{\nb,j} : \Yb_j^\closed(\nb) \to \Yb(c_j(\nb))$$
denote the restriction of $\pi_j$. Note that 
$\pi_{\nb,j}$ is $(\Gb^F,\Tb^{wF})$-equivariant and commutes with the action of $F^m$. 

\bigskip

\begin{lem}\label{lem:fibration}
If $2 \le j \le r$, then $\pi_{\nb,j}$ is a smooth morphism with fibers
isomorphic to an affine space of dimension $d_j(\nb)$. Moreover, the codimension of 
$\Yb_j^\closed(\nb)$ in $\Yb(\nb)$ is also equal to $d_j(\nb)$.
\end{lem}

\begin{proof}
Let $\Lb=\Tb$, $\Vb_1=\Ub$, $\Vb_2=\lexp{n_{j-1}}{\Ub}$ and
$\Vb_3=\lexp{n_{j-1}n_j}{\Ub}$.
There is a cartesian square
$$\xymatrix{
\Yb_j^\closed(\nb) \ar[d]_{\pi_{\nb,j}} \ar[rrrrrr]^{(g_1\Ub,\ldots,g_r\Ub)\mapsto
(g_{j-1}\Vb_1,g_jn_{j-1}^{-1}\Vb_2,g_{j+1}n_j^{-1}n_{j-1}^{-1}\Vb_3)}
 &&&&&& \YCB_{\! 1,2,3}^\closed\ar[d]^{\pi_{1,3}} \\
\Yb(c_j(\nb)) \ar[rrrrrr]_{(h_1\Ub,\ldots,h_{r-1}\Ub)\mapsto (g_{j-1}\Vb_1,g_{j+1}n_j^{-1}n_{j-1}^{-1}\Vb_3)} &&&&&&
 \YCB_{\! 1,3}
}$$
The lemma follows from now from Corollary~\ref{coro:dim-closed} by base change.
\end{proof}

\bigskip

The map $\pi_{\nb,j}$ induces a quasi-isomorphism of complexes of
$(\L\Gb^F,\L\Tb^{wF})$-bimodules
\equat\label{eq:iso-fibre}
\rgammac(\Yb_j^\closed(\nb),\L) \longisom \rgammac(\Yb(c_j(\nb)),\L)[-2 d_j(\nb)](-d_j(\nb)).
\endequat
Composing this isomorphism with the morphism 
$i_{\nb,j}^* : \rgammac(\Yb(\nb),\L) \to \rgammac(\Yb_j^\closed(\nb),\L)$, 
we obtain a morphism of complexes of $(\L\Gb^F,\L\Tb^{wF})$-bimodules  
$$\Psi_{\nb,j} : \rgammacdim(\Yb(\nb),\L) \longto \rgammacdim(\Yb(c_j(\nb)),\L)$$
which commutes with the action of $F^m$, and whose cone is quasi-isomorphic
to $\rgammacdim(\Yb_j^\open(\nb),\L)[1]$.
\bigskip

\subsection{Preliminaries} 
\label{se:prelim}
We first recall some results from~\cite{BR}.

\bigskip
We denote by $B$ the braid group of $W$, and by $\sigma:W\to B$
the unique map (not a group morphism) that is a right inverse to the 
canonical map $B\to W$ and that preserves lengths. We extend
it to sequences of elements of $W$ by $\sigma(w_1,\ldots,w_r)=
\sigma(w_1)\cdots\sigma(w_r)$.

We denote by $n\mapsto\bar{n}:N_\Gb(\Tb)\to W$ the quotient map.
 We fix $\dot{\sigma}:N_{\Gb}(\Tb)\to B\ltimes
\Tb$ a map (not a group morphism)
such that $\dot{\sigma}(nt)=\dot{\sigma}(n)t$ for all $t\in\Tb$ and
such that the image of $\dot{\sigma}(n)$ in $B=(B\ltimes\Tb)/\Tb$ is
equal to $\sigma(\bar{n})$.  We extend
it to sequences of elements of $N_\Gb(\Tb)$ by $\dot{\sigma}(n_1,\ldots,n_r)=
\dot{\sigma}(n_1)\cdots\dot{\sigma}(n_r)$.

\smallskip
The following result is \cite[Proposition~5.4]{BR}.

\begin{lem}\label{lem:tresse-tore}
Let $\nb'$ be a sequence of elements of $N_\Gb(\Tb)$. Then:
\begin{itemize}
\itemth{a} If $\dot{\s}(\nb)=\dot{\s}(\nb')$ (they are elements of $B \ltimes \Tb$), then the varieties 
$\Yb(\nb)$ and $\Yb(\nb')$ are canonically isomorphic $\Gb^F$-varieties-$\Tb^{wF}$.

\itemth{b} If $\s(\overline{\nb})=\s(\overline{\nb}')$ (they are elements of $B$), then the varieties 
$\Yb(\nb)$ and $\Yb(\nb')$ are (non-canonically) isomorphic
$\Gb^F$-varieties-$\Tb^{wF}$.
\end{itemize}
\end{lem}

\begin{proof}
(a) is proved in~\cite[5.5]{BR}, while (b) is \cite[Proposition~5.4]{BR}.
\end{proof}

\bigskip

Using Lemma~\ref{lem:tresse-tore}(a), we shall now
write $\Yb(\nb) = \Yb(\nb')$ when
$\dot{\s}(\nb)=\dot{\s}(\nb')$. Strictly speaking, 
Lemma~\ref{lem:tresse-tore}(a) says that these two varieties are only isomorphic but, 
since this isomorphism is canonical, we shall use the symbol $=$ to simplify
the exposition.

We define the {\it cyclic shift} $\shift(\nb)$ of $\nb$ by
$$\shift(\nb)=(n_2,\dots,n_r,F(n_1)).$$
The next result is proved in~\cite[Proposition~3.1.6]{DMR} for the varieties 
$\Xb(\wb)$ and $\Xb(\wb')$. The same proof shows the more precise result below.

\medskip

\begin{lem}\label{lem:facile}
The map 
$$\fonctio{\Yb(\nb)}{\Yb(\shift(\nb))}{(g_1\Ub,\dots,g_r\Ub)}{(g_2\Ub,\dots,g_r\Ub,F(g_1)\Ub)}$$
induces an equivalence of \'etale sites. Moreover, it is a morphism of
$\Gb^F$-varieties-$\Tb^{w F}$, where $t\in\Tb^{w F}$ acts on
$\Yb(\shift(\nb))$ by right multiplication by $n_1^{-1}tn_1$.
Consequently, the diagram
$$\diagram
\Drm^b(\Lambda\Tb^{w_1^{-1}wF(w_1)F}) \rrto^{\DS{n_{1,*}}} \ddrto_{\DS{\RC_{\shift(\nb)}}} &&
\Drm^b(\Lambda\Tb^{wF})
\ddlto^{\DS{\RC_\nb}} \\
&& \\ & \Drm^b(\Lambda\Gb^F)& \\
\enddiagram$$
is commutative.
\end{lem}

\bigskip

\bigskip

\def\intersec{{\mathbf{INT}}}
\def\fibre{{\mathbf{FIB}}}

Assume in the remaining part of \S\ref{se:prelim} that $3 \le j \le r$ 
(in particular, $r \ge 3$). Note that
$c_{j-1}(c_j(\nb))=c_{j-1}(c_{j-1}(\nb))$. Consider the diagram 
\equat\label{eq:transitivite}
\diagram
\rgammacdim(\Yb(\nb),\L) \rrto^{\DS{\Psi_{\nb,j}}} \ddto_{\DS{\Psi_{\nb,j-1}}} && 
\rgammacdim(\Yb(c_j(\nb)),\L) \ddto^{\DS{\Psi_{c_j(\nb),j-1}}} \\
&&\\
\rgammacdim(\Yb(c_{j-1}(\nb)),\L) \rrto_{\DS{\Psi_{c_{j-1}(\nb),j-1}}} && 
\rgammacdim(\Yb(c_{j-1}(c_j(\nb)),\L).
\enddiagram
\endequat
It does not seem reasonable to expect that the diagram~(\ref{eq:transitivite}) is commutative 
in general. However, it is in some cases. 

Let us first define the following two varieties:
$$\Yb_{j,j-1}^\closed(\nb)=\Yb_{j-1}^\closed(c_j(\nb)) \times_{\Yb(c_j(\nb))} \Yb_j^\closed(\nb)$$
$$\Yb_{j-1,j}^\closed(\nb)=\Yb_{j-1}^\closed(c_{j-1}(\nb)) \times_{\Yb(c_{j-1}(\nb))} \Yb_{j-1}^\closed(\nb).
\leqno{\text{and}}$$
More concretely, they are the closed subvarieties of $\Yb(\nb)$ defined by
$$\Yb_{j,j-1}^\closed(\nb)=\{(g_1\Ub,\dots,g_r\Ub) \in \Yb(\nb)~|~
g_{j-2} \longtrait{n_{j-2}n_{j-1}n_j} g_{j+1}
\text{ and } g_{j-1} \longtrait{n_{j-1}n_j} g_{j+1}\}$$
$$\Yb_{j-1,j}^\closed(\nb)=\{(g_1\Ub,\dots,g_r\Ub) \in \Yb(\nb)~|~
g_{j-2} \longtrait{n_{j-2}n_{j-1}n_j} g_{j+1}
\text{ and } g_{j-2} \longtrait{n_{j-2}n_{j-1}} g_{j}\}.\leqno{\text{and}}$$

\begin{lem}
\label{le:trans}
If $\Yb_{j,j-1}^\closed(\nb)=\Yb_{j-1,j}^\closed(\nb)$, then 
the diagram~(\ref{eq:transitivite}) is commutative.
\end{lem}

\begin{proof}
There is a commutative diagram, in which all the arrows of the form $\longinjto$ are closed immersions
and all the arrows of the form $\longsurto$ are smooth morphisms with fibers isomorphic 
to an affine space:

\medskip

\refstepcounter{theo}

\begin{centerline}{
$$\diagram
(\arabic{section}.\arabic{theo})\label{diag}&&& \Yb(\nb) &&& \\
\Yb_{j-1}^\closed(\nb) \ar@{^{(}->}[urrr]^{\DS{i_{\nb,j-1}}} \ar@{->>}[ddd]_{\DS{\pi_{\nb,j-1}}} &&&&&&
\Yb_{j}^\closed(\nb) \ar@{_{(}->}[ulll]_{\DS{i_{\nb,j}}} \ar@{->>}[ddd]^{\DS{\pi_{\nb,j}}}\\
&& \Yb_{j-1,j}^\closed(\nb) \ar@{_{(}->}[ull]_{\DS{i}} \ar@{->>}[ddd]_{\DS{\pi}} \ar@{=}[rr]
&& \Yb_{j,j-1}^\closed(\nb) \ar@{^{(}->}[urr]^{\DS{i'}} \ar@{->>}[ddd]^{\DS{\pi'}} &&\\
& \blacksquare && &&\blacksquare&\\
\Yb(c_{j-1}(\nb)) &&&  &&& \Yb(c_j(\nb)) \\
&& \Yb_{j-1}^\closed(c_{j-1}(\nb)) \ar@{_{(}->}[ull]_{\DS{i_{c_{j-1}(\nb),j-1}}} \ar@{->>}[dr]_{\DS{\pi_{c_{j-1}(\nb),j-1}}} 
&& \Yb_{j-1}^\closed(c_{j}(\nb)) \ar@{^{(}->}[urr]^{\DS{i_{c_{j}(\nb),j-1}}} 
\ar@{->>}[dl]^{\DS{\pi_{c_{j}(\nb),j-1}}} &&\\
&&& \Yb(c_{j-1}(c_j(\nb))) &&&
\enddiagram$$}\end{centerline}

\medskip

\noindent 
Note that the two squares marked with the symbol\hskip0.2cm$\blacksquare$\hskip0.2cm are cartesian by definition. 
By the proper base change Theorem, the composition $\Psi_{c_{j-1}(\nb),j-1} \circ \Psi_{\nb,j-1}$ 
is obtained as the composition of $(i_{\nb,j-1} \circ i)^*$ with the inverse of the isomorphism 
induced by $(\pi_{c_{j-1}(\nb),j-1} \circ \pi)^*$. Similarly, 
the composition $\Psi_{c_{j}(\nb),j-1} \circ \Psi_{\nb,j}$ 
is equal to the composition of $(i_{\nb,j} \circ i')^*$ with the inverse of the isomorphism 
induced by $(\pi_{c_{j}(\nb),j-1} \circ \pi')^*$. The lemma follows.
\end{proof}

\bigskip

\begin{lem}\label{lem:transitivite}
Assume that one of the following holds:
\begin{itemize}
\itemth{1} $l(w_{j-2}w_{j-1})=l(w_{j-2})+l(w_{j-1})$.

\itemth{2} $l(w_{j-1} w_{j})=l(w_{j-1})+l(w_{j})$.

\itemth{3} $l(w_{j-2}w_{j-1})=l(w_{j-2}w_{j-1}w_{j}) + l(w_{j})$.

\itemth{4} $l(w_{j-1}w_{j}) = l(w_{j-2}) + l(w_{j-2}w_{j-1}w_{j})$.
\end{itemize}
Then the diagram~(\ref{eq:transitivite}) is commutative.
\end{lem}

\bigskip

\begin{proof}
It is sufficient, by Lemma \ref{le:trans},
 to prove that, if (1), (2), (3) or (4) holds, then 
$\Yb_{j,j-1}^\closed(\nb)=\Yb_{j-1,j}^\closed(\nb)$. This follows, 
after base change, from Corollary~\ref{coro:egalite-y}.
\end{proof}

\subsection{Comparison of complexes}
We start with the description of varieties of the form $\Yb_1^\open(\nb)$
in a very special case, 
which will be the fundamental step in
the proof of Theorem~\ref{theo:theo-d-borel}.

\bigskip
Let $\wb=\bar{\nb}=(w_1,\dots,w_r)$.
Given $\alpha\in\Delta$, we define a subgroup of $\Tb^{r+1}$
\begin{multline*}
\Sb(\alpha,\wb)=\{(a_1,\ldots,a_{r+1})\in\Tb^{r+1}\ | \
a_1^{-1}s_\alpha a_2 s_\alpha^{-1}\in \Tb_{\alpha^\vee},\
a_i^{-1}w_{i-1}a_{i+1}w_{i-1}^{-1}=1 \text{ for }2\le i\le r\\
\text{ and }
a_{r+1}^{-1}w_rF(a_1)w_r^{-1}=1\}.
\end{multline*}
Let $x\in\{1,s_\alpha\}$.
The group morphism
$$\Tb\to\Tb^{r+1},\ a\mapsto (a,x^{-1}ax,w_1^{-1}x^{-1}axw_1,\ldots,
w_{r-1}^{-1}\cdots w_1^{-1}x^{-1}axw_1\cdots w_{r-1})$$
restricts to an embedding of $\Tb^{xwF}$ in $\Sb(\alpha,\wb)$.

\smallskip
Given $\ab=(a_1,\ldots,a_m)$ and $\bb=(b_1,\ldots,b_n)$ two sequences,
we denote the concatenation of the sequences by $\ab\bullet\bb=
(a_1,\ldots,a_m,b_1,\ldots,b_n)$.

\begin{lem}\label{lem:crucial}
Let $\a \in \D$ and let $\sdo$ be a representative of $s_\a$ in 
$N_\Gb(\Tb) \cap \Gb_\a$. We assume that $\Gb_\a \simeq \Sb\Lb_2(\FM)$. 
There exists a closed immersion 
$\Yb(\sdo \bullet \nb) \injto \Yb_2^\open((\sdo,\sdo^{-1}) \bullet \nb)$ 
and an action of $\Sb(\a,\wb)$ 
on $\Yb_2^\open((\sdo,\sdo^{-1}) \bullet \nb)$ such that 
$$\Yb_2^\open((\sdo,\sdo^{-1}) \bullet \nb) \simeq 
\Yb(\sdo \bullet \nb) \times_{\Tb^{s_\a wF}} \Sb(\a,\wb),$$
as $\Gb^F$-varieties-$\Tb^{wF}$.
\end{lem}

%

\bigskip

\begin{proof}
Given $i\in\{1,\ldots,r\}$, consider a reduced decomposition
$w_i=s_{i,1}\cdots s_{i,d_i}$. We put
$\tilde{\wb}=(s_{1,1},\ldots,s_{1,d_1},s_{2,1},\ldots,s_{2,d_2},\ldots,
s_{r,1},\ldots,s_{r,d_r})$.
Note that $\Sb(\alpha,\wb)$ is isomorphic to the group
$\Sb(s_\alpha\bullet\tilde{\wb},1\bullet\tilde{\wb})$ 
defined in \cite[\S 4.4.3]{BR}:
$$\Sb(s_\alpha\bullet\tilde{\wb},1\bullet\tilde{\wb})\xrightarrow{\sim}
\Sb(\alpha,\wb),\
(a_1,\ldots,a_{1+d_1+\cdots+d_r})\mapsto
(a_1,a_2,a_{2+d_1},a_{2+d_1+d_2},\ldots,a_{2+d_1+\cdots+d_{r-1}}).$$

The following computation in $\Sb\Lb_2(\FM) \simeq \Gb_\a$
$$\begin{pmatrix} 1 & x \\ 0 & 1 \end{pmatrix} \begin{pmatrix} 0 & 1 \\ -1 & 0\end{pmatrix} \begin{pmatrix} 1 & y \\ 0 & 1\end{pmatrix}
\begin{pmatrix} 0 & -1 \\ 1 & 0\end{pmatrix} \begin{pmatrix} 1 & z \\ 0 & 1\end{pmatrix}
=\begin{pmatrix} 1-xy && x+z-xyz \\ -y && 1-yz\end{pmatrix}\leqno{(\#)}$$
shows that the map
$$\fonctio{\Ub_\a \times (\Ub_\a\setminus\{1\}) \times \Ub_\a}{\Ub_\a\Tb_{\a^\ve}\sdo \Ub_\a=\Gb_\a\setminus \Bb\cap\Gb_\a}
{(u_1,u_2,u_3)}{
u_1\sdo u_2 \sdo^{-1} u_3}$$
is an isomorphism of varieties.
Let $\Ub^\a=\Ub\cap \lexp{\sdo}{\Ub}$.
Let $(g_1\Ub,\ldots,g_{r+2}\Ub)\in \Yb((\sdo,\sdo^{-1}) \bullet \nb)$. We have
$(g_1\Ub,\ldots,g_{r+2}\Ub)\in \Yb_2^\open((\sdo,\sdo^{-1}) \bullet \nb)$ if and only if
$$g_1^{-1}g_3\in (\Ub\sdo \Ub \sdo^{-1}\Ub)\setminus \Ub=\Ub^\a\cdot (\Ub_\a \sdo \Ub_a \sdo^{-1}\Ub_\a\setminus
\Ub_\a)=\Ub^\a\cdot (\Ub_\a\Tb_{\a^\ve}\sdo \Ub_\a)=\Ub\Tb_{\a^\ve}\sdo \Ub.$$
Furthermore, if $(g_1\Ub,\ldots,g_{r+2}\Ub)\in \Yb_2^\open((\sdo,\sdo^{-1}) \bullet \nb)$, then $g_2\Ub$ is determined
by $g_1\Ub$ and $g_3\Ub$.

 Therefore, one may forget the second coordinate in the definition 
of the variety $\Yb_2^\open((\sdo,\sdo^{-1}) \bullet \nb)$ and we get
\equat\label{eq:iso-epaississement}
\begin{array}{rcl}
\Yb_2^\open((\sdo,\sdo^{-1}) \bullet \nb)&\simeq &\{(g\Ub,g_1\Ub,\dots,g_r\Ub)~|~g^{-1}g_1 \in \Ub\Tb_{\a^\ve}\sdo \Ub
\quad \mathrm{and}~\\
&&\quad\vphantom{\frac{\DS{A}}{\DS{A}}}
\qquad g_1 \longtrait{n_1} g_2 \longtrait{n_2} \cdots \longtrait{n_{r-1}} g_r
 \longtrait{n_r} F(g) \}.
\end{array}
\endequat

This description shows that the group $\Sb(\a,\wb)$ acts 
on $\Yb_2^\open((\sdo,\sdo^{-1}) \bullet \nb)$ (as the restriction of
the action by right multiplication of $\Tb^{r+1}$ on $(\Gb/\Ub)^{r+1}$).
Also, as $\Ub\sdo\Ub$ is closed in 
$\Ub\Tb_{\a^\ve}\sdo\Ub$, the natural map 
$\Yb(\sdo \bullet \nb) \injto \Yb_2^\open((\sdo,\sdo^{-1}) \bullet \nb)$ is a
closed immersion. We have embeddings
$\Tb^{s_\a w F} \injto \Sb(\a,\wb)$ and $\Tb^{w F} \injto \Sb(\a,\wb)$ 
and
$$\Sb(\a,\wb)=\Tb^{s_\a w F} \cdot \Sb(\a,\wb)^\circ = 
\Sb(\a,\wb)^{\circ}\cdot \Tb^{w F}$$
(see~\cite[Proposition~4.11]{BR}). 
The stabilizer of the closed subvariety $\Yb(\sdo \bullet \nb)$ under this action 
is $\Tb^{s_\a w F}$, so it is readily checked that the componentwise multiplication induces an isomorphism 
of $\Gb^F$-varieties-$\Tb^{w F}$
$$\Yb(\sdo \bullet \nb) \times_{\Tb^{s_\a w F}} \Sb(\a,\wb) \longisom 
\Yb_2^\open((\sdo,\sdo^{-1}) \bullet \nb),$$
as desired.
\end{proof}

\bigskip
The next theorem is the main result of this section.
It provides a sufficient condition for $\Psi_{\nb,j}$
to induce a quasi-isomorphism 
$\rgammacdim(\Yb(\nb),\L)e_\th \longisom \rgammacdim(\Yb(c_j(\nb)),\L)e_\th$.

\medskip
Given $x,y \in W$, we put
$$\Phi^+(x,y)=\{\a \in \Phi^+~|~x^{-1}(\a) \in - \Phi^+\text{~and~}(xy)^{-1}(\a) \in \Phi^+\}.$$

We define $N_w:Y(\Tb)\to\Tb^{wF},\ \lambda\mapsto N_{F^d/wF}(\lambda(\zeta))$
(cf \S \ref{se:rationalseries}).

\bigskip

\begin{theo}\label{theo:theo-d-borel}
Let $\th : \Tb^{wF} \to \L^\times$ be a character.
Let $j \in \{2,3,\dots,r\}$ and assume that $\th(N_w(w_1\cdots w_{j-2}(\a^\ve))) \neq 1$ 
for all $\a \in \Phi^+(w_{j-1},w_{j})$. We have
$\rgammac(\Yb^\open_j(\nb),\L)e_\th=0$ and
$$\Psi_{\nb,j,\th} : \rgammacdim(\Yb(\nb),\L)e_\th \longisom \rgammacdim(\Yb(c_j(\nb)),\L)e_\th$$
is a quasi-isomorphism of complexes of $(\L\Gb^F,\L\Tb^{wF})$-bimodules commuting
with the action of $F^m$.
\end{theo}

\bigskip

%
%

\begin{proof}
If $2 \le j \le r$, we denote by $\PC(\nb,j,\th)$ the following property:
$$\text{For all $\a \in \Phi^+(w_{j-1},w_j)$, 
we have $\th(N_w(w_1\cdots w_{j-2}(\a^\ve))) \neq 1$.}\leqno{\PC(\nb,j,\th)}$$
We want to prove that $\PC(\nb,j,\th)$ implies that $\rgammac(\Yb_j^\open(\nb),\L)e_\th = 0$. 
By~\cite[Proposition~5.19~and~Remark~5.21]{BR}, it is sufficient to prove it 
whenever $[\Gb,\Gb]$ is simply connected, and we will assume this holds.

\medskip

So assume from now that $\PC(\nb,j,\th)$ holds. 
We will prove by induction on $l(w_{j-1})$ that $\rgammac(\Yb_j^\open(\nb),\L)e_\th = 0$. 
Note that the induction hypothesis does not depend on $r$. 
But, first, note that, if $j \ge 3$, then 
$\PC(\nb,j,\th)$ is equivalent to $\PC(\shift(\nb),j-1,\th \circ n_1)$ 
and that the morphism constructed in Lemma~\ref{lem:facile} sends $\Yb_j^\open(\nb)$ to 
$\Yb_{j-1}^\open(\shift(\nb))$. Thus 
$\rgammac(\Yb_j^\open(\nb),\L)e_\th = 0$ is equivalent to 
$\rgammac(\Yb_{j-1}^\open(\shift(\nb)),\L)e_{\th \circ n_1} = 0$. 
By successive applications of this remark, this shows that we may assume that $j=2$. 

\bigskip

\noindent{\it First case: Assume that $l(w_1)=0$.} This means that $n_1 \in \Tb$ and 
it follows from Lemma~\ref{lem:fibration} (or Lemma~\ref{lem:tresse-tore}(a)) 
that $\Yb_2^\open(\nb)=\vide$. So the result 
follows in this case.

\bigskip

\noindent{\it Second case: Assume that $l(w_1)=1$ and $n_1n_2=1$.} 
Let $\a \in \D$ be such that $w_1=s_\a$. By Lemma \ref{lem:tresse-tore}, we may assume that 
$n_1=\sdo$ is a representative of $s_\a$ lying in $\Gb_\a$. Note 
also that, since $[\Gb,\Gb]$ is simply connected, we have $\Gb_\a \simeq \Sb\Lb_2(\FM)$. 
Define $\Sb=\Sb(\a,(w_3,\dots,w_r))$.
Lemma~\ref{lem:crucial} shows that
$$\rgammac(\Yb_2^\open(\nb),\L)e_\th = \rgammac(\Yb(\sdo,n_3,\dots,n_r),\L) 
\otimes_{\L\Tb^{s_\a wF}} \rgammac(\Sb,\L) e_\th.$$
But $\Phi^+(w_1,w_2)=\Phi^+(s_\a,s_\a)=\{\a\}$, so $\th(N_w(\a^\ve)) \neq 1$ 
by hypothesis. Note also that $\Tb^{wF} \cap \Sb^\circ$ acts trivially on the cohomology 
groups of the complex $\rgammac(\Sb)$, as its action extends to the 
connected group $\Sb^\circ$. Since $N_w(\a^\ve) \in \Sb^\circ$ 
(see~\cite[Proof of Proposition~4.11,~Equality~(a)]{BR}), this proves that 
$\rgammac(\Sb,\L) e_\th=0$ and so $\rgammac(\Yb_2^\open(\nb),\L)e_\th=0$, 
as desired.

\bigskip

\noindent{\it Last case: Assume that $l(w_1) \ge 1$.} 
Let $\a \in \D$ be such that $w_1=s_\a w_1'$, 
with $l(w_1')=l(w_1)-1$. Let $\sdo$ be a representative of $s_\a$ in $\Gb_\a$ and let 
$n_1'=\sdo^{-1} n_1$. We will write $\nb'=(n_1',n_2,\dots,n_r)$. 
Then $n_1'$ is a representative of $w_1'$ and, 
by Lemma~\ref{lem:tresse-tore}(a), we have $\Yb(\nb) = \Yb(\sdo \bullet \nb')$ 
(see also the remark following Lemma~\ref{lem:tresse-tore}). 

It is well-known that $\Phi^+(w_1,w_1^{-1})=\Phi^+\cap w_1(-\Phi^+)=\{\a\} \coprod s_\a(\Phi^+(w_1',w_1^{-1}))$. Therefore
\begin{align*}
\Phi^+(w_1,w_2)&=\Phi^+\cap w_1(-\Phi^+)\cap w_1w_2(\Phi^+)=
\bigl(\{\alpha\}\cap w_1w_2(\Phi^+)\bigr)\coprod s_\a\bigr(\Phi^+\cap w'_1(-\Phi^+)\cap w'_1w_2(\Phi^+)\bigr)\\
&=\bigl(\{\alpha\}\cap w_1w_2(\Phi^+)\bigr)\coprod s_\a(\Phi^+(w'_1,w_2)),
\end{align*}
hence
$$\Phi^+(w_1,w_2)=
\begin{cases}
s_\a\bigl(\Phi^+(w_1',w_2)\bigr) & \text{if $l(w_1'w_2) < l(w_1w_2)$,}\\
\{\a\} \coprod s_\a\bigl(\Phi^+(w_1',w_2)\bigr) & \text{if $l(w_1'w_2) > l(w_1w_2)$.}\\
\end{cases}
\leqno{(\#)}
$$
Let us now consider the diagram~(\ref{diag}) with $\nb$ replaced by $\sdo \bullet \nb'$ 
and $j$ is replaced by $3$. Since $c_2(\sdo\bullet\nb')=\nb$, it follows from 
Lemma~\ref{lem:transitivite}(1) that (\ref{eq:transitivite}) gives a commutative diagram
$$\xymatrix{\rgammacdim(\Yb(\sdo\bullet\nb'))e_\theta\ar[r]\ar[d] &  
\rgammacdim(\Yb(c_3(\sdo\bullet\nb')))e_\theta\ar[d] \\
\rgammacdim(\Yb(\nb))e_\theta\ar[r] &  
\rgammacdim(\Yb(c_2(\nb)))e_\theta
}$$
The left vertical map is an isomorphism since $\Yb(\sdo\bullet\nb')\simeq\Yb(\nb)$.
By $(\#)$, we have $\Phi^+(w'_1,w_2)\subset s_\a(\Phi^+(w_1,w_2))$, hence
Property $\PC(\sdo \bullet \nb',3,\th)$ is fulfilled. So, the top horizontal map is an isomorphism by induction.

In order to show that the bottom horizontal map is an isomorphism, 
it remains to show that the right vertical map is an isomorphism.
Note that $c_3(\sdo \bullet \nb')=\sdo \bullet c_2(\nb')$.  Two cases may occur:

\medskip

$\bullet$ Assume first that $l(w_1'w_2) < l(w_1 w_2)$. Then 
$\Yb_2^\open(c_3(\sdo \bullet \nb'))=\vide$, and the result follows. 

\medskip

$\bullet$ Assume now that $l(w_1' w_2) > l(w_1 w_2)$. 
Then, again by Lemma~\ref{lem:tresse-tore}(a), 
we have $\Yb(\sdo \bullet c_2(\nb')) = \Yb((\sdo,\sdo^{-1}) \bullet c_2(\nb))$ 
and, through this identification, $\Yb_2^\open(c_3(\sdo \bullet \nb'))$ 
is identified with $\Yb_2^\open((\sdo,\sdo^{-1}) \bullet c_2(\nb))$. 
So the result now follows from the {\it second case} (thanks to $(\#)$). 
\end{proof}

\begin{rema}
\label{re:Frob}
Theorem \ref{theo:theo-d-borel} provides a comparison of modules, together with the Frobenius
action. Consider the case $\Lambda=K$. We have an isomorphism of $K\Gb^F$-modules,
compatible with the Frobenius action
$$H^i_c(\Yb(\nb),K)\otimes_{K\Tb^{wF}}K_\theta\simeq
H^{i-2r}_c(\Yb(c_j(\nb)),K)\otimes_{K\Tb^{wF}}K_\theta(-r).$$
where $r=d_j(\nb)$.
\end{rema}

\bigskip
Following the same lines as in the proof of Theorem~\ref{theo:theo-d-borel}, we
obtain a new proof of
the following classical result.

\bigskip

\begin{theo}\label{theoe}
If $\L$ is a field, then $\Rrm_\nb=\Rrm_w$. 
\end{theo}

\bigskip

\begin{proof}
By~\cite[Proposition~5.19~and~Remark~5.21]{BR}, it is sufficient to prove the 
Theorem whenever $[\Gb,\Gb]$ is simply connected, and we assume this holds. Also, by proceeding 
step-by-step, it is enough to prove that $\Rrm_\nb=\Rrm_{c_j(\nb)}$. For this, let 
$\rgamma_{\! \nb,j}^\open$ denote the class of the complex 
$\rgammac(\Yb_j^\open(\nb),\L)$ in $G_0(\L\Gb^F\otimes \L\Tb^{wF})$. 
We only need to prove that $\rgamma_{\! \nb,j}^\open=0$.

Proceeding by induction on $l(w_j)$ as in the proof of Theorem~\ref{theo:theo-d-borel}, 
and following the same strategy and arguments, we see that it is enough to prove 
Theorem~\ref{theoe} whenever $j=1$, $n_1=\sdo=n_2^{-1}$, where $\sdo$ is a representative 
in $\Gb_\a$ of $s_\a$ (for some $\a \in \D$). By Lemma~\ref{lem:crucial}, it is sufficient to 
prove that the class $\rgamma_{\! \a,\wb}$ of the complex
$\rgammac(\Sb(\a,\wb),\L)$ 
in $G_0(\L\Tb^{s_\a w F} \otimes \L\Tb^{wF})$ is equal to $0$. 

Now, let $T$ denote the subgroup of $\Tb^{s_\a w F} \times \Tb^{wF}$ consisting 
of pairs $(t_1,t_2)$ such that $t_1t_2 \in \Sb(\a,\wb)^\circ$ 
and let $\rgamma^\circ$ denote the class of the complex \
$\rgammac(\Sb(\a,\wb)^\circ)$ 
in $G_0(\L T)$. Then $\rgamma_{\! \a,\wb}=\Ind_T^{\Tb^{s_\a w F} \times \Tb^{wF}} \rgamma^\circ$. 
But the action of $T$ on $\Sb(\a,\wb)$ extends to an action 
of the connected group $\Sb(\a,\wb)^\circ$, hence $T$ acts trivially on the cohomology 
groups of $\Sb(\a,\wb)^\circ$. Since the Euler characteristic 
of a torus is equal to $0$, this gives $\rgamma^\circ=0$, and
consequently $\rgamma_{\! \a,\wb}=0$, as desired.
\end{proof}

\bigskip

\begin{coro}\label{coro:theoe-weyl}
Let $\nb'=(n_1',n_2',\dots,n_{r'}')$ be a sequence of elements of $N_\Gb(\Tb)$,  
let $x \in W$ and let $w'$ denote the image of $n_1'n_2'\cdots n_{r'}'$ in $W$. 
We assume that $\L$ is a field and that $w'=x^{-1}wF(x)$. Then 
the diagram
$$\diagram
G_0(\Lambda\Tb^{\nb'F}) \rrto^{\DS{x_*}} \ddrto_{\DS{\Rrm_{\nb'}}} && 
G_0(\Lambda\Tb^{\nb F})
\ddlto^{\DS{\Rrm_\nb}} \\ && \\ & G_0(\Lambda\Gb^F)& \\
\enddiagram
$$
is commutative.
\end{coro}

\bigskip

\begin{proof}
Let $\nb''=(\xdo^{-1},n_1,n_2,\dots,n_r,F(\xdo))$. Then, by Lemma~\ref{lem:facile}, 
$$\Rrm_{\nb''}=\Rrm_{\nb \bullet (F(\xdo),F(\xdo)^{-1})} \circ x_*.$$ 
But, by Theorem~\ref{theoe}, 
$\Rrm_{\nb''}=\Rrm_{w'}=\Rrm_{\nb'}$ and 
$\Rrm_{\nb \bullet (F(\xdo),F(\xdo)^{-1})}=\Rrm_w=\Rrm_\nb$.
\end{proof}

\bigskip

The following result is a reformulation of Corollary~\ref{coro:theoe-weyl} as in \cite[\S{11.1}]{BR}.

\begin{coro}\label{coro:theoe}
Let $\Tb'$ be an $F$-stable maximal torus of $\Gb$ and let $\Bb'$ and $\Bb''$ 
be two Borel subgroups of $\Gb$ containing $\Tb'$. Then 
$\Rrm_{\Tb' \subset \Bb'}^\Gb(\th') = \Rrm_{\Tb' \subset \Bb''}^\Gb(\th')$ 
for all $\th' \in \Irr(\Tb^{\prime F})$.
\end{coro}

\bigskip
\begin{rema}
Corollary \ref{coro:theoe} is well-known. In~\cite[Corollary~4.3]{DL}, 
this result is first proved ``geometrically'' for $\th'=1$~\cite[Theorem~1.6]{DL} by relating 
the varieties $\Xb_{\Bb'}^\Gb$ and $\Xb_{\Bb''}^\Gb$, 
and extended to the general case using the {\it character formula}~\cite[Theorem~4.2]{DL}. 
Note that this result is then used in~\cite[Theorem~6.8]{DL} to deduce the {\it Mackey formula} 
for Deligne-Lusztig induction functors.

In~\cite{L}, Lusztig proposed another argument: the Mackey formula is proved ``geometrically'' 
and {\it a priori}~\cite[Theorem~2.3]{L}, and Corollary
\ref{coro:theoe} follows \cite[Corollary~2.4]{L}.

Our argument relies neither on the Mackey formula nor on the character
formula:
we 
lift Deligne-Lusztig's comparison of $\Xb_{\Bb'}^\Gb$ and $\Xb_{\Bb''}^\Gb$ to a relation between 
the varieties $\Yb_{\Ub'}^\Gb$ and $\Yb_{\Ub''}^\Gb$ (here $\Ub'$ and $\Ub''$ are the unipotent radicals 
of $\Bb'$ and $\Bb''$ respectively).
\end{rema}

\begin{rema}
Some of the results in~\cite{BR2} (Lemma~4.3, Proposition~4.5 and Theorem~4.6) 
rely on a disjointness result used in~\cite[Line~16~of~Page~30]{BR2}. This disjointness result 
was ``proved'' using the isomorphism in~\cite[Line~18~of~Page~30]{BR2}:  
it has been pointed out to the attention of the authors by H.~Wang that 
this equality is false. However, Wang provided a complete proof of this disjointness 
result~\cite[Proposition~3.4.3]{W}, so~\cite[Lemma~4.3,~Proposition~4.5~and~Theorem~4.6]{BR2} remain valid. 

Another proof of this disjointness result has been obtained independently 
by Nguyen~\cite{N} (with slightly different methods).
Using a version of Remark \ref{re:Frob}, Wang and Nguyen have been able to
keep track of the Frobenius eigenvalues.
\end{rema}

%
%
%
%
%
%
%
%

\section{Independence with respect to the parabolic subgroup}\label{sec:proof}

\medskip
We assume in this section \S\ref{sec:proof} that $\Gb$ is connected.
We fix an $F$-stable maximal torus
$\Tb$ of $\Gb$ and we denote by $(\Gb^*,\Tb^*,F^*)$ a triple dual to $(\Gb,\Tb,F)$. 

\smallskip

We fix a family of parabolic subgroups 
$\Pb_1$, $\Pb_2$,\dots, $\Pb_r$ admitting $\Lb$ as a Levi complement. 

The identification of the 
root system of $\Gb$ with the coroot system of $\Gb^*$ allows us to define 
parabolic subgroups $\Pb_1^*$, $\Pb_2^*$,\dots, $\Pb_r^*$, admitting a common 
$F^*$-stable Levi complement $\Lb^*$ and such that $\Lb^*$ and $\Pb_j^*$ and 
are dual to $\Lb$ and $\Pb_j$ respectively. We denote by $\Vb_j$ and $\Vb_j^*$ 
the unipotent radicals of $\Pb_j$ and $\Pb_j^*$ respectively. 
We denote by $\Vb_\bullet$ the sequence $(\Vb_1,\dots,\Vb_r)$.

Finally, we fix a semisimple element $s \in \Lb^{*F^*}$ 
whose order is invertible in $\L$.

\bigskip

\subsection{Isomorphisms} As announced in the introduction, 
the isomorphism of functors described in Theorem \ref{thD} is {\it canonical}. 
So, before giving the proof, we will explain how it is realized. For this, let us define 
$$\Yb_{\Vb_\bullet}=\{(g_1\Vb_1,\dots,g_r\Vb_r) 
\in \Gb/\Vb_1  \times \cdots \times \Gb/\Vb_r~|~
\forall~j \in \{1,2,\dots,r\},~g_j^{-1} g_{j+1} \in \Vb_j \cdot \Vb_{j+1}\}$$
where $\Vb_{r+1}=F(\Vb_1)$ and $g_{r+1}=F(g_1)$. Given  $2 \le j \le r$, we set
$$\Yb_{\Vb_\bullet,j}^\closed=\{(g_1\Vb_1,g_2\Vb_2,\dots,g_r\Vb_r) \in \Yb_{\Vb_\bullet}~|~
g_{j-1}^{-1}g_{j+1} \in \Vb_{j-1}\cdot \Vb_{j+1}\}.$$
It is a closed subvariety of $\Yb_{\Vb_\bullet}$ and we denote by 
$i_{\Vb_\bullet,j} : \Yb_{\Vb_\bullet,j}^\closed \injto \Yb_{\Vb_\bullet}$ 
the closed immersion. Let $\Yb_{\Vb_\bullet,j}^\open$ denote 
its open complement. 
We define the sequence $c_j(\Vb_\bullet)$ as obtained from the sequence $\Vb_\bullet$ 
by removing the $j$-th component. We then define 
$$\pi_{\Vb_\bullet,j} : \Yb_{\Vb_\bullet,j}^\closed 
\longto \Yb_{c_j(\Vb_\bullet)}$$
as the map which forgets the $j$-th component and we set 
$$d_j(\Vb_\bullet)=\dim(\Vb_{j-1} \cap \Vb_{j+1})-\dim(\Vb_{j-1} \cap \Vb_j \cap \Vb_{j+1}).$$
Note that $\Gb^F$ acts diagonally on $\Yb_{\Vb_\bullet}$ by left translation, 
that $\Lb^F$ acts diagonally by right translation, and that this 
endows $\Yb_{\Vb_\bullet}$ with a structure of $\Gb^F$-variety-$\Lb^F$. 
The varieties $\Yb_{\Vb_\bullet}^\closed$ and $\Yb_{\Vb_\bullet}^\open$ 
are stable under these actions, and the morphisms $i_{\Vb_\bullet,j}$ and 
$\pi_{\Vb_\bullet,j}$ are equivariant. 
As for their analogues $i_{\nb,j}$ and $\pi_{\nb,j}$ defined in \S\ref{sub:notation}, we have the following 
properties, which follow from Corollary~\ref{coro:dim-closed} 
by base change.

\bigskip

\begin{lem}\label{lem:fibration-p}
The map $\pi_{\Vb_\bullet,j}$ is smooth with fibers isomorphic to an 
affine space of dimension $d_j(\Vb_\bullet)$. The codimension of $\Yb_{\Vb_\bullet}^\closed$ 
in $\Yb_{\Vb_\bullet}$ is also equal to $d_j(\Vb_\bullet)$. 
\end{lem}

\bigskip

We deduce that $\pi_{\Vb_\bullet,j}$ induces a quasi-isomorphism of complexes of $(\L\Gb^F,\L\Lb^F)$-bimodules 
$$\rgammac(\Yb_{\Vb_\bullet,j}^\closed,\L) \simeq \rgammac(\Yb_{c_j(\Vb_\bullet)},\L)
[-2d_j(\Vb_\bullet)](-d_j(\Vb_\bullet)).$$
The closed immersion $i_{\Vb_\bullet,j} : \Yb_{\Vb_\bullet,j}^\closed \injto \Yb_{\Vb_\bullet}$ 
induces a morphism of complexes of $(\L\Gb^F,\L\Lb^F)$-bimodules 
$$i_{\Vb_\bullet,j}^* : \rgammac(\Yb_{\Vb_\bullet},\L) \longto \rgammac(\Yb_{\Vb_\bullet,j}^\closed,\L)$$
which, composed with the previous isomorphism, induces a morphism 
$$\Psi_{\Vb_\bullet,j} : \rgammacdim(\Yb_{\Vb_\bullet},\L) \longto \rgammacdim(\Yb_{c_j(\Vb_\bullet)},\L).$$
The main result of this section is the following theorem. We put $e_s^{\Lb^F}=e_\YC$, where 
$\YC\in\nabla_\Lambda(\Lb,F)/\equiv$ is the rational series corresponding to the $\Lb^{*F^*}$-conjugacy class of
$s$.

\bigskip

\begin{theo}\label{theo:theod}
Let $j \in \{2,3,\dots,r\}$ such that 
$C_{\Vb_{j-1}^* \cap \Vb_{j+1}^*}(s) \subset C_{\Vb_j^*}(s)$. 
We have $$\rgammac(\Yb_{\Vb_\bullet,j}^\open,\L)e_s^{\Lb^F}=0,$$ 
hence
$\Psi_{\Vb_\bullet,j}$ induces a quasi-isomorphism of complexes of $(\L\Gb^F,\L\Lb^F)$-bimodules 
$$\Psi_{\Vb_\bullet,j,s} : 
\rgammacdim(\Yb_{\Vb_\bullet,j},\L) e_s^{\Lb^F} \longisom \rgammacdim(\Yb_{c_j(\Vb_\bullet)},\L)e_s^{\Lb^F}.$$
\end{theo}

\bigskip

%
%
%

\begin{proof}
The proof will proceed in two steps. We first prove the theorem when $\Lb$ is a maximal torus: 
in fact, it will be shown that it is a consequence of Theorem~\ref{theo:theo-d-borel}. 
We then use~\cite[Theorem~A']{BR} to deduce the general case from this particular one. 

\bigskip

\noindent{\it First step: Assume here that $\Lb$ is a maximal torus.} 
Let $a_1$,\dots, $a_r$ be elements of $\Gb$ such that $(\Lb,\Pb_i)=\lexp{a_i}{(\Tb,\Bb)}$ 
for all $i \in \{1,2,\dots,r\}$). As usual, we set $a_{r+1}=F(a_r)$. 
Now, let $n_i=a_i^{-1} a_{i+1}$. It follows from the definition of the $a_i$'s that 
$n_i \in N_\Gb(\Tb)$. We set $\nb=(n_1,\dots,n_r)$. Note that $n_1n_2\cdots n_r=a_1^{-1}F(a_1)$. 
We denote by $w_i$ the image of $n_i$ in $W$ and we set $w=w_1w_2\cdots w_r$. 
It is then easily checked that the map
$$(g_1\Vb_1,\dots,g_r\Vb_r) \longmapsto (g_1\Vb_1a_1,\dots, g_r \Vb_r a_r)$$
induces an isomorphism of varieties 
$$\Yb_{\Vb_\bullet} \longisom \Yb(\nb)$$
which sends $\Yb_{\Vb_\bullet,j}^\closed$ to $\Yb_j^\closed(\nb)$. Moreover, 
conjugacy by $a_1$ induces an isomorphism $\Tb^{wF} \simeq \Lb^F$ and it is easily 
checked that the above isomorphism is $(\Gb^F,\Lb^F)$-equivariant through this 
idenfication. Now, to $s$ is associated a linear character of $\Lb^F$ which, 
through the identification $\Tb^{wF} \simeq \Lb^F$, defines a linear 
character $\th : \Tb^{wF} \to \L^\times$. 

By Theorem~\ref{theo:theo-d-borel}, we only need to prove that 
Condition $C_{\Vb_{j-1}^* \cap \Vb_{j+1}^*}(s) \subset C_{\Vb_j^*}(s)$ 
is equivalent to $\PC(\nb,j,\th)$. So let us prove this last fact. 
The property $\PC(\nb,j,\th)$ can be rewritten 
as follows:
\begin{quotation}
\noindent{\bfit Property $\PC(\nb,j,\th)$.} 
{\it If $\a \in \Phi^+$ is such that $\th(N_w(w_1\cdots w_{j-2}(\a^\ve)))=1$ and 
$(w_{j-1}w_j)^{-1}(\a) \in \Phi^+$, then $w_{j-1}^{-1}(\a) \in \Phi^+$.}
\end{quotation}
Let $s'=a_1^{-1}sa_1 \in \Tb^{*wF^*}$. Note that $\PC(\nb,j,\th)$ is equivalent to 
$$C_{\lexp{w_1\dots w_{j-2}}{\Ub}^*}(s') \cap \lexp{w_1\cdots w_j}{\Ub^*} \subset 
\lexp{w_1\cdots w_{j-1}}{\Ub^*}.$$
By conjugating by $a_1$, and since $\lexp{a_1 n_1 \cdots n_i}{\Ub^*}=\Vb_i^*$, 
we get that $\PC(\nb,j,\th)$ is equivalent to 
$C_{\Vb_{j-1}^*}(s) \cap \Vb_{j+1}^*\subset \Vb_j^*$, as desired.

\bigskip

\noindent{\it Second step: The general case.} 
Let us now come back to the general case: we no longer assume that 
$\Lb$ is a maximal torus. Since 
$\rgammac(\Yb_{\Vb_\bullet,j}^\open,\L)e_s^{\Lb^F}=\rgammac(\Yb_{\Vb_\bullet,j}^\open,\L) 
\otimes_{\L\Lb^F} \L\Lb^Fe_s^{\Lb^F}$, and since $\L\Lb^Fe_s^{\Lb^F}$ lives 
in the category generated by the complexes 
$\RC_{\Tb' \subset \Bb'}^\Lb(\L\Tb^{\prime F} e_s^{\Tb^{\prime F}})$, where 
$\Bb'$ runs over the set of Borel subgroups of $\Lb$ admitting 
an $F$-stable maximal torus $\Tb'$ whose dual torus contains $s$ (see~\cite[Theorem~A']{BR}), 
it is sufficient to prove that 
$$\rgammac(\Yb_{\Vb_\bullet,j}^\open,\L) \otimes_{\L\Lb^F} 
\RC_{\Tb' \subset \Bb'}^\Lb(\L\Tb^{\prime F} e_s^{\Tb^{\prime F}})=0.$$

So let $(\Tb',\Bb')$ be a pair as above. Let $\Ub'$ denote the unipotent radical 
of $\Bb'$, let $\Tb^{\prime *}$ be an $F^*$-stable maximal torus of $\Lb^*$, containing 
$s$ and dual to $\Tb'$ and let $\Bb^{\prime *}$ be a Borel subgroup of $\Lb^*$ 
containing $\Tb^{\prime *}$ and dual to $\Bb'$. Then \cite[11.5]{dmbook}
$$\Yb_{\Vb_\bullet} \times_{\Lb^F} \Yb_{\Ub'}^\Lb \simeq \Yb_{\Ub'\Vb_\bullet},$$
(as $\Gb^F$-varieties-$\Tb^{\prime F}$). Here, we have set 
$\Ub'\Vb_\bullet=(\Ub'\Vb_1,\dots,\Ub'\Vb_r)$. Moreover, through this isomorphism, 
$\Yb_{\Vb_\bullet,j}^\open \times_{\Lb^F} \Yb_{\Ub'}^\Lb$ is sent to 
$\Yb_{\Ub'\Vb_\bullet,j}^\open$ hence, by applying the first step of this proof, 
we only need to prove that 
$C_{\Ub^{\prime *}\Vb_{j-1}^* \cap \Ub^{\prime *}\Vb_{j+1}^*}(s) 
\subset C_{\Ub^{\prime *}\Vb_j^*}(s)$. Since $\Vb_{j-1}^*$ and $\Vb_{j+1}^*$ both admit $\Lb^*$ as a Levi complement and
$\Ub^{\prime *}\subset\Lb^*$, it follows that
$\Ub^{\prime *}\Vb_{j-1}^* \cap \Ub^{\prime *}\Vb_{j+1}^*\subset \Ub^{\prime *}(\Vb_{j-1}^* \cap \Vb_{j+1}^*)$.
On the other hand, 
$C_{\Ub^{\prime *}(\Vb_{j-1}^* \cap \Vb_{j+1}^*)}(s)=C_{\Ub^{\prime *}}(s)C_{\Vb_{j-1}^* \cap \Vb_{j+1}^*}(s)\subset
C_{\Ub^{\prime *}}(s)C_{\Vb_j^*}(s)$ by assumption and this completes the proof.
\end{proof}

\bigskip

\begin{rema}
\label{re:FrobLevi}
Theorem \ref{theo:theod} provides a comparison of modules, together with the Frobenius
action. We have an isomorphism of $(\L\Gb^F,\L\Lb^F)$-bimodules compatible with the
Frobenius action
$$H^i_c(\Yb_{\Vb,j},\L)e_s^{\Lb^F}\simeq
H^{i-2r}_c(\Yb_{c_j(\Vb_\bullet)},\L)e_s^{\Lb^F}(-r).$$
where $r=d_j(\Vb_\bullet)$.
\end{rema}

\bigskip

Let $\shift(\Vb_\bullet)=(\Vb_2,\dots,\Vb_r,\lexp{F}{\Vb_1})$. The map 
$$\fonction{\shift_{\Vb_\bullet}}{\Yb_{\Vb_\bullet}}{\Yb_{\shift(\Vb_\bullet)}}{
(g_1\Vb_1,\dots,g_r\Vb_r)}{(g_2\Vb_2,\dots,g_r\Vb_r,F(g_1\Vb_1))}$$
is $(\Gb^F,\Lb^F)$-equivariant and 
induces an equivalence of \'etale sites. Therefore, it induces a quasi-isomorphism of complexes 
of bimodules
$$\shift_{\Vb_\bullet}^* : \rgammac(\Yb_{\shift(\Vb_\bullet)},\L) \longisom \rgammac(\Yb_{\Vb_\bullet},\L).$$
Applying twice Theorem~\ref{theo:theod}, we obtain the following result.

\bigskip
\begin{coro}\label{coro:theod-suite}
Let $j \in \{2,\dots,r\}$ and assume that 
$$C_{\Vb_{j-1}^* \cap \Vb_{j+1}^*}(s) \subset C_{\Vb_j^*}(s)
\quad\text{\it and}\quad
C_{\Vb_j^* \cap \Vb_{j+2}^*}(s) \subset C_{\Vb_{j+1}^*}(s).$$
The map $\Psi_{\Vb_\bullet,j,s} \circ \shift_{\Vb_\bullet}^* \circ \Psi_{\shift(\Vb_\bullet),j,s}^{-1}$ is 
a quasi-isomorphism of complexes of $(\L\Gb^F,\L\Lb^F)$-bimodules
$$\rgammacdim(\Yb_{c_j(\shift(\Vb_\bullet))},\L)e_s^{\Lb^F} \longisom 
\rgammacdim(\Yb_{c_j(\Vb_\bullet)},\L)e_s^{\Lb^F}.$$
\end{coro}

\bigskip

In the case $r=2$, Corollary~\ref{coro:theod-suite} becomes the following
result.

\bigskip

\begin{coro}\label{coro:theod}
Assume 
$$C_{\Vb_1^* \cap \lexp{F^*}{\Vb_1^*}}(s) \subset C_{\Vb_2^*}(s)
\quad\text{\it and}\quad
C_{\Vb_2^* \cap \lexp{F^*}{\Vb_2^{*}}}(s) \subset C_{\lexp{F^*}{\Vb_1^*}}(s).$$
The map $\Psi_{\Vb_1,\Vb_2,2,s} \circ \shift_{\Vb_1,\Vb_2}^* \circ \Psi_{\Vb_2,F(\Vb_1),2,s}^{-1}$ is 
a quasi-isomorphism of complexes of $(\L\Gb^F,\L\Lb^F)$-bimodules
$$\rgammacdim(\Yb_{\Vb_2},\L)e_s^{\Lb^F} \longisom 
\rgammacdim(\Yb_{\Vb_1},\L)e_s^{\Lb^F}.$$
As a consequence, we obtain a quasi-isomorphism of functors between
$$\RC_{\Lb \subset \Pb_1}^\Gb[\dim(\Yb_{\Vb_1})] : 
\Drm^b(\L\Lb^F e_s^{\Lb^F}) \longto \Drm^b(\L\Gb^Fe_s^{\Gb^F})$$ 
$$\RC_{\Lb \subset \Pb_2}^\Gb[\dim(\Yb_{\Vb_2})] : 
\Drm^b(\L\Lb^F e_s^{\Lb^F}) \longto \Drm^b(\L\Gb^Fe_s^{\Gb^F}).
\leqno{\text{\it and}}$$
\end{coro}

\bigskip

\begin{rema}
The isomorphism of functors of Corollary \ref{coro:theod} comes 
with a Tate twist. Keeping track of this twist has important
applications~\cite{W},~\cite{N}.
\end{rema}
\bigskip

\begin{rema}\label{rem:condition}
Let us make here some comments about the condition 
$$C_{\Vb_1^* \cap \lexp{F^*}{\Vb_1^*}}(s) \subset C_{\Vb_2^*}(s)
\quad\text{\it and}\quad
C_{{\Vb_2^*} \cap \lexp{F^*}{\Vb_2^*}}(s) \subset C_{\lexp{F^*}{\Vb_1^*}}(s).
\leqno{(\CC_{\Vb_1,\Vb_2})}$$

Note that if $C_{\Vb_1^*}(s)=C_{\Vb_2^{*}}(s)$, then Condition $(\CC_{\Vb_1,\Vb_2})$ 
is satisfied. 

Since $C_{\Vb_i^*}(s)$ is connected, it follows that if $C^\circ_{\Gb^*}(s)
\subset\Lb^*$, then 
Condition $(\CC_{\Vb_1,\Vb_2})$ is satisfied.
\end{rema}

\begin{exemple}\label{exemple:trivial}
Of course, Condition $(\CC_{\Vb_1,\Vb_1})$ is fulfilled for all $s$. 
Gluing the quasi-isomorphisms obtained from Corollary~\ref{coro:theod}, 
we get a quasi-isomorphism of complexes of bimodules
$$\Th_{\Vb_1,\Vb_1} : \rgammac(\Yb_{\Vb_1},\L) \longisom \rgammac(\Yb_{\Vb_1},\L).$$
But, since $\Yb_{\Vb_1,\Vb_1}^\open=\vide$, it is readily checked that 
$\Th_{\Vb_1,\Vb_1} = \Id_{\rgammac(\Yb_{\Vb_1},\L)}$.
\end{exemple}

\bigskip

\begin{exemple}\label{exemple:frob}
Similarly, Condition $(\CC_{\Vb_1,F(\Vb_1)})$ is fulfilled for all $s$. 
Gluing the quasi-isomorphisms obtained from Corollary~\ref{coro:theod}, 
we obtain a quasi-isomorphism of complexes of bimodules
$$\Th_{\Vb_1,F(\Vb_1)} : \rgammac(\Yb_{\Vb_1},\L) \longisom \rgammac(\Yb_{F(\Vb_1)},\L).$$
But, since $\Yb_{\Vb_1,F(\Vb_1)}^\open=\vide$, it is readily checked that 
$\Th_{\Vb_1,F(\Vb_1)} = F$.
\end{exemple}

\bigskip

\begin{rema}\label{rem:transitivite}
If $(\CC_{\Vb_1,\Vb_2})$ holds, we denote by 
$$\Th_{\Vb_1,\Vb_2,s} : \rgammacdim(\Yb_{\Vb_2},\L)e_s^{\Lb^F} \longisom 
\rgammacdim(\Yb_{\Vb_1},\L)e_s^{\Lb^F}$$
the quasi-isomorphism defined by 
$\Th_{\Vb_1,\Vb_2,s} = \Psi_{\Vb_1,\Vb_2,2,s} \circ \shift_{\Vb_1,\Vb_2}^* \circ \Psi_{\Vb_2,F(\Vb_1),2,s}^{-1}$. 
Assume moreover that $(\CC_{\Vb_1,\Vb_3})$ and $(\CC_{\Vb_2,\Vb_3})$ hold, so that 
the quasi-isomorphisms of complexes 
$\Th_{\Vb_1,\Vb_3,s}$ and $\Th_{\Vb_2,\Vb_3,s}$ are also well-defined. It is natural to ask the following
\begin{quotation}
\noindent{\bf Question.} 
{\it When does the equality $\Th_{\Vb_1,\Vb_3,s}=\Th_{\Vb_1,\Vb_2,s} \circ \Th_{\Vb_2,\Vb_3,s}$ hold?}
\end{quotation}
For instance, taking Example~\ref{exemple:trivial} into account, when does the equality 
$\Th_{\Vb_1,\Vb_2,s}^{-1} = \Th_{\Vb_2,\Vb_1,s}$ hold?

We do not know the answer to this question, but we can just say that the equality does not always hold. 
Indeed, if $m$ is minimal such that $F^m(\Vb_1)=\Vb_1$, 
then the isomorphisms $\Th_{\Vb_1,F(\Vb_1),s}$, $\Th_{F(\Vb_1),F^2(\Vb_1),s}$,\dots, 
$\Th_{F^{m-1}(\Vb_1),\Vb_1}$ are well-defined and all coincide with the Frobenius endomorphism $F$ 
(see Example~\ref{exemple:frob}), and so 
$$\Th_{\Vb_1,F(\Vb_1),s} \circ \Th_{F(\Vb_1),F^2(\Vb_1),s} \circ \cdots \circ 
\Th_{F^{m-1}(\Vb_1),\Vb_1,s} = F^m \neq \Id = \Th_{\Vb_1,\Vb_1,s}$$
(see Example~\ref{exemple:trivial}).
\end{rema}

\bigskip

\begin{exemple}
Let $\Pb_0$ be a parabolic subgroup admitting an $F$-stable Levi subgroup $\Lb_0$ containing $\Lb$. 
We denote by $\Vb_0$ the unipotent radical of $\Pb_0$ and $\Lb_0^*$ the corresponding Levi subgroup 
of a parabolic subgroup of $\Gb^*$ containing $\Lb^*$, which is dual to $\Lb_0$. 
We assume in this example that $C_{\Gb^*}^\circ(s) \subset \Lb_0^*$. Then 
it follows from~\cite[Theorem~11.7]{BR}, Corollary~\ref{coro:theod} 
and Remark~\ref{rem:condition} that 
we have an isomorphism of $(\L\Gb^F,\L\Lb^F)$-bimodules
$$\Hrm_c^{d_0}(\Yb_{\Vb_0},\L) \otimes_{\L\Lb^F_0} 
\rgammacdim(\Yb_{\Vb \cap \Lb_0}^{\Lb_0},\L) e_s^{\Lb^F}
\simeq \rgammacdim(\Yb_\Vb,\L)e_s^{\Lb^F},
$$
where $d_0=\dim(\Yb_{\Vb_0})$.
\end{exemple}

\bigskip
\begin{rema}
Let us consider the Harish-Chandra case: assume 
that $\Vb_1$ and $\Vb_2$ are $F$-stable. The functors
$\RC_{\Lb \subset \Pb_1}^\Gb$ and
$\RC_{\Lb \subset \Pb_2}^\Gb$ are isomorphic without truncating by any series \cite{DiDu,HoLe}. Such
isomorphisms are given by explicit isomorphisms of bimodules, which do not rely on any algebraic geometry.
We do not know if after truncation by a series satisfying ($\CC_{\Vb_1,\Vb_2}$), they coincide with our
isomorphisms.
\end{rema}

\bigskip

\subsection{Transitivity}
We will provide here an analogue to Lemma~\ref{lem:transitivite} in the 
more general context of this section. Assume in this subsection, and only in this subsection, 
that $3 \le j \le r$ (in particular, $r \ge 3$). 
Since $c_{j-1}(c_j(\Vb_\bullet))=c_{j-1}(c_{j-1}(\Vb_\bullet))$, we can build a diagram 
\equat\label{eq:transitivite-parab}
\diagram
\rgammacdim(\Yb_{\Vb_\bullet},\L) \rrto^{\DS{\Psi_{\Vb_\bullet,j}}} \ddto_{\DS{\Psi_{\Vb_\bullet,j-1}}} && 
\rgammacdim(\Yb_{c_j(\Vb_\bullet)},\L) \ddto_{\DS{\Psi_{c_j(\Vb_\bullet),j-1}}} \\
&&\\
\rgammacdim(\Yb_{c_{j-1}(\Vb_\bullet)},\L) \rrto^{\DS{\Psi_{c_{j-1}(\Vb_\bullet),j-1}}} && 
\rgammacdim(\Yb_{c_{j-1}(c_j(\Vb_\bullet)},\L).
\enddiagram
\endequat
It does not seem reasonable to expect that the diagram~(\ref{eq:transitivite-parab}) is commutative 
in general. However, we have the following result, obtained from the
results of section \S~\ref{appendice:dimension} below by copying the proof of
 Lemma~\ref{lem:transitivite}.

\bigskip

\begin{lem}\label{lem:transitivite-parab}
Assume that one of the following holds:
\begin{itemize}
\itemth{1} $\Vb_{j-2} \subset \Vb_{j+1} \cdot \Vb_{j-1}$.

\itemth{2} $\Vb_{j-1} \subset \Vb_{j-2} \cdot \Vb_j$.

\itemth{3} $\Vb_{j\hphantom{+1}} \subset \Vb_{j-1} \cdot \Vb_{j+1}$.

\itemth{4} $\Vb_{j+1} \subset \Vb_j \cdot \Vb_{j-2}$.
\end{itemize}
Then the diagram~(\ref{eq:transitivite-parab}) is commutative.
\end{lem}

\bigskip

\section{Jordan decomposition and quasi-isolated blocks}
\label{se:jordan}

\medskip
In this section, we assume $\Gb$ is connected. We fix an $F$-stable maximal torus
$\Tb$ of $\Gb$ and we denote by $(\Gb^*,\Tb^*,F^*)$ a triple dual to $(\Gb,\Tb,F)$. 

\medskip
We start in \S\ref{se:quasiisolatedsetting} with a recollection of some of the results
of \cite{BR} on the vanishing of the truncated cohomology of certain Deligne-Lusztig
varieties outside the middle degree. We fix an $F$-stable Levi subgroup $\Lb$ and 
consider $s\in\Gb^{*F^*}$
of order invertible in $\Lambda$ such that $C_{\Gb^*}^\circ(s)\subset\Lb^*$ (and we take
$\Lb$ minimal with that property).
We show that the corresponding middle degree
$(\Lambda \Gb^F,\Lambda\Lb^F)$-bimodule 
$\Hrm_c^{\dim(\Yb_{\Pb})}(\Yb_{\Pb},\L)e_s^{\Lb^F}$ does not
depend on the choice of the parabolic subgroup $\Pb$, up to isomorphism, thanks to the results of
\S\ref{sec:proof}. In particular, it is stable under the action of the stabilizer $N$ of
$e_s^{\Lb^F}$ in $N_{\Gb^F}(\Lb)$.

Section \S\ref{se:Clifford} develops some Clifford theory tools in order to extend the action
of $\Lb^F$ on $\Hrm_c^{\dim(\Yb_{\Pb})}(\Yb_{\Pb},\L)e_s^{\Lb^F}$ to an action of $N$.
We apply this in \S\ref{se:Morita}
by embedding $\Gb$ in a group $\Gbt$ with
connected center. This provides a Morita equivalence, extending
the main result of \cite{BR} to the quasi-isolated case.

In section \S\ref{se:splendid}, we show that the action of $\Lb^F$
on the complex of cohomology $C=\Grm\Gamma_c(\Yb_{\Pb},\L)e_s^{\Lb^F}$ also extends
to $N$, and the resulting complex provides a splendid Rickard equivalence. This relies
on checking that given $Q$ an $\ell$-subgroup of $\Lb^F$, the complex $\mathrm{Br}_{\Delta Q}(C)$
arises in a Jordan decomposition
setting for $C_{\Gb}(Q)$, and then applying the results of the Appendix. The main
difficulty is to prove that $\mathrm{br}_Q(e_s^{\Lb^F})$ is a sum of idempotents associated to a
Jordan decomposition setting for $C_{\Gb}(Q)$.  An added difficulty is that
the group $C_{\Gb}(Q)$ need not be connected.

\medskip
\subsection{Quasi-isolated setting}
\label{se:quasiisolatedsetting}

We fix 
a semisimple element $s \in \Gb^{*F^*}$ whose order is invertible in $\L$. 
Let $\Lb^*=C_{\Gb^*}\bigl(\Zrm(C_{\Gb^*}^\circ(s))^\circ\bigr)$,
an $F^*$-stable Levi complement of some parabolic subgroup 
$\Pb^*$ of $\Gb^*$. Note that $\Lb^*$ is a minimal Levi subgroup with respect to the property of 
containing $C_{\Gb^*}^\circ(s)$ and 
$C_{\Gb^*}(s)/C_{\Gb^*}^\circ(s)$ is an abelian $\ell'$-group~\cite[Corollary~2.8(b)]{qi}. In particular,
the series corresponding to $s$ is $(\Gb,\Lb)$-regular.

We denote by $(\Lb,\Pb)$ 
a pair dual to $(\Lb^*,\Pb^*)$. Note that $\Pb$ is a parabolic subgroup 
of $\Gb$ admitting $\Lb$ as an $F$-stable Levi complement. The unipotent 
radical of $\Pb$ will be denoted by $\Vb$. We put $d=\dim(\Yb_\Vb)$.

The group $C_{\Gb^*}(s)$ normalizes $\Lb^*$ and we set 
$\Nb^*=C_{\Gb^*}(s)^{F^*}\cdot\Lb^*$: it is a subgroup of $N_{\Gb^*}(\Lb^*)$ containing 
$\Lb^*$. Via the canonical isomorphism between $N_{\Gb^*}(\Lb^*)/\Lb^*$ and $N_\Gb(\Lb)/\Lb$,
we define the subgroup $\Nb$ of $N_\Gb(\Lb)$ 
containing $\Lb$ such that $\Nb/\Lb$ corresponds to $\Nb^*/\Lb^*$. Note that 
$\Nb^*$ is $F^*$-stable and so $\Nb$ is $F$-stable, and that $\Nb^{*F^*}/\Lb^{*F^*}$ and 
$\Nb^F/\Lb^F$ are abelian $\ell'$-groups.

\bigskip
\def\jordan{\Hrm_c^d(\Yb_\Vb,\L) e_s^{\Lb^F}}
\def\jordann{\Hrm_c^d(\Yb_{\lexp{n}{\Vb}},\L) e_s^{\Lb^F}}
\def\jordano{\Hrm_c^d(\Yb_\Vb,\OC) e_s^{\Lb^F}}
\def\jordank{\Hrm_c^d(\Yb_\Vb,K) e_s^{\Lb^F}}
\def\jordant{\Hrm_c^d(\Ybt_\Vb,\L) e_s^{\Lb^F}}

\medskip

Let us first derive some consequences of these assumptions. 
Note that $\Nb^*/\Lb^*=(\Nb^*/\Lb^*)^{F^*}=\Nb^{*F^*}/\Lb^{*F^*}$, 
so that $\Nb/\Lb=(\Nb/\Lb)^F=\Nb^F/\Lb^F$. Also, $\Nb^{*F^*}$ is the stabilizer, in 
$N_{\Gb^{*F^*}}(\Lb^*)$, of the $\Lb^{*F^*}$-conjugacy class of $s$. 
Therefore
\equat\label{eq:stab}
\text{\it $\Nb^F$ is the stabilizer of $e_s^{\Lb^F}$ in $N_{\Gb^F}(\Lb)$.}
\endequat
It follows that $e_s^{\Lb^F}$ is a central idempotent of $\L\Nb^F$. 
By~\cite[Theorem~11.7]{BR}, we have
$$\Hrm_c^i(\Yb_\Vb,\L)e_s^{\Lb^F}=0 \text{ for }i{\not=}d.$$

\bigskip

Our first result on the Jordan decomposition is the independence of the choice of
parabolic subgroups.

\begin{theo}
\label{th:indepjordan}
Given $\Pb'$ a parabolic subgroup of $\Gb$ with Levi complement $\Lb$ and unipotent radical
$\Vb'$, then
$\Hrm_c^{\dim(\Yb_{\Vb})}(\Yb_{\Vb},\L)e_s^{\Lb^F}\simeq
\Hrm_c^{\dim(\Yb_{\Vb'})}(\Yb_{\Vb'},\L)e_s^{\Lb^F}$ as
$(\Lambda \Gb^F,\Lambda\Lb^F)$-bimodules.

The $(\Lambda \Gb^F,\Lambda\Lb^F)$-bimodule $\Hrm_c^d(\Yb_{\Vb},\L)e_s^{\Lb^F}$
is $\Nb^F$-stable.
\end{theo}

\begin{proof}
The first result follows from Remark \ref{rem:condition} and
Corollary \ref{coro:theod}.

Let $n \in \Nb^F$. 
The isomorphism 
of varieties $\Gb/\Vb \longisom \Gb/\lexp{n}{\Vb},\ g\Vb \mapsto g\Vb n^{-1}$ 
induces an isomorphism of varieties
$\Yb_\Vb \xrightarrow{\sim} \Yb_{\lexp{n}{\Vb}}$.
As a consequence, we have an isomorphism
of $(\L\Gb^F,\L\Lb^F)$-bimodules 
$$\Hrm_c^d(\Yb_\Vb,\L) \simeq
n_*\bigl(\Hrm_c^d(\Yb_{\lexp{n}{\Vb}},\L)\bigr),$$
where 
$n_*\bigl(\Hrm_c^d(\Yb_{\lexp{n}{\Vb}},\L)\bigr)=
\Hrm_c^d(\Yb_{\lexp{n}{\Vb}},\L)$ as a left $\L\Gb^F$-module
and the right action of $a\in\L\Lb^F$ on 
$n_*\bigl(\Hrm_c^d(\Yb_{\lexp{n}{\Vb}},\L)\bigr)$ is given by the right
action of $nan^{-1}$ on $\Hrm_c^d(\Yb_{\lexp{n}{\Vb}},\L)$.

Since $n$ fixes $e_s^{\Lb^F}$, we deduce that
$$ \jordan \simeq n_*\bigl(\jordann\bigr).$$
On the other hand, the first part of the theorem shows that
$$\Hrm_c^d(\Yb_{\Vb},\L)e_s^{\Lb^F}\simeq
\Hrm_c^d(\Yb_{\lexp{n}{\Vb}},\L)e_s^{\Lb^F}.$$
It follows that
$\jordan \simeq n_*\bigl(\jordan\bigr)$.
\end{proof}

Recall that, if $\Nb^F=\Lb^F$ (that is, if $C_{\Gb^*}(s)^{F^*} \subset \Lb^*$), then 
$\jordan$ induces a Morita equivalence between $\L\Gb^Fe_s^{\Gb^F}$ and 
$\L\Lb^Fe_s^{\Lb^F}$ by \cite[Theorem~B']{BR}. Note that the assumption
in~\cite[Theorem~B']{BR} is $C_{\Gb^*}(s) \subset \Lb^*$, but it can easily be seen that the
proof requires only the assumption $C_{\Gb^*}(s)^{F^*} \subset \Lb^*$.
Theorem \ref{th:indepjordan} shows that this Morita equivalence does not depend 
on the choice of a parabolic subgroup.

\medskip
We will generalize the Morita equivalence to our situation. The main difficulty is to extend 
the action of $\Lb^F$ on $\Hrm_c^d(\Yb_\Vb,\L)e_s^{\Lb^F}$ to $\Nb^F$.

\subsection{Clifford theory}
\label{se:Clifford}

Let us recall some basic facts of Clifford theory. Let $\kb$ be a field.
Let $Y$ be a finite group and $X$ a normal subgroup of $Y$. Let $M$ be a finitely generated
$\kb X$-module that is $Y$-stable and let $A=\End_{kX}(M)$.

Given $y\in Y$, let $N_y$ be the set of $\phi\in\End_\kb(M)^\times$ such that
$\phi(xm)=yxy^{-1}\phi(m)$ for all $x\in X$ and $m\in M$. Note that $N_yN_{y'}=
N_{yy'}$ for all $y,y'\in Y$.

Let $N=\bigcup_{y\in Y}N_y$,
a subgroup of $\End_\kb(M)^\times$ containing $N_1=A^\times$ as a normal subgroup.
The action of $x\in X$ on $M$ defines an element of $N_x$, and this gives a morphism
$X\to N$. The $Y$-stability of $M$ gives a surjective morphism of groups $Y\to N/N_1,\
y\mapsto N_y$.


Let $\hat{Y}=Y\times_{N/N_1}N$. There is a diagonal embedding of $X$ as a normal subgroup of $\hat{Y}$.
There is a commutative diagram whose horizontal and
vertical sequences are exact:
$$\xymatrix{
& & 1\ar[d] & 1\ar[d] \\
&& X\ar@{=}[r]\ar[d] & X\ar[d] \\
1\ar[r] & A^\times\ar[r]\ar@{=}[d] & \hat{Y}\ar[r]\ar[d] & Y\ar[r]\ar[d] & 1 \\
1\ar[r] & A^\times\ar[r] & \hat{Y}/X\ar[r]\ar[d] & Y/X\ar[d]\ar[r]& 1 \\
&& 1 & 1
}$$

The action of $X$ on $M$ extends to an action of $Y$ if and only if the canonical morphism of
groups $\hat{Y}\to Y$ has a splitting that is the identity on $X$.
This is equivalent to the fact that
the canonical morphism of groups $\hat{Y}/X\to Y/X$ is a split surjection.

\smallskip
The extension of groups
$$1\to 1+J(A)\to A^\times\to A^\times/(1+J(A))\to 1$$
splits. Indeed, since $A$ is a finite-dimensional $\kb$-algebra,
there exists a $\kb$-subalgebra $S$ of $A$ such that the composition $S\hookrightarrow A\twoheadrightarrow
A/J(A)$ is an isomorphism. Since $A=S\oplus J(A)$, we have $A^\times=(1+J(A))\rtimes S^\times$.

\smallskip
If $[Y:X]\in\kb^\times$, then every group extension $1\to 1+J(A)\to Z\to Y/X\to 1$
splits, since $1+J(A)$ is the finite extension of abelian groups
$$(1+J(A)^i)/(1+J(A)^{i+1})\simeq
J(A)^i/J(A)^{i+1},$$
and those are $\kb(Y/X)$-modules.
Consequently, if $[Y:X]\in\kb^\times$, then the action of $X$ on $M$ extends to an action of 
$Y$ if and only if the extension
$$1\to A^\times/(1+J(A))\to \hat{Y}/X(1+J(A))\to Y/X\to
 1$$
splits.

\medskip
Consider now $\tilde{Y}$ a finite group with $Y$ and $\tilde{X}$ two normal subgroups such that
$X=Y\cap \tilde{X}$ and $\tilde{Y}=Y\tilde{X}$.
Let $\tilde{M}=\Ind_X^{\tilde{X}}(M)$, a $\tilde{Y}$-stable $\kb\tilde{X}$-module. We define $\tilde{N}_y$, $\tilde{N}$ and
$\hat{\tilde{Y}}$ as above, replacing $M$ by $\tilde{M}$.

Given $y\in Y$, we define a map $\rho:N_y\to\tilde{N}_y,\ \phi\mapsto (a\otimes m\mapsto yay^{-1}\otimes\phi(m))$
for $a\in \kb\tilde{X}$ and $m\in M$. This gives a morphism of groups $N\to\tilde{N}$ extending
the canonical morphism $A\to\End_{\kb\tilde{X}}(\tilde{M})$ and a morphism of
groups $\hat{Y}/X\to \hat{\tilde{Y}}/\tilde{X}$ giving 
a commutative diagram
$$\xymatrix{
1\ar[r] & A^\times\ar[r]\ar@{^{(}->}[d] & \hat{Y}/X \ar[r]\ar[d] & Y/X\ar[d]^\sim \ar[r]& 1 \\
1\ar[r] & \End_{\kb\tilde{X}}(\tilde{M})^\times\ar[r] & \hat{\tilde{Y}}/\tilde{X}\ar[r] & \tilde{Y}/\tilde{X}\ar[r]& 1
}$$

It induces a commutative diagram
$$\xymatrix{
1\ar[r] & A^\times/(1+J(A))\ar[r]\ar@{^{(}->}[d] & 
\hat{Y}/X(1+J(A))\ar[r]\ar[d] & Y/X\ar[d]^\sim \ar[r]& 1 \\
1\ar[r] & \End_{\kb\tilde{X}}(\tilde{M})^\times/(1+J(\End_{\kb\tilde{X}}(\tilde{M})))\ar[r] & 
\hat{\tilde{Y}}/\tilde{X}(1+J(\End_{\kb\tilde{X}}(\tilde{M})))\ar[r] & \tilde{Y}/\tilde{X}\ar[r]& 1
}$$

Assume the inclusion
$$\End_{\kb X}(M)^\times/(1+J(\End_{\kb X}(M)))\hookrightarrow
\End_{\kb\tilde{X}}(\tilde{M})^\times/(1+J(\End_{\kb\tilde{X}}(\tilde{M})))$$
splits (this happens for example if
$\End_{\kb\tilde{X}}(\tilde{M})/
J\bigl(\End_{\kb\tilde{X}}(\tilde{M}))\bigr)\simeq 
\kb^n$ for some $n$, for in that case the algebra embedding
$\End_{\kb X}(M)/J(\End_{\kb X}(M))\hookrightarrow
\End_{\kb \tilde{X}}(\tilde{M})/J(\End_{\kb \tilde{X}}(\tilde{M}))$ has a section).
If the surjection
$\hat{\tilde{Y}}/\tilde{X}(1+J(\End_{\kb\tilde{X}}(\tilde{M})))\to\tilde{Y}/\tilde{X}$
splits, then the surjection 
$\hat{Y}/X(1+J(\End_{\kb X}(M)))\to Y/X$ splits.

As a consequence, we have the following proposition.

\begin{prop}
\label{pr:Clifford}
Let $\tilde{Y}$ be a finite group and $Y$, $\tilde{X}$ be two normal subgroups of $\tilde{Y}$.
Let $X=Y\cap \tilde{X}$. We assume $\tilde{Y}=Y\tilde{X}$.
Let $\kb$ be a field with $[Y:X]\in\kb^\times$.

Let $M$ be a finitely generated $\kb X$-module that is $Y$-stable. We assume
that
$$\End_{\kb\tilde{X}}(\Ind_X^{\tilde{X}}(M))/
J\bigl(\End_{\kb\tilde{X}}(\Ind_X^{\tilde{X}}(M))\bigr)\simeq 
\kb^n\text{ for some }n.$$

If $\Ind_X^{\tilde{X}}(M)$ extends to $\tilde{Y}$, then $M$ extends to $Y$.
\end{prop}

\bigskip
\subsection{Embedding in a group with connected center and Morita equivalence}
\label{se:Morita}

We fix a connected reductive algebraic group $\Gbt$ containing $\Gb$ as a closed subgroup,
with an extension of $F$ to an endomorphism of $\Gbt$ such that $F^\delta$ is a Frobenius
endomorphism of $\Gbt$ defining an $\FM_q$-structure, and such that $\Gbt=\Gb\cdot Z(\Gbt)$
and $Z(\Gbt)$ is connected~\cite[proof of Corollary 5.18]{DL}. The inclusion
$\Gb\hookrightarrow\Gbt$ is called a {\em regular embedding}.

Let $\Tbt=\Tb\cdot Z(\Gbt)$, an $F$-stable maximal torus of $\Gbt$.
Fix a triple $(\Gbt^*,\Tbt^*,F^*)$ dual to $(\Gbt,\Tbt,F)$.
The inclusion $i:\Gb\hookrightarrow\Gbt$ induces a surjection
$i^*:\Gbt^*\twoheadrightarrow\Gb^*$.
Let $\Lbt=\Lb\cdot Z(\Gbt)$, so that $\Lbt^*=(i^*)^{-1}(\Lb^*)$.
Let $\tilde{\Nb}=\Nb\tilde{\Lb}$.

Let $J$ be a set of representatives of conjugacy classes of $\ell'$-elements $\tilde{t}\in
\Gbt^{*F^*}$ such that $i^*(\tilde{t})=s$ (recall that $C_{\Gbt^*}(\tilde{t})$ is connected because
$Z(\Gbt)$ is connected). Note that $J\subset\Lbt^{*F^*}$.

\begin{lem}
\label{le:idemptilde}
We have $e_s^{\Gb^F}=\sum_{\tti\in J}e_\tti^{\Gbt^F}$ and
$e_s^{\Lb^F}=\sum_{n\in\Nb^F/\Lb^F}\sum_{\tti\in J}ne_\tti^{\Lbt^F}n^{-1}$.
\end{lem}

\begin{proof}
The first statement is a classical translation from $\Gb^*$ to $\Gb$,
cf for instance \cite[Proposition 11.7]{asterisque}.

Let $\sti$ be a semisimple element of $\Gbt^{*F^*}$ such that $i^*(\sti)=s$. If $\L \neq K$, 
we will assume that $\sti$ has order prime to $\ell$ (this is always possible as we may 
replace $\sti$ by its $\ell'$-part if necessary). 
Note that $\sti \in \Lbt^{*F^*}$. 

Let $n\in\tilde{\Nb}^{*F^*}$ such that
$n\sti n^{-1}$ is $\Lbt^{*F^*}$-conjugate to $\sti$. Then 
$n\in \Lbt^{*F^*}\cdot C_{\Gbt^*}(\sti)$. Since $i^*(C_{\Gbt^*}(\sti))\subset C^\circ_{\Gb^*}(s)\subset
\Lb^*$, it follows that $i^*(n)\in \Lb^{*F^*}$. We have $\Nb^{*F^*}/\Lb^{*F^*}=
\Nbt^{*F^*}/\Lbt^{*F^*}$, hence $n\in \Lbt^{*F^*}$.

It follows that $\Nb^{*F^*}/\Lb^{*F^*}$ acts freely on the set of conjugacy classes of
$\Lbt^{*F^*}$ whose image under $i^*$ is the $\Lb^{*F^*}$-conjugacy class of $s$.
Through the identification of $\Nb^{*F^*}/\Lb^{*F^*}$ with $\Nb^F/\Lb^F$, this shows that
given $\tilde{t}\in J$,
the stabilizer in $\Nb^F$ of $e_{\tilde{t}}^{\Lbt^F}$ is $\Lb^F$.
\end{proof}

\medskip
\begin{theo}
\label{th:Moritaquasi}
The action of $k\Gb^Fe_s^{\Gb^F}\otimes (k\Lb^Fe_s^{\Lb^F})^\opp$
on $\Hrm_c^d(\Yb_\Vb,k) e_s^{\Lb^F}$ extends to an action of
$k\Gb^Fe_s^{\Gb^F}\otimes (k\Nb^Fe_s^{\Lb^F})^\opp$. The resulting 
$(k\Gb^Fe_s^{\Gb^F},k\Nb^Fe_s^{\Lb^F})$-bimodule induces a Morita equivalence.
\end{theo}

\begin{proof}
Let $\Pbt=\Pb\cdot Z(\Gbt)$
and let $\Pbt^*=i^{*-1}(\Pb^*)$. 
Note that $\Lbt$ (resp. $\Lbt^*$) is a Levi complement 
of $\Pbt$ (resp. $\Pbt^*$) and it is $F$-stable (resp. $F^*$-stable)
and the pair $(\Lbt^*,\Pbt^*)$ is dual to $(\Lbt,\Pbt)$.

We put 
$$X=(\Gb^F\times(\Lb^F)^\opp)\cdot\Delta \tilde{\Lb}^F, \quad
Y=(\Gb^F\times(\Nb^F)^\opp)\cdot\Delta \tilde{\Nb}^F,$$
$$\tilde{X}=\tilde{\Gb}^F\times(\tilde{\Lb}^F)^\opp \quad\text{ and }\quad
\tilde{Y}=\tilde{\Gb}^F\times(\tilde{\Nb}^F)^\opp.$$

Let $\Ybt_\Vb=\Yb_\Vb^{\Gbt}$. Through the embedding $\Gb/\Vb\hookrightarrow\Gbt/\Vb$, we identify $\Yb_\Vb$ with a
subvariety of $\Ybt_\Vb$. The stabilizer in $\tilde{X}$ of the subvariety $\Yb_\Vb$
of $\Ybt_\Vb$ is $X$, hence we have an isomorphism of $\tilde{X}$-varieties
$\Ind_X^{\tilde{X}}\Yb_\Vb\xrightarrow{\sim}\Ybt_\Vb$.

Let $M=\Hrm_c^d(\Yb_\Vb,k) e_s^{\Lb^F}$, a 
$(kX(e_s^{\Gb^F}\otimes e_s^{\Lb^F}))$-module.
Let $\tilde{M}=\Ind_X^{\tilde{X}}M$, a 
$(k\tilde{X}(e_s^{\Gb^F}\otimes e_s^{\Lb^F}))$-module. We have an isomorphism
of $(k\tilde{X}(e_s^{\Gb^F}\otimes e_s^{\Lb^F}))$-modules 
$\tilde{M}\xrightarrow{\sim}\Hrm_c^d(\Ybt_\Vb,k) e_s^{\Lb^F}$.

We put $e=\sum_{\tti\in J}e_\tti^{\Lbt^F}$. We have
$e_s^{\Lb^F}=\sum_{n\in\Nbt^F/\Lbt^F}nen^{-1}$ and $e$ is a central idempotent of $k\Lbt^F$
(Lemma \ref{le:idemptilde}).

The $kX$-module $M$ is $\Nb^F$-stable (Theorem \ref{th:indepjordan}), hence
the $k\tilde{X}$-module $\tilde{M}$ is $\Nb^F$-stable as well. It follows that given
$\tilde{t}\in J$ and $n\in\Nb^F$, we have
$n_*(\Hrm_c^d(\tilde{\Yb}_\Vb)e_{\tilde{t}}^{\Lbt^F})\simeq
\Hrm_c^d(\tilde{\Yb}_\Vb)e_{n\tilde{t}n^{-1}}^{\Lbt^F}$ as $k\tilde{X}$-modules. The classical Mackey formula for induction and restriction in finite
groups shows now that
$$\Hrm_c^d(\Ybt_\Vb,k) \bigl(\sum_{n\in\Nb^F/\Lb^F}e_{n\tilde{t}n^{-1}}^{\Lbt^F}\bigr)
\simeq
\Res_{\tilde{X}}^{\tilde{Y}}\Ind_{\tilde{X}}^{\tilde{Y}}\bigl(
\Hrm_c^d(\tilde{\Yb}_\Vb)e_{\tilde{t}}^{\Lbt^F}\bigr),$$
hence
$$\tilde{M}\simeq \Res_{\tilde{X}}^{\tilde{Y}}\Ind_{\tilde{X}}^{\tilde{Y}}\bigl(\tilde{M}e).$$

Lemma \ref{le:idemptilde} shows that
$\tilde{M}e$ induces a Morita equivalence between
$k\Gbt^Fe_s^{\Gb^F}$ and $k\Lbt^Fe$ (cf \cite[Theorem~B']{BR}).
In particular, it is a direct sum
of indecomposable modules no two of which are isomorphic.

Since $ek\Nbt^F$ induces a Morita equivalence between $k\Lbt^Fe$
and $k\Nbt^Fe_s^{\Lb^F}$, we deduce that the right action of $\Lbt^F$ on 
$\tilde{M}\simeq \tilde{M}e\otimes_{k\Lbt e}ek\Nbt^F$ extends
to an action of $\Nbt^F$ commuting with the left action of $\Gbt^F$ and the extended bimodule
$\tilde{M}'$ induces a Morita equivalence between $k\Gbt^Fe_s^{\Gb^F}$ and 
$k\Nbt^Fe_s^{\Lb^F}$. It follows that
$$\End_{k\tilde{X}}(\tilde{M})\simeq\End_{k(\Nbt^F\times(\Lbt^F)^{\opp})}(
k\Nbt^Fe_s^{\Lb^F}).$$

Given $n_1,n_2\in\Nbt^F$ with $n_1{\not\in}n_2\Lbt^F$,
the central idempotents $n_1en_1^{-1}$ and
$n_2en_2^{-1}$ of $k\Lbt^F$ are orthogonal. It follows that
$$\End_{k(\Nbt^F\times(\Lbt^F)^{\opp}}(k\Nbt^Fe_s^{\Lb^F})\simeq
\prod_{n\in\Nb^F/\Lb^F}\End_{k(\Nbt^F\times(\Lbt^F)^{\opp})}(
k\Nbt^Fnen^{-1})\simeq 
\bigl(Z(k\Lbt^Fe)\bigr)^{[\Nb^F/\Lb^F]},$$
the last isomorphism following from the fact that $k\Nbt^Fnen^{-1}$ induces a Morita equivalence between
$k\Nbt e_s^{\Lb^F}$ and $k\Lbt^Fnen^{-1}\simeq k\Lbt^Fe$.

We deduce that 
$\End_{k\tilde{X}}(\tilde{M})^\times/
\bigl(1+J(\End_{k\tilde{X}}(\tilde{M}))\bigr)
\simeq (k^\times)^r$ for some $r$.
Since $[Y:X]=[\Nb:\Lb]$ is prime to $\ell$, 
it follows from Proposition \ref{pr:Clifford} that the action of $X$ on $M$ extends to
an action of $Y$. Denote by $M'$ the extended module. We have $\Res_{\tilde{X}}^{\tilde{Y}}
\Ind_Y^{\tilde{Y}}(M')e\simeq \tilde{M}e\simeq \Res_{\tilde{X}}^{\tilde{Y}}
(\tilde{M}')e$, hence
$\Ind_Y^{\tilde{Y}}(M')\simeq \tilde{M}'$.
It follows that $\Ind_Y^{\tilde{Y}}(M')$ induces a Morita
equivalence between $k\Gbt^Fe_s^{\Gb^F}$ and $k\Nbt^Fe_s^{\Lb^F}$.
We have 
$$\End_{k\Gbt^F}(\Ind_Y^{\tilde{Y}}(M'))\simeq
\End_{k\Gbt^F}(k\Gbt^F\otimes_{k\Gb^F}M')\simeq
\Hom_{k\Gb^F}(M',M'\otimes_{k\Nb^F}k\Nbt^F)\simeq
\End_{k\Gb^F}(M')\otimes_{k\Nb^F}k\Nbt^F.$$
The canonical map $k\Nbt^Fe_s^{\Lb^F}\to\End_{k\Gbt^F}(\Ind_Y^{\tilde{Y}}M')$ is an isomorphism,
hence the canonical map $k\Nb^Fe_s^{\Lb^F}\to\End_{k\Gb^F}(M')$ is an isomorphism
as well. Also, $M$ is a faithful $k\Gb^F e_s^{\Gb^F}$-module, since
$\tilde{M}=\Ind_{\Gb^F}^{\Gbt^F}M$ is a faithful $k\Gbt^Fe_s^{\Gb^F}$-module.
We deduce that $M'$ induces a Morita equivalence between $k\Gb^Fe_s^{\Gb^F}$ and
$k\Nb^Fe_s^{\Lb^F}$.
\end{proof}

\bigskip

\subsection{Splendid Rickard equivalence and local structure} 
\label{se:splendid}

Recall that $\Lb$ is the minimal $F$-stable Levi subgroup of $\Gb$ such that $C_{\Gb^*}^\circ(s)\subset\Lb^*$.

\begin{theo}
\label{th:Rickardk}
The action of $k\Gb^Fe_s^{\Gb^F}\otimes (k\Lb^Fe_s^{\Lb^F})^\opp$
on $\Grm\Gamma_c(\Yb_\Vb,k) e_s^{\Lb^F}$ extends to an action of
$k\Gb^Fe_s^{\Gb^F}\otimes (k\Nb^Fe_s^{\Lb^F})^\opp$. The resulting complex induces a splendid
Rickard equivalence between $k\Gb^Fe_s^{\Gb^F}$ and $k\Nb^Fe_s^{\Lb^F}$.
\end{theo}

\begin{proof}
$\bullet\ ${\em Step 1}: Identification of
$\End_{k\Gb^F}^\bullet(\Grm\Gamma_c(\Yb_\Vb,k)e_s^{\Lb^F})$ in $\Ho^b(k(\Lb^F\times (\Lb^F)^\opp))$.

\smallskip
Let $C=(\Grm\Gamma_c(\Yb_\Vb,k)e_s^{\Lb^F})^{\red}$.
The vertices of the indecomposable direct summands of components of $C$ are
contained in $\Delta \Lb^F$ by Corollary \ref{cor:deltavertex}.
Let $Q$ be an $\ell$-subgroup of $\Lb^F$. We have
$\brauer_{\Delta Q}(C)\simeq \Grm\Gamma_c(\Yb_{C_\Vb(Q)}^{C_{\Gb}(Q)},k)\mathrm{br}_Q(e_s^{\Lb^F})$ in
$\Ho^b(k(C_{\Gb^F}(Q)\times C_{\Lb^F}(Q)^\opp))$ by Proposition \ref{prop:centralisateur}.
Let $\XC$ be the rational series of $(\Lb,F)$ corresponding to $s$, so that
$e_s^{\Lb^F}=e_\XC$. Theorem \ref{theo:brauer} shows that
$$\mathrm{br}_Q(e_\XC)=\sum_{\YC\in (i_Q^\Lb)^{-1}(\XC)}e_{\YC}.$$

Let $\YC\in (i_Q^{\Lb})^{-1}(\XC)$. Proposition \ref{pr:seriesQ} shows that $\YC$
is $(C_{\Gb}^\circ(Q),C_{\Lb}^\circ(Q))$-regular.
It follows from \cite[Theorem~11.7]{BR}
that $\Hrm_c^i(\Yb_{C_\Vb(Q)}^{C_{\Gb}^\circ(Q)},k)e_\YC=0$ for 
$i{\not=}\dim\Yb_{C_\Vb(Q)}^{C_{\Gb}^\circ(Q)}$, hence
$\Hrm_c^i(\Yb_{C_\Vb(Q)}^{C_{\Gb}(Q)},k)e_\YC=0$ for 
$i{\not=}\dim\Yb_{C_\Vb(Q)}^{C_{\Gb}(Q)}$.
We have shown that the cohomology of
$\brauer_{\Delta Q}(C)$ is concentrated in a single degree. Note that
$\Res_{kC_{\Gb^F}(Q)}(\brauer_{\Delta Q}(C))$ is a perfect complex, hence
its homology is projective as a $kC_{\Gb^F}(Q)$-module.
We deduce from Theorem \ref{th:splitEnd} that
$$\End_{k\Gb^F}^\bullet(C)\simeq \End_{D^b(k\Gb^F)}(C) \quad\text{ in }\Ho^b(k(\Lb^F\times
(\Lb^F)^\opp)).$$

\medskip
$\bullet\ ${\em Step 2}: Study of 
$\End_{\Ho^b(k(\Gb^F\times(\Nb^F)^\opp))}(
\Ind_{\Gb^F\times(\Lb^F)^\opp}^{\Gb^F\times(\Nb^F)^\opp}\Grm\Gamma_c(\Yb_\Vb,k)e_s^{\Lb^F})$.

\smallskip
Let $C'=\Ind_{\Gb^F\times(\Lb^F)^\opp}^{\Gb^F\times(\Nb^F)^\opp}C$. 
Let $P$ be a projective resolution of $k\Nb^F$, i.e., a
complex of $k(\Nb^F\times(\Nb^F)^\opp)\mproj$ with $P^i=0$ for $i>0$, together
with a quasi-isomorphism $P\to k\Nb^F$ of $k(\Nb^F\times(\Nb^F)^\opp)$-modules.
As the terms of $C'$ are projective $k\Gb^F$-modules, we have
a commutative diagram
$$\xymatrix{
\End_{\Ho^b(k(\Gb^F\times(\Nb^F)^\opp))}(C') \ar[r] \ar[d]_\sim & 
\End_{D^b(k(\Gb^F\times(\Nb^F)^\opp))}(C') \ar[d]^\sim \\
\Hom_{\Ho^b(k(\Nb^F\times (\Nb^F)^\opp))}(k\Nb^F,\End_{k\Gb^F}^\bullet(C')) \ar[r] &
\Hom_{\Ho^b(k(\Nb^F\times (\Nb^F)^\opp))}(P,\End_{k\Gb^F}^\bullet(C'))
}$$

Using the isomorphisms of complexes in $\Ho^b(k(\Nb^F\times
(\Nb^F)^\opp))$
$$\End_{k\Gb^F}^\bullet(C')\simeq \Ind_{\Lb^F\times (\Lb^F)^\opp}^{\Nb^F\times (\Nb^F)^\opp}(
\End_{k\Gb^F}^\bullet(C))$$
and
$$\End_{D^b(k\Gb^F)}(C')\simeq \Ind_{\Lb^F\times (\Lb^F)^\opp}^{\Nb^F\times (\Nb^F)^\opp}(
\End_{D^b(k\Gb^F)}(C)),$$
we deduce that
$$\End_{k\Gb^F}^\bullet(C')\simeq \End_{D^b(k\Gb^F)}(C')\text{ in }\Ho^b(k(\Nb^F\times
(\Nb^F)^\opp)).$$
Now, the canonical map
$$\Hom_{\Ho^b(k(\Nb^F\times (\Nb^F)^\opp))}(k\Nb^F,\End_{D^b(k\Gb^F)}(C')) \to
\Hom_{\Ho^b(k(\Nb^F\times (\Nb^F)^\opp))}(P,\End_{D^b(k\Gb^F)}(C'))$$
is an isomorphism, since $\End_{D^b(k\Gb^F)}(C')$ is a complex concentrated in degree $0$.
It follows that the top horizontal map in the commutative diagram
above is an isomorphism, hence we have canonical isomorphisms
$$\End_{\Ho^b(k(\Gb^F\times(\Nb^F)^\opp))}(C')\xrightarrow{\sim}
\End_{D^b(k(\Gb^F\times(\Nb^F)^\opp))}(C')\xrightarrow{\sim}
\End_{k(\Gb^F\times(\Nb^F)^\opp)}(\Ind_{\Gb^F\times(\Lb^F)^\opp}^{\Gb^F\times(\Nb^F)^\opp}
\Hrm_c^d(\Yb_\Vb,k))
.$$

\medskip
$\bullet\ ${\em Step 3}: Construction of a direct summand $\tilde{C}$ of
$\Ind_{\Gb^F\times(\Lb^F)^\opp}^{\Gb^F\times(\Nb^F)^\opp}
(\Grm\Gamma_c(\Yb_\Vb,k)e_s^{\Lb^F})$.

\smallskip

We have shown (Theorem \ref{th:Moritaquasi}) that there is a direct summand $M'$ of
 $\Ind_{\Gb^F\times(\Lb^F)^\opp}^{\Gb^F\times(\Nb^F)^\opp}
\Hrm_c^d(\Yb_\Vb,k)$ whose restriction to $\Gb^F\times (\Lb^F)^\opp$ is isomorphic to 
$\Hrm_c^d(\Yb_\Vb,k)$. Let $i$ be the corresponding idempotent of 
$\End_{k(\Gb^F\times(\Nb^F)^\opp)}(\Ind_{\Gb^F\times(\Lb^F)^\opp}^{\Gb^F\times(\Nb^F)^\opp}
\Hrm_c(\Yb_\Vb,k))$ and $j$ its inverse image in 
$\End_{\Ho^b(k(\Gb^F\times(\Nb^F)^\opp))}(C')$ via the isomorphisms above.
We have a surjective homomorphism of finite-dimensional $k$-algebras
$$\End_{\Comp(k(\Gb^F\times(\Nb^F)^\opp))}(C')\twoheadrightarrow
\End_{\Ho^b(k(\Gb^F\times(\Nb^F)^\opp))}(C').$$
Consequently, $j$ lifts to an idempotent $j'$ of $\End_{\Comp(k(\Gb^F\times(\Nb^F)^\opp))}(C')$ 
\cite[Theorem 3.2]{Th}.
It corresponds to a direct summand $\tilde{C}$ of $C'$ quasi-isomorphic to $M'$ and
$\Res_{\Gb^F\times(\Lb^F)^\opp}^{\Gb^F\times(\Nb^F)^\opp}(\tilde{C})$ is a direct summand of
$\Res_{\Gb^F\times(\Lb^F)^\opp}^{\Gb^F\times(\Nb^F)^\opp}(C')\simeq
C^{\oplus [\Nb^F:\Lb^F]}$.

\medskip
$\bullet\ ${\em Step 4}: $\tilde{C}$ lifts $\Grm\Gamma_c(\Yb_\Vb,k)e_s^{\Lb^F}$.

\smallskip

Let $C=\bigoplus_{1\le r\le n}C_r$ be a decomposition into a direct sum of indecomposable objects
of $\Ho^b(k(\Gb^F\times(\Lb^F)^\opp)$. This induces a decomposition
$M=\bigoplus_{1\le r\le n}M_r$, where $M_r=H^d(C_r)$ and
$M_r$ and $M_{r'}$ have no isomorphic indecomposable summands for $r{\not=}r'$
(cf proof of Theorem \ref{th:Moritaquasi}).
We have 
$\Res_{\Gb^F\times(\Lb^F)^\opp}^{\Gb^F\times(\Nb^F)^\opp}(\tilde{C})\simeq
\bigoplus_{1\le r\le n}C_r^{\oplus a_r}$
in $\Ho^b(k(\Gb^F\times(\Lb^F)^\opp)$ for some integers $0\le a_r\le [\Nb^F:\Lb^F]$ and
$\bigoplus_{1\le r\le n}H^d(C_r)^{\oplus a_r}\simeq M$. It follows that $a_r=1$ for all $r$, hence
$\Res_{\Gb^F\times(\Lb^F)^\opp}^{\Gb^F\times(\Nb^F)^\opp}(\tilde{C})\simeq C$
in $\Ho^b(k(\Gb^F\times(\Lb^F)^\opp)$.
This shows the first statement.

\medskip
$\bullet\ ${\em Step 5}: Rickard equivalence.

\smallskip

We have shown above that $\End_{k\Gb^F}^\bullet(\tilde{C})\simeq\End_{D^b(k\Gb^F)}(\tilde{C})$
in $\Ho^b(k(\Nb^F\times (\Nb^F)^\opp))$. On the other hand,
$\End_{D^b(k\Gb^F)}(\tilde{C})\simeq\End_{k\Gb^F}(M')\simeq k\Nb^Fe_s^{\Lb^F}$.
It follows from Corollary \ref{cor:splendid} that $\tilde{C}$
induces a splendid Rickard equivalence.
\end{proof}

We now summarize and complete the description of the Jordan decomposition of blocks.

\begin{theo}
The complex of 
$(\OC\Gb^Fe_s^{\Gb^F},\OC\Lb^Fe_s^{\Lb^F})$-bimodules
$\Grm\Gamma_c(\Yb_\Vb,\OC)^\red e_s^{\Lb^F}$ extends to a complex $C$ of
$(\OC\Gb^Fe_s^{\Gb^F},\OC\Nb^Fe_s^{\Lb^F})$-bimodules. The complex $C$ induces a splendid Rickard
equivalence between $\OC\Gb^Fe_s^{\Gb^F}$ and $\OC\Nb^Fe_s^{\Lb^F}$.

There is a (unique) bijection $b\mapsto b'$ between blocks of $\OC\Gb^Fe_s^{\Gb^F}$ and
$\OC\Nb^Fe_s^{\Lb^F}$ such that $bC\simeq Cb'$.

Given $b$ a block of $\OC\Gb^Fe_s^{\Gb^F}$, then:
\begin{itemize}
\item 
the bimodule $H^{\dim\Yb_\Vb}(bCb')$ induces a Morita equivalence
between $\OC\Gb^Fb$ and $\OC\Nb^F b'$
\item
the complex $bCb'$ induces a splendid Rickard equivalence between
$\OC\Gb^Fb$ and $\OC\Nb^F b'$
\item there is a (unique) equivalence $(Q,b'_Q)\mapsto (Q,b_Q)$
from the category of $b'$-subpairs to the category of $b$-subpairs such that
$b_Q\brauer_{\Delta Q}(C)=\brauer_{\Delta Q}(C)b'_Q$. In particular, if $D$ is a defect group
of $b'$, then $D$ is a defect group of $b$.
\end{itemize}
\end{theo}

\begin{proof}
Theorem \ref{th:Rickardk} provides a complex $C'$ of
$(k\Gb^Fe_s^{\Gb^F}\otimes (k\Nb^Fe_s^{\Lb^F})^\opp)$-modules
inducing a splendid Rickard equivalence.
By Rickard's lifting Theorem \cite[Theorem 5.2]{Ri2}, there is a splendid complex $C$ of
$(\OC\Gb^Fe_s^{\Gb^F}\otimes (\OC\Nb^Fe_s^{\Lb^F})^\opp)$-modules, unique up to isomorphism in
$\Comp(\OC(\Gb^F\times(\Nb^F)^\opp))$, such that $kC\simeq C'$. Also,
\cite[proof of Theorem 5.2]{Ri2} shows that $\Grm\Gamma_c(\Yb_\Vb,\OC)^\red e_s^{\Lb^F}$
is the unique splendid complex that lifts $\Grm\Gamma_c(\Yb_\Vb,k)^\red e_s^{\Lb^F}$.
As a consequence, there is an isomorphism of complexes
$$\Res_{\Gb^F\times (\Lb^F)^\opp}^{\Gb^F\times(\Nb^F)^\opp}(C)\simeq
\Grm\Gamma_c(\Yb_\Vb,\OC)^\red e_s^{\Lb^F}.$$
By \cite[Theorem 5.2]{Ri2}, the complex $C$ induces a splendid Rickard equivalence.

\smallskip
Since $H^d(bkCb')$ induces a Morita equivalence, it follows that $H^d(bCb')$ induces a Morita
equivalence (cf e.g. \cite[proof of Theorem 5.2]{Ri2}).

The existence of the bijection between blocks follows from the isomorphism of algebras
$Z(\OC\Gb^Fe_s^{\Gb^F})\xrightarrow{\sim} Z(\OC\Nb^Fe_s^{\Lb^F})$ induced by the Morita
equivalence, and the blockwise statements on Morita and Rickard equivalence are clear.

By \cite[Theorem 19.7]{Pu}, it follows that the Brauer categories of $k\Gb^Fb$ and $k\Nb^Fb'$ are
equivalent, and in particular, $k\Gb^Fb$ and $k\Nb^F b'$ have isomorphic defect groups.
\end{proof}

\bigskip
\begin{rema}
If was already known that given $b$ a block of $\OC\Gb^Fe_s^{\Gb^F}$,
then $b$ and $b'$ have isomorphic defect groups under
one of the following assumptions:
\begin{itemize}
\item 
$\ell$ doesn't divide $|Z(\Gb)/Z(\Gb)^\circ)^F|$ nor
$|Z(\Gb^*)/Z(\Gb^*)^\circ)^F|$, $\ell\ge 5$ and $\ell\ge 7$ if $\Gb$ has a
component of type $E_8$ \cite[Proposition 5.1]{CaEn1}.
\item $C_{\Gb^*}(s)\subset\Lb^*$ and 
either $b$ or $b'$ has a defect group that is abelian
modulo the $\ell$-center of $\Gb^F$ \cite[Theorem 1.3]{KeMa}.
\end{itemize}
\end{rema}

\bigskip

\begin{exemple}\label{exemple:cgs-levi}
Assume in this example that $C_{\Gb^*}^\circ(s)=\Lb^*$ and 
that $(C_{\Gb^*}(s)/C_{\Gb^*}^\circ(s))^{F^*}$ is cyclic. The element
$s$ defines a linear character $\shat : \Lb^F \to \OC^\times$ 
which induces an isomorphism of algebra $\OC\Lb^Fe_s^{\Lb^F} \simeq \OC\Lb^F e_1^{\Lb^F}$. 
The linear character $\shat$ is stable under the action of $\Nb^F$ so, since $\Nb^F/\Lb^F$ is cyclic, 
it extends to a linear character $\shat^+ : \Nb^F \to \OC^\times$. Again, 
$\shat^+$ induces an isomorphism of algebra $\OC\Nb^Fe_s^{\Lb^F} \simeq \OC\Nb^F e_1^{\Lb^F}$. 
Combined with this, Theorem~\ref{th:Moritaquasi} provides a Morita equivalence between
$\OC\Nb^F e_1^{\Lb^F}$ and $\OC\Gb^Fe_s^{\Gb^F}$.
\end{exemple}

\bigskip

\begin{exemple}[{\bfit Type ${\boldsymbol{A}}$}]\label{exemple:type-a}
Assume in this example that all the simple components of 
$\Gb$ are of type $A$ (no assumption is made on the action of $F$). 
Then $C_{\Gb^*}^\circ(s)=\Lb^*$ and $C_{\Gb^*}(s)/C_{\Gb^*}^\circ(s)$
is cyclic. 
Therefore, Example~\ref{exemple:cgs-levi} can be 
applied to provide a Morita equivalence 
between $\OC\Nb^F e_1^{\Lb^F}$ and $\OC\Gb^Fe_s^{\Gb^F}$.
\end{exemple}

\bigskip
\begin{rema}
This article was announced at the end of the introduction of \cite{BR}.
Unfortunately, we have not been able to settle the problem of finiteness of source algebras.
On the other hand, in addition to what was announced in \cite{BR}, we have provided an extension
of the Jordan decomposition to the quasi-isolated case.
\end{rema}

\bigskip
\setcounter{section}{0}
\renewcommand\thesection{\Alph{section}}
\def\sectionname{Appendix}

\section{About $\ell$-permutation modules}\label{app:l-perm}

\bigskip

In this section, we assume $\Lambda=\OC$ or $\Lambda=k$. 

Let us recall here some results of Brou\'e and Puig, cf
\cite[\S{3.6}]{broue perm}. 
Let $G$ be a finite group.
Note that an $\ell$-permutation $\OC G$-module $M$ is indecomposable
if and only if $kM$ is an indecomposable $kG$-module.

Let $P$ be an $\ell$-subgroup of $G$.
An indecomposable $\ell$-permutation $\Lambda G$-module $M$ 
has a vertex containing $P$ if and only if $\brauer_P(M) \neq 0$. 
Also, given $V$ an indecomposable projective $k[N_G(P)/P]$-module 
(it is then an $\ell$-permutation $kG$-module), there exists 
a unique indecomposable $\ell$-permutation $\Lambda G$-module $\Mrm(P,V)$ 
such that $\brauer_P \Mrm(P,V) \simeq V$. The $\Lambda G$-module
$\Mrm(P,V)$ has vertex $P$. Moreover, every indecomposable 
$\ell$-permutation $\Lambda G$-module with vertex $P$ is isomorphic to 
such an $\Mrm(P,V)$.

\bigskip
The following lemma is a variant of \cite[Proposition 6.4]{Bou}.

\begin{lem}\label{lem:split}
Let $M$ and $N$ be $\ell$-permutation $\Lambda G$-modules and let $\psi \in \Hom_{\Lambda G}(M,N)$. 
Assume that all indecomposable summands of $N$ have a vertex 
equal to a given subgroup $P$ of $G$ and that $\brauer_P(\psi)$ is a 
surjection. Then $\psi$ is a split surjection.
\end{lem}

\bigskip

\begin{proof}
Proceeding by induction on the dimension of $N$, we can assume that 
$N$ is indecomposable. Fix a decomposition $M=\bigoplus_{i \in I} M_i$ where 
$M_i$ is indecomposable for all $i \in I$ and let $\psi_i : M_i \to N$ 
denote the restriction of $\psi$. Since $\brauer_P(\psi)$ is a surjection
and $\brauer_P(N)$ is an indecomposable projective $k[N_G(P)/P]$-module,
we deduce that $\brauer_P(\psi_i) : \brauer_P(M_i) 
\to \brauer_P(N)$ is a split surjection for some $i \in I$.

By \cite[Theorem 3.2(4)]{broue perm}, it follows that $N$ is isomorphic to
a direct summand
of $M_i$. Since $M_i$ is indecomposable, there is an isomorphism
$\psi' : N \longisom M_i$.
The morphism $\brauer_P(\psi_i \psi')=\brauer_P(\psi_i)\brauer_P(\psi')$ is
an isomorphism, 
so it is not nilpotent. Therefore, $\psi_i \psi'$ does not belong to the radical 
of $\End_{\Lambda G}(N)$, hence it is 
invertible (because $\End_{\Lambda G}(N)$ is a local ring). So $\psi_i$ is an 
isomorphism, as desired.
\end{proof}

\bigskip

\begin{lem}\label{lem:bounded}
Let $C$ be a bounded complex of $\ell$-permutation $\L G$-modules and
$P$ an $\ell$-subgroup of $G$
such that $\brauer_Q(C)$ is acyclic for all $\ell$-subgroups of $G$ that
are not conjugate to a subgroup of $P$.
Let $D$ be a bounded 
complex of finitely generated projective $k[N_G(P)/P]$-modules. We assume that 
$\brauer_P(C)\simeq D$ in $\Ho^b(k[N_G(P)/P])$.

Then there exists a bounded complex $C'$ of $\ell$-permutation 
$\L G$-modules, all of whose indecomposable summands have a vertex contained in $P$, 
such that $C'\simeq C$ in $\Ho^b(\L G)$ and $\brauer_P(C')\simeq D$
in $\Comp^b(k[N_G(P)/P])$.
\end{lem}

\bigskip

\begin{proof}
Up to isomorphism in $\Ho^b(\L G)$, we may assume that $C=C^\red$. We write 
$C=(C^\bullet,d^\bullet)$. 
We will first show by induction on the length of $C$ that $\brauer_P(C)=\brauer_P(C)^\red$ and that the indecomposable summands of $C$ have a vertex contained
in $P$.

Let $n$ be maximal such that $C^{n+1} \neq 0$. We fix a decomposition 
$C^{n+1}=\bigoplus_{i \in I} M_i$ where $M_i$ is indecomposable for all $i \in I$ 
and we denote by $p_i : C^{n+1} \to M_i$ the projection. 

Let $i\in I$ and let $Q$ be the vertex of $M_i$.
Assume that the composition 
$$\diagram \brauer_Q(C^n) \rrto^{\DS{\brauer_Q(d^n)}} && \brauer_Q(C^{n+1}) 
\rrto^{\DS{\brauer_Q(p_i)}} && \brauer_Q(M_i) \enddiagram$$
is surjective.
It follows from Lemma~\ref{lem:split} that $p_i d^n : C^n \to M_i$ 
is a split surjection: this contradicts the fact that $C=C^\red$. 
If $Q$ is not conjugate to a subgroup of $P$, then $\brauer_Q(d^n)$ is
sujective by assumption, hence a contradiction.
We deduce by induction that the indecomposable summands of $C$ have a vertex
contained in $P$.

$\brauer_Q(C)=0$ if $Q$ is not conjugate to a
subgroup of $P$.

We deduce also that the complex 
$$\diagram 0 \rto & \brauer_P(C^n) \rrto^{\DS{\brauer_P(d^n)}} &&
\brauer_P(C^{n+1}) \rto & 0 \enddiagram$$
has no non-zero direct summand that is homotopy equivalent to $0$. 
By the induction hypothesis, the complex
$$\diagram \cdots  \rto & \brauer_P(C^{n-1}) \rrto^{\DS{\brauer_P(d^{n-1})}} &&
\brauer_P(C^{n}) \rto & 0 \enddiagram$$
has no non-zero direct summand that is homotopy equivalent to $0$. 
It follows that $\brauer_P(C)=\brauer_P(C)^\red$.

We deduce from this that $D \simeq \brauer_P(C) \oplus D'$, where $D'$ is homotopy equivalent to $0$. 
So $D'$ is a sum of complexes of the form $0 \to V \xrightarrow{\Id} V \to 0$ 
with $V$ projective indecomposable
(up to a shift), hence there is a bounded complex $C'$ of $\ell$-permutation 
$\L G$-modules that is a direct sum of 
complexes of the form $0 \to \Mrm(P,V) \xrightarrow{\Id} \Mrm(P,V) \to 0$
with $V$ projective indecomposable
such that $\brauer_P(C') \simeq D'$. We have $\brauer_P(C \oplus C') \simeq D$, 
as desired.
\end{proof}

\bigskip
The following lemma is close to \cite[Proposition 7.9]{Bou}.

\begin{lem}
\label{le:H0Brauer}
Let $G$ be a finite group and $C$ be a bounded complex of $\ell$-permutation
 $kG$-modules.
Assume $H^i(\brauer_Q(C))=0$ for all $i\not=0$ and all $\ell$-subgroups $Q$ of $G$.

Then $C\simeq H^0(C)$ in $\Ho^b(kG)$.
\end{lem}

\begin{proof}
Replacing $C$ by $C^{\red}$, we can and will assume that $C$ has
no nonzero direct summands that are homotopy equivalent to $0$.

Let $i>0$ be maximal such that $C^i{\not=}0$. The map
$d^{i-1}_{\brauer_Q(C)}=\brauer_Q(d^{i-1}_C):\brauer_Q(C^{i-1})\to\brauer_Q(C^i)$ is surjective for all
$\ell$-subgroups $Q$. It follows from Lemma \ref{lem:split}
that $d^{i-1}_C$ is a split surjection: this
contradicts our assumption on $C$. So $C^i=0$ for $i>0$.
Replacing $C$ by $C^*$, we obtain similarly that $C_Q^i=0$ for $i<0$. The lemma
follows.
\end{proof}

\bigskip

The following theorem is a variant of \cite[Theorem 5.6]{Rou2}.

\begin{theo}
\label{th:splitEnd}
Let $G$ be a finite group and $H$ a subgroup of $G$.
Let $C$ be a bounded complex of
$\ell$-permutation $k(G\times H^\opp)$-modules all of whose indecomposable summands have a vertex
contained in $\Delta H$.

Assume $\Hom_{D^b(kC_G(Q))}(\brauer_{\Delta Q}(C),\brauer_{\Delta Q}(C)[i])=0$ for all $i\not=0$ and
all $\ell$-subgroups $Q$ of $H$.

Then $\End^\bullet_{kG}(C)$ is isomorphic to $\End_{D^b(kG)}(C)$ in $\Ho^b(k(H\times H^\opp))$.
\end{theo}

\begin{proof}
Let $R$ be an $\ell$-subgroup of $H\times H^\opp$.
By \cite[proof of Theorem 4.1]{Ri2}, we have 
$\brauer_R(\End_{kG}^\bullet(C))=0$ if $R$ is not conjugate to a subgroup 
of $\Delta H$, and given $Q\le H$ an $\ell$-subgroup, we have
$$\brauer_{\Delta Q}(\End_{kG}^\bullet(C))\simeq
\End_{kC_G(Q)}^\bullet(\brauer_{\Delta Q}(C))$$
in $\Comp(k(C_H(Q)\times C_H(Q)^\opp))$.

Note that the indecomposable summands of $\brauer_{\Delta Q}(C)$ are
projective for $kC_G(Q)$ since their vertices are contained
in $x(\Delta H)x^{-1} \cap (C_G(Q)\times 1)$ for some $x\in G\times H^\opp$,
 hence
$$H^i(\End_{kC_G(Q)}^\bullet(\brauer_{\Delta Q}(C)))\simeq
\Hom_{D^b(kC_G(Q))}(\brauer_{\Delta Q}(C),\brauer_{\Delta Q}(C)[i])$$
and this vanishes for $i\not=0$. Consequently,
$$\brauer_{\Delta Q}(\End_{kG}^\bullet(C))\simeq \End_{D^b(kC_G(Q))}(\brauer_{\Delta Q}(C))$$
in $D^b(k(C_H(Q)\times C_H(Q)^\opp))$.

The conclusion of the theorem follows now from Lemma \ref{le:H0Brauer} applied to
the complex $\End_{kG}^\bullet(C)$.
%
%
%
\end{proof}

The following corollary, used in the proof of Theorem \ref{th:Rickardk}, might be useful in other settings.

\begin{coro}
\label{cor:splendid}
Let $G$ be a finite group, $H$ a subgroup of $G$, $b$ a block idempotent of $\OC G$, $c$ a
block idempotent of $\Lambda H$. Let $C$ be a bounded complex of $\ell$-permutation
$(\Lambda Gb,\Lambda Hc)$-bimodules all of whose indecomposable summands have a vertex contained in
$\Delta H$.
Assume $\Hom_{D^b(kC_G(Q))}(\brauer_{\Delta Q}(C),\brauer_{\Delta Q}(C)[i])=0$ for all $i\not=0$
and all $\ell$-subgroups $Q$ of $H$ and the canonical map
$kHc\to \End_{D^b(kG)}(kC)$ is an isomorphism.

Then $C$ induces a splendid Rickard equivalence between $\Lambda Gb$ and $\Lambda Hc$.
\end{coro}

\begin{proof}
Theorem \ref{th:splitEnd} shows that the canonical map $kHc\to \End_{kG}^\bullet(kC)$ is an
isomorphism in $\Ho^b(k(H\times H^\opp))$. It follows from \cite[Theorem 2.1]{Ri2} that
 $kC$ induces a Rickard equivalence between $kGb$ and $kHc$.
The result follows now from \cite[proof of Theorem 5.2]{Ri2}.
\end{proof}

\end{document}